\documentclass[reqno,11pt,letterpaper]{amsart}

\usepackage{amsmath,amssymb,amsthm,mathrsfs,bbm}
\usepackage{arydshln}
\usepackage{upgreek}
\usepackage{slashed}
\usepackage{xifthen}
\usepackage{graphicx}
\usepackage{longtable}
\usepackage[inline]{enumitem}
\usepackage[all]{xy}

\usepackage{xcolor}
\definecolor{winered}{rgb}{0.7,0,0}
\definecolor{lessblue}{rgb}{0,0,0.7}

\usepackage[pdftex,colorlinks=true,linkcolor=winered,citecolor=lessblue,breaklinks=true,bookmarksopen=true]{hyperref}

\makeatletter
\newcommand{\myitem}[3]{\item[#2]\def\@currentlabel{#3}\label{#1}}
\makeatother

\usepackage{titletoc}

\makeatletter

\def\@tocline#1#2#3#4#5#6#7{
\begingroup
  \par
    \parindent\z@ \leftskip#3 \relax \advance\leftskip\@tempdima\relax
                  \rightskip\@pnumwidth plus 4em \parfillskip-\@pnumwidth
    \ifcase #1 
       \vskip 0.6em \hskip 0em 
       \or
       \or \hskip 0em 
       \or \hskip 1em 
    \fi%
    #6
    \nobreak\relax{\leavevmode\leaders\hbox{\,.}\hfill}
    \hbox to\@pnumwidth {\@tocpagenum{#7}}
  \par
\endgroup
}

 \def\l@section{\@tocline{0}{0pt}{0pc}{}{}}

\renewcommand{\tocsection}[3]{%
  \indentlabel{\@ifnotempty{#2}{ 
    \ignorespaces\bfseries{#2. #3}}}
  \indentlabel{\@ifempty{#2}{\ignorespaces\bfseries{#3}}{}} 
    \vspace{1.5pt}}

\renewcommand{\tocsubsection}[3]{%
  \indentlabel{\@ifnotempty{#2}{
    \ignorespaces#2. #3}}
  \indentlabel{\@ifempty{#2}{\ignorespaces #3}{}}
    \vspace{1.5pt}}

\renewcommand{\tocsubsubsection}[3]{%
  \indentlabel{\@ifnotempty{#2}{
    \ignorespaces#2. #3}}
  \indentlabel{\@ifempty{#2}{\ignorespaces #3}{}}
    \vspace{1.5pt}}

\makeatother

\makeatletter
\def\@nomenstarted{0}
\newlength{\@nomenoldtabcolsep}

\newcommand{\nomenstart}
  {%
    \def\@nomenstarted{1}%
    \setlength{\@nomenoldtabcolsep}{\tabcolsep}%
    \setlength{\tabcolsep}{3.5pt}%
    \begin{longtable}{p{0.11\textwidth} p{0.86\textwidth}}
  }

\newcommand{\nomenitem}[2]{%
    \ifcase\@nomenstarted%
      \or 
      \or \\ 
    \fi%
    #1\,{\leavevmode\leaders\hbox{\,.}\hfill} & #2%
    \def\@nomenstarted{2}%
  }%
\newcommand{\nomenend}
  {\\%
      \end{longtable}%
      \setlength{\tabcolsep}{\@nomenoldtabcolsep}%
      \def\@nomenstarted{0}%
  }
\makeatother

\makeatletter
\newcommand{\vast}{\bBigg@{4}}
\newcommand{\Vast}{\bBigg@{5}}
\makeatother

\allowdisplaybreaks

\numberwithin{equation}{section}
\numberwithin{figure}{section}
\newtheorem{thm}{Theorem}[section]

\newtheorem{prop}[thm]{Proposition}
\newtheorem{lemma}[thm]{Lemma}

\newtheorem*{thm*}{Theorem}
\newtheorem*{prop*}{Proposition}
\newtheorem*{cor*}{Corollary}
\newtheorem*{conj*}{Conjecture}

\theoremstyle{definition}
\newtheorem{definition}[thm]{Definition}

\theoremstyle{remark}
\newtheorem{rmk}[thm]{Remark}
\newtheorem{example}[thm]{Example}

\newcommand{\mc}{\mathcal}
\newcommand{\cA}{\mc A}
\newcommand{\cB}{\mc B}
\newcommand{\cC}{\mc C}
\newcommand{\cD}{\mc D}
\newcommand{\cE}{\mc E}
\newcommand{\cF}{\mc F}

\newcommand{\cL}{\mc L}
\newcommand{\cM}{\mc M}

\newcommand{\cO}{\mc O}

\newcommand{\cS}{\mc S}

\newcommand{\cU}{\mc U}
\newcommand{\cV}{\mc V}
\newcommand{\cW}{\mc W}
\newcommand{\cX}{\mc X}

\newcommand{\ms}{\mathscr}

\newcommand{\sD}{\ms D}

\newcommand{\N}{\mathbb{N}}
\newcommand{\R}{\mathbb{R}}

\newcommand{\Sph}{\mathbb{S}}

\newcommand{\sfp}{\mathsf{p}}
\newcommand{\sfr}{\mathsf{r}}

\newcommand{\sfH}{\mathsf{H}}

\newcommand{\fM}{\mathfrak{M}}
\newcommand{\fX}{\mathfrak{X}}

\newcommand{\codim}{\operatorname{codim}}

\renewcommand{\Im}{\operatorname{Im}}

\newcommand{\supp}{\operatorname{supp}}

\newcommand{\rank}{\operatorname{rank}}

\newcommand{\eps}{\epsilon}

\newcommand{\la}{\langle}

\newcommand{\ol}{\overline}
\newcommand{\pa}{\partial}
\newcommand{\ra}{\rangle}

\newcommand{\ul}[1]{\underline{#1}{}}

\newcommand{\wh}{\widehat}
\newcommand{\wt}{\widetilde}

\newcommand{\bop}{{\mathrm{b}}}

\newcommand{\cp}{{\mathrm{c}}}
\newcommand{\cuop}{{\mathrm{cu}}}

\newcommand{\Diff}{\mathrm{Diff}}

\DeclareMathOperator{\Op}{Op}

\newcommand{\Diffb}{\Diff_\bop}

\newcommand{\Psib}{\Psi_\bop}

\newcommand{\Vcu}{\cV_\cuop}
\newcommand{\Diffcu}{\Diff_\cuop}
\newcommand{\Psicu}{\Psi_\cuop}

\newcommand{\WF}{\mathrm{WF}}
\newcommand{\Ell}{\mathrm{Ell}}

\newcommand{\esssupp}{\mathrm{ess}\supp}

\newcommand{\WFcu}{\WF_{\cuop}}

\newcommand{\Ellcu}{\mathrm{Ell}_{\cuop}}

\newcommand{\Tb}{{}^{\bop}T}
\newcommand{\Tcu}{{}^{\cuop}T}
\newcommand{\Scu}{{}^{\cuop}S}

\newcommand{\half}{{\tfrac{1}{2}}}

\newcommand{\bhm}{{\mathfrak m}}
\newcommand{\bha}{{\mathbf a}}

\newcommand{\loc}{{\mathrm{loc}}}
\newcommand{\CI}{\cC^\infty}
\newcommand{\CIdot}{\dot\cC^\infty}

\newcommand{\CIc}{\cC^\infty_\cp}

\newcommand{\Hb}{H_{\bop}}

\newcommand{\Hcu}{H_{\cuop}}

\newcommand{\openbigpmatrix}[1]
  {%
    \def\@bigpmatrixsize{#1}%
    \addtolength{\arraycolsep}{-#1}%
    \begin{pmatrix}%
  }
\newcommand{\closebigpmatrix}
  {%
    \end{pmatrix}%
    \addtolength{\arraycolsep}{\@bigpmatrixsize}%
  }

\newcommand{\usref}[1]{{\upshape\ref{#1}}}

\DeclareGraphicsExtensions{.mps}
\newcommand{\inclfig}[1]{\includegraphics{#1}}

\makeatletter
\newcounter{@enumsave}
\newcommand{\ctrsave}{\setcounter{@enumsave}{\theenumi}}
\newcommand{\ctrset}{\setcounter{enumi}{\the@enumsave}}
\makeatother

\begin{document}

\title[Trapping on asymptotically stationary spacetimes]{Normally hyperbolic trapping on asymptotically stationary spacetimes}



\date{\today}

\author{Peter Hintz}
\address{Department of Mathematics, Massachusetts Institute of Technology, Cambridge, MA 02139-4307, USA}
\email{phintz@mit.edu}

\date{\today}

\begin{abstract}
  We prove microlocal estimates at the trapped set of asymptotically Kerr spacetimes: these are spacetimes whose metrics decay inverse polynomially in time to a stationary subextremal Kerr metric. This combines two independent results. The first one is purely dynamical: we show that the stable and unstable manifolds of a decaying perturbation of a time-translation-invariant dynamical system with normally hyperbolic trapping are smooth and decay to their stationary counterparts. The second, independent, result provides microlocal estimates for operators whose null-bicharacteristic flow has a normally hyperbolic invariant manifold, under suitable non-degeneracy conditions on the stable and unstable manifolds; this includes operators on closed manifolds, as well as operators on spacetimes for which the invariant manifold lies at future infinity.
\end{abstract}

\maketitle


\section{Introduction}
\label{SI}

This paper has two independent parts: in the first (\S\ref{SM}), we study perturbations of dynamical systems which exhibit normally hyperbolic trapping `at infinity', and in the second (\S\ref{SE}) we prove microlocal estimates for solutions of pseudodifferential equations whose null-bicharacteristic flow has this dynamical structure. The application tying the two together---the main motivation for the present paper---concerns the study of waves on perturbations of Kerr black holes (\S\ref{SK}).

We first describe the dynamical result. Let $\cX$ denote a closed manifold, and let $\bar V\in\cV(\cX)$ be a smooth vector field which is tangent to a submanifold $\Gamma$, which we call the \emph{trapped set}; suppose the $\bar V$-flow is $r$-normally hyperbolic at $\Gamma$ for every $r$, see~\S\ref{SsMF}. By classical theorems of Fenichel \cite{FenichelInvariant} and Hirsch--Pugh--Shub \cite{HirschPughShubInvariantManifolds}, there exist smooth stable/unstable manifolds $\bar\Gamma^{s/u}$ near $\Gamma$ to which $\bar V$ is tangent. Extending this to a dynamical system on the `spacetime'
\[
  \cM = \R_t \times \cX,\quad
  V_0 = \frac{\pa}{\pa t} + \bar V,
\]
$V_0$ is tangent to the spacetime trapped set $\Gamma_0=\R_t\times\bar\Gamma$, which has unstable/stable manifolds $\Gamma_0^{u/s}=\R_t\times\bar\Gamma^{u/s}$.

We consider perturbations of this dynamical system which decay as $t\to\infty$: denoting by $\CI_b(\cM)$ the space of functions which are bounded together with all their derivatives along $\pa_t$ and vector fields on $\cX$, we consider
\begin{equation}
\label{EqIPert}
  V = V_0+\wt V,\qquad \wt V \in \rho\CI_b(\cM;T\cX),\ \ \rho(t)=\la t\ra^{-\alpha},\ \alpha>0.
\end{equation}
Thus, $\wt V$ is a `spatial' vector field whose components (in local coordinate systems on $\cX$) decay to zero at the rate $\rho(t)$. Our first result concerns the existence and regularity of the perturbed stable and unstable manifolds:

\begin{thm}
\label{ThmIMfd}
  There exist a stable manifold $\Gamma^s\subset\cM$ and an unstable manifold $\Gamma^u\subset\cM$ to which $V$ is tangent, and so that $\Gamma^{s/u}$ is $\rho\CI_b$-close to $\Gamma^{s/u}_0$. More precisely, there exist open neighborhoods $\cU^{s/u}$ of $\Gamma$ inside of $\bar\Gamma^{s/u}$ such that for $T$ large, $\Gamma^{s/u}\cap t^{-1}((T,\infty))$ is the graph over $(T,\infty)\times\cU^{s/u}$ of a function in $\rho\CI_b((T,\infty)\times\cU^{s/u};N\cU^{s/u})$.
\end{thm}

See Theorem~\ref{ThmMFlow} for the full statement. Our proof of Theorem~\ref{ThmIMfd} is an application of Hadamard's idea, called `graph transform' in \cite{HirschPughShubInvariantManifolds}, for the construction of stable and unstable manifolds, namely the repeated application of the time ($-1$) flow of $V$ to $\Gamma^s_0$, which converges to $\Gamma^s$. (We can closely follow the outstanding presentation of \cite{HirschPughShubInvariantManifolds}.) For $\Gamma^u$, one instead starts with a piece of a $V$-invariant manifold over $t^{-1}((T,T+2))$ and repeatedly applies the time 1 flow.\footnote{Thus, $\Gamma^s$ is a `canonical' object, i.e.\ independent of choices, whereas $\Gamma^u$ is `non-canonical', as it does depend on the choice of some initial $V$-invariant piece.}  In fact, we will deduce Theorem~\ref{ThmIMfd} from an analogous statement about diffeomorphisms on $\cX$ which are normally hyperbolic at $\Gamma$, see Theorem~\ref{ThmMInv}.

Since $t\to\infty$ along integral curves of $V_0$, trapping only really occurs at `$t=\infty$'. One can make this precise by introducing a partial compactification of $\cM$ in which one adds a boundary $\{\tau_e=0\}\cong\cX$, $\tau_e=e^{-t}$. Indeed, $\pa_t=-\tau_e\pa_{\tau_e}$, hence $\bar\Gamma\subset\cX\cong\tau_e^{-1}(0)$ is a $V_0$-invariant set at infinity, with unstable manifold $\bar\Gamma^u\subset\cX\cong\tau_e^{-1}(0)$ and stable manifold (the closure of) $\Gamma^s_0$. From this perspective, the perturbations considered here have size $1/|\log(\tau_e)|^\alpha$, i.e.\ are very far from differentiable. Such singular perturbations can be analyzed because the simple nature of the flow of $V_0$ and $V$ in the $t$-variable; see also Remark~\ref{RmkMCptExp}. Figure~\ref{FigIMap} illustrates Theorem~\ref{ThmIMfd} from this point of view.

\begin{figure}[!ht]
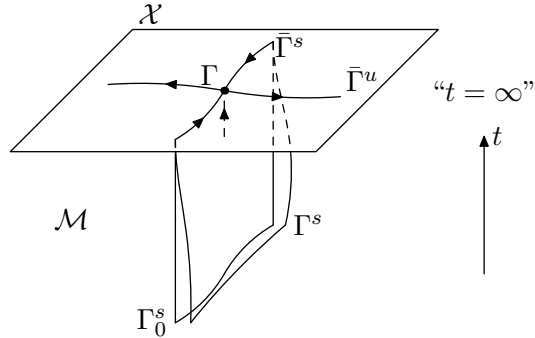

  \centering
  \inclfig{FigIMap}
  \caption{Illustration of Theorem~\ref{ThmIMfd} from a compactified perspective, with the (canonical) stable manifold $\Gamma^s$ shown.}
  \label{FigIMap}
\end{figure}

Our second, independent, result describes the propagation of microlocal Sobolev regularity for solutions of non-elliptic pseudodifferential operators $P\in\Psi^m$ whose null-bi\-char\-ac\-ter\-is\-tic flow, i.e.\ the flow of the Hamilton field $H_p$ of its principal symbol $p$ within the characteristic set $p^{-1}(0)$, has a normally hyperbolic invariant manifold $\Gamma\subset p^{-1}(0)$. We first describe this in a simple setting. Let $X$ be a closed manifold and suppose
\[
  P\in\Psi^m(X),\quad P-P^*\in\Psi^{m-2}(X),
\]
is a classical ps.d.o.\ with principal symbol $p=\sigma(P)$ and characteristic set $\Sigma=p^{-1}(0)\cap(T^*X\setminus o)$; suppose there exists a conic, normally hyperbolic submanifold $\Gamma\subset\Sigma$ for the $H_p$-flow, of codimension $2$ and with conic stable/unstable manifolds $\Gamma^{s/u}\subset\Sigma$ (defined in a small neighborhood of $\Gamma$) of codimension $1$. Assume that the Poisson bracket of defining functions of $\Gamma^{s/u}$ inside of $\Sigma$, extended to functions on $T^*X$, does not vanish at $\Gamma$, and assume that in a neighborhood of $\Gamma$ there exists an order function (non-vanishing and homogeneous of degree $1$ in the fibers of $T^*X\setminus o$) that commutes with $H_p$. Then:

\begin{thm}
\label{ThmIEst}
  Let $v\in\cD'(X)$, $P v=f$, and $s\in\R$. Suppose $\WF^{s+1}(v)\cap(\Gamma^s\setminus\Gamma)=\emptyset$ and $\WF^{s-m+2}(f)\cap\Gamma=\emptyset$. Then $\WF^s(v)\cap\Gamma=\emptyset$. The same conclusion remains valid when instead $\WF^{s+1}(v)\cap(\Gamma^u\setminus\Gamma)=\emptyset$.
\end{thm}

(See Example~\ref{ExESEx} for an explicit operator $P$ to which this theorem applies.) One can also allow $P$ to be a principally scalar operator acting on sections of a vector bundle, and one can furthermore allow for a non-trivial skew-adjoint part $\frac{1}{2 i}(P-P^*)\in\Psi^{m-1}$ as long as its principal symbol has a suitable positive upper bound; see Remark~\ref{RmkESSkew} for details. Thus, microlocal regularity can be propagated into the trapped set, where we can control $v$ with two derivatives less relative to elliptic estimates. Recall here that a point $(x_0,\xi_0)\in T^*\R^n$, $\xi_0\neq 0$, does \emph{not} lie in $\WF^s(u)$ for a distribution $u\in\sD'(\R^n)$ iff there exist cutoff functions $\phi\in\CIc(\R^n)$, $\psi\in\CI(\Sph^{n-1})$ with $\phi(x_0)\neq 0$, $\psi(\xi_0/|\xi_0|)\neq 0$, such that $\la\xi\ra^s \psi(\xi/|\xi|)\wh{\phi u}(\xi)\in L^2(\R^n_\xi)$, where $\wh f(\xi)=\int e^{-i x\xi}f(x)d x$ denotes the Fourier transform.

We really prove quantitative estimates for Sobolev norms of the microlocalization of $v$ to a small neighborhood of $\Gamma$: if $B_0,B_1,G\in\Psi^0(X)$ are such that the elliptic set $\Ell(B_0)$ of $B_0$ contains $\Gamma$, furthermore $\WF'(B_0)\subset\Ell(G)$, and all backwards null-bicharacteristics from $\WF'(B_0)$ either tend to $\Gamma$ or enter $\Ell(B_1)$ in finite time, all while remaining in $\Ell(G)$, then
\begin{equation}
\label{EqIEst}
  \|B_0 v\|_{H^s} \leq C\bigl( \|G P v\|_{H^{s-m+2}} + \|B_1 v\|_{H^{s+1}} + \|v\|_{H^{-N}} \bigr).
\end{equation}
See Figure~\ref{FigIEst}.

\begin{figure}[!ht]
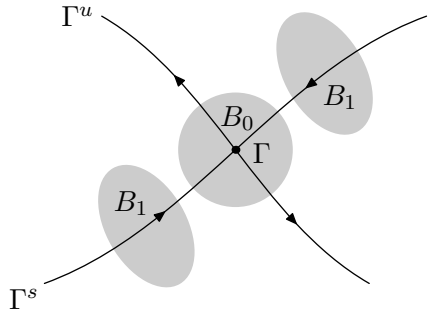

\centering
\inclfig{FigIEst}
\caption{Illustration of Theorem~\ref{ThmIEst}: shown are the elliptic sets of $B_1$ (where we require a priori $H^{s+1}$ control of $u$) and $B_0$ (where we conclude $H^s$ regularity of $u$).}
\label{FigIEst}
\end{figure} 

For $v\in\CI(X)$, such an estimate can be proved by a simple adaptation of the elegant semiclassical argument of Dyatlov~\cite{DyatlovSpectralGaps}, which proceeds by assuming that a constant $C$ for which this estimate holds does not exist, and reaching a contradiction using defect measures \cite{GerardDefect,LionsPaulDefect,TartarHMeasures} and their concentration/Lipschitz properties along $\Gamma^u$. In order to \emph{conclude} regularity of $v$ at $\Gamma$ as in Theorem~\ref{ThmIEst} however, we proceed directly using positive commutator estimates (together with a standard regularization argument); the argument is described in~\S\ref{SsEC} and sketched after Theorem~\ref{ThmIKerr} below. (We remark that the proof takes place entirely in the \emph{standard} pseudodifferential calculus on $X$, i.e.\ it is only utilizes quantizations of symbols in the class $S^m_{1,0}$.) An added benefit, important for applications to nonlinear problems, is the (in principle) quantitative control of the constant $C$ on Sobolev norms of the coefficients of $P$. See Remarks~\ref{RmkESDefect} and \ref{RmkECPTame}.

For us, the main interest lies in analogues of Theorem~\ref{ThmIEst} when the trapped set is `at infinity' as explained after Theorem~\ref{ThmIMfd}. We refer the reader to~\S\ref{SsEC} for the general microlocal result; here, we merely describe the special case of asymptotically Kerr spacetimes. Recall that a (stationary) Kerr spacetime is a manifold
\[
  M^\circ = \R_t \times X,\quad X=(r_+,\infty)_r \times \Sph^2 \subset \R^3_x,
\]
equipped with a certain Lorentzian metric $g_0$ of signature $(+,-,-,-)$ which is stationary: $\cL_{\pa_t}g_0=0$. (The metric, given in~\eqref{EqKMetric}, depends on two real parameters $\bhm,\bha$, the black hole mass and angular momentum; we only consider the \emph{subextremal} case $0\leq|\bha|<\bhm$.) Here, $r_+>0$ is the radius of the event horizon of the black hole. Let $G_0(\zeta)=|\zeta|_{g_0^{-1}}^2$, $\zeta\in T^*M^\circ$, and denote the future component of the characteristic set, i.e.\ the collection of future light cones, by
\[
  \Sigma_0=\{\zeta\in T^*M^\circ\setminus o\colon G_0(\zeta)=0,\ g_0^{-1}(\zeta,d t)>0\}.
\]
The lifted (future) null-geodesic flow of $g_0$ is the flow of the Hamilton vector field $H_{G_0}$ on $\Sigma_0$. Since all notions of interest here are conic in the fibers of $T^*M^\circ$, let us pass to the cosphere bundle $S^*M^\circ=(T^*M^\circ\setminus o)/\R_+$, where $\R_+$ acts by dilations on the fibers; the rescaled vector field $\sfH_{G_0}|_\zeta:=(g_0^{-1}(\zeta,d t))^{-1}H_{G_0}|_\zeta$, which is homogeneous of degree $0$ in the fibers, descends to a vector field on $\Sigma_0$, identified with a subset of $S^*M^\circ$. The feature of interest of the $\sfH_{G_0}$-flow here is the existence of a smooth, conic, flow-invariant trapped set $\Gamma_0=\R_t\times\Gamma\subset\Sigma_0$: null-geodesics in $\Gamma_0$ never escape to $r=r_+$ or $r=\infty$ and instead, when projected to $X$, remain in a compact subset of $X$. The $\sfH_{G_0}$-flow in $\Sigma_0$ is $r$-normally hyperbolic for every $r$ at $\Gamma_0$, as described in~\S\S\ref{SsMF} and \ref{SK}. This was first observed by Wunsch and Zworski~\cite{WunschZworskiNormHypResolvent} for slowly rotating Kerr black holes, and proved in the full subextremal range by Dyatlov~\cite{DyatlovWaveAsymptotics}; see also \cite{VasyMicroKerrdS} for the Kerr--de~Sitter case. The unstable/stable manifolds
\[
  \Gamma^{u/s}_0 = \R_t\times\bar\Gamma^{u/s} \subset \Sigma_0,
\]
consisting of those covectors $\zeta\in\Sigma_0$ for which the backward/forward integral curve with initial condition $\zeta$ tends to $\Gamma_0$, are smooth conic codimension $1$ submanifolds.

We consider metric perturbations $g=g_0+\wt g$ of $g_0$ which are asymptotically (as $t\to\infty$) stationary: analogously to~\eqref{EqIPert}, we assume
\[
  \wt g \in \rho\CI_b,
\]
i.e.\ all components of $\wt g$ in $(t,x)$-coordinates are $\rho(t)=\la t\ra^{-\alpha}$ times functions in $\CI_b(M^\circ)$. The rescaled Hamilton vector field $\sfH_G$ of $G(\zeta)=|\zeta|_{g^{-1}}^2$ on the perturbed future part $\Sigma$ of the characteristic set $G^{-1}(0)\subset S^*M^\circ$ fits (after a coordinate change in $S^*M^\circ$ and an additional rescaling) into the framework of Theorem~\ref{ThmIMfd}, providing us with perturbed stable/unstable manifolds $\Gamma^{s/u}$ which are $\rho\CI_b$-perturbations of $\Gamma_0^{s/u}$.

We then study the propagation of weighted (in $t$) uniform (as $t\to\infty$) microlocal Sobolev regularity of solutions of wave equations on $(M,g)$. Working in $t\geq 1$, we say that $(x_0,\zeta_0)$, with $\zeta_0=\sigma_0\,d t+\xi_0\in\R\,d t\oplus T_{x_0}^*X$ non-zero, does not lie in $\WFcu^{s,r}(v)$ for a distribution $v\in t^{-r}H^{-N}(\R^4_{t,x})$ if there exists a cutoff $\chi=\chi(t)$, identically $1$ for sufficiently large $t$, and cutoffs $\phi\in\CIc(X)$ (non-zero at $x_0$) and $\psi\in\CI(\Sph^3)$ (non-zero at $\zeta_0/|\zeta_0|$, defined using the Euclidean norm on $\R^4\cong \R\,d t\oplus T_{x_0}^*X$) such that
\[
  \la\zeta\ra^s \psi(\zeta/|\zeta|)\cF(t^r\chi\phi v)(\zeta) \in L^2(\R^4_\zeta),\quad \zeta=(\sigma,\xi),
\]
where $\cF(f)(\sigma,\xi)=\int e^{-i(\sigma t+\xi x)}f(t,x)d t\,d x$ is the spacetime Fourier transform. Thus, $\WFcu^{s,r}(v)$ captures those positions and spacetime frequencies where weighted (by $\la\zeta\ra^s$) amplitudes of high frequency components (in conic directions in the momentum variable) of $t^r v$ fail to be square integrable in spacetime.

\begin{thm}
\label{ThmIKerr}
  Let $v\in t^{-r}H^{-N}(\R^4)$, $\Box_g v=f$. Suppose that $\WF^{s+1}(v)\cap\Gamma^s=\emptyset$ and $\WFcu^{s+1,r}(v)\cap(\bar\Gamma^s\setminus\Gamma)=\emptyset$; suppose further that $\WFcu^{s,r}(f)\cap\Gamma=\emptyset$. Then $\WFcu^{s,r}(v)\cap\Gamma=\emptyset$.
\end{thm}

The proof of this theorem depends crucially on the aforementioned breakthrough work \cite{WunschZworskiNormHypResolvent}, as it strongly uses the \emph{dynamical (normally hyperbolic) nature of the trapping}; in contrast, the special algebraic structure of the trapped set of Kerr, namely the complete integrability of the null-geodesic flow which allows for separation of variables, is \emph{irrelevant} (except insofar it is useful for actually proving the normal hyperbolicity), and in fact by itself seems to be \emph{insufficient} for proving this theorem, as exact control of the perturbed (and in general certainly not completely integrable) dynamics of $g$ is strongly used in the proof.

Theorem~\ref{ThmIKerr} is closely related to the estimates in \cite{WunschZworskiNormHypResolvent,DyatlovResonanceProjectors,DyatlovSpectralGaps} for semiclassical operators; we discuss this further below. There is an analogous statement for propagation from the unstable manifold into $\Gamma$, as well as extensions to principally scalar, non-self-adjoint (with suitable upper bound on the subprincipal symbol) operators between vector bundles. Thus, uniform microlocal regularity propagates from a punctured neighborhood of $\Gamma$ within $\Gamma^s$ into $\Gamma$, with derivative losses as in Theorem~\ref{ThmIEst}. This theorem can be phrased more naturally on a compactification of $M$ to a manifold with boundary
\begin{equation}
\label{EqICpt}
  [0,\infty)_\tau \times \cX,\quad \tau := t^{-1},
\end{equation}
in which case $\WFcu^{s,r}(v)$ is the (complete) cusp wave front set of $v$; we explain these notions in~\S\ref{SssECC}, following \cite{MazzeoMelroseFibred,VasyMinicourse}. The trapped set $\Gamma$ then lies over $\tau=0$, while the stable manifold is the closure $\Gamma^s\sqcup\bar\Gamma^s$ of $\Gamma^s$, where one views $\bar\Gamma^s\subset\{\tau=0\}$. As in the closed manifold setting, we prove Theorem~\ref{ThmIKerr} by means of a positive commutator argument (now employing the cusp pseudodifferential algebra) inspired by \cite{DyatlovSpectralGaps}; this provides quantitative microlocal bounds in weighted Sobolev spaces analogous to~\eqref{EqIEst}.

A key ingredient of the proof is that the microlocalization $\Phi^u v$ of $v$ (i.e.\ microlocalizing in a weak manner away from $\Gamma^u$), where $\Phi^u$ quantizes a defining function of $\Gamma^u$, satisfies a pseudodifferential equation which effectively has a damping term at $\Gamma$; for this step it is \emph{crucial} that $\Gamma^u$ be \emph{exactly} invariant by the $\sfH_G$-flow (rather than merely asymptotically so, as is the case for $\Gamma^u_0$); this is discussed after equation~\eqref{EqECPPhiuvForc}. Using a \emph{quantitative} version of real principal type propagation based on a careful construction of commutants and G\aa{}rding's inequality, this equation implies that the squared $t^{-r}H^s$ mass of $v$ is evenly spread out along $\bar\Gamma^u$. On the other hand, quantitative propagation for $\Box_g v=f$ and the unstable nature of the $\sfH_G$-flow on $\bar\Gamma^u$ imply that the squared mass on $\bar\Gamma^u$, but with distance from $\Gamma$ between $\delta$ and $1$, can be bounded by $\log(\delta^{-1})$ times the squared mass $\delta$-close to $\Gamma$, which is $\sim\delta$ (plus contributions from $u$ away from $\bar\Gamma^u$, and from $f$). Since $\delta\log(\delta^{-1})\ll 1$ for small $\delta>0$, this provides the desired bound of $v$ near $\Gamma$.

Previously, spacetime bounds at normally hyperbolic trapping were obtained by the author in joint work with Vasy~\cite{HintzVasyNormHyp}; these took place on \emph{exponentially} weighted (growing) function spaces, or on unweighted but mildly degenerate (at $\Gamma$) function spaces when studying symmetric operators. Here, we obtain estimates for non-symmetric operators on polynomially weighted (possibly \emph{decaying}), or even mildly exponentially decaying function spaces, though with additional loss of regularity; see Remarks~\ref{RmkECPNhyptr} and \ref{RmkKdS}. The more delicate estimates we prove here were not needed in the analysis of quasilinear waves on (and the nonlinear stability of) Kerr--\emph{de~Sitter} spacetimes \cite{HintzVasyQuasilinearKdS,HintzVasyKdSStability}, since the exponential decay of metric perturbations there implied that estimates on exponentially growing function spaces together with an exact analysis of the stationary (exact Kerr--de~Sitter) model were sufficient to prove exponential decay to a finite sum of resonant states (mode solutions); in particular, one could in principle have used separation of variable techniques at the trapped set in those works, akin to~\cite{DyatlovQNM,DyatlovAsymptoticDistribution}, though this would have distracted from the conceptual, namely \emph{dynamical}, reason for having (high energy) estimates at the trapped set.

Most prior results on microlocal estimates at normally hyperbolic trapping take place in the semiclassical setting, which is closely related, via the Fourier transform in time, to estimates for stationary (time-translation-invariant) problems. This is the context of \cite{WunschZworskiNormHypResolvent,DyatlovSpectralGaps,DyatlovZworskiTrapping} as well as the fine analysis of resonances associated with normally hyperbolically trapped sets by Dyatlov \cite{DyatlovResonanceProjectors}; see G\'erard--Sj\"ostrand \cite{GerardSjostrandHyperbolic} and Christianson \cite{ChristiansonNonconc} for the case of isolated hyperbolic orbits. Nonnenmacher--Zworski \cite{NonnenmacherZworskiDecay} study estimates at trapped sets when the stable and unstable normal bundles have low regularity. We also mention the work by Bony--Burq--Ramond \cite{BonyBurqRamondTrapping} who prove that a loss of semiclassical control (powers of $h^{-1}$), which heuristically corresponds to a loss of control in the Sobolev regularity sense in~\eqref{EqIEst},\footnote{We encourage the reader to compare \cite[Theorem~4.7]{HintzVasyQuasilinearKdS} with the estimate~\eqref{EqIEst}.} \emph{does} occur in the presence of trapping; see also Ralston \cite{RalstonLocalized}. We refer the reader to the excellent review articles \cite{ZworskiResonanceReview,WunschMild} as well as~\cite[\S6.3]{DyatlovZworskiBook} for further references in these directions.

The study of waves on Schwarzschild and Kerr spacetimes has a long history, starting with \cite{WaldSchwarzschild,KayWaldSchwarzschild}; Tataru~\cite{TataruDecayAsympFlat} and Dafermos--Rodnianski--Shlapentokh-Rothman \cite{DafermosRodnianskiShlapentokhRothmanDecay} proved sharp polynomial decay for scalar waves on all subextremal Kerr spacetimes, see also \cite{BlueSterbenzSemilinear,FinsterKamranSmollerYauKerr,TataruTohaneanuKerrLocalEnergy,DonningerSchlagSofferPrice,DonningerSchlagSofferSchwarzschild,TohaneanuKerrStrichartz,LukKerrNonlinear,AnderssonBlueHiddenKerr,AnderssonBlueMaxwellKerr} and the recent \cite{DafermosHolzegelRodnianskiSchwarzschildStability,DafermosHolzegelRodnianskiTeukolsky} for the linear stability of the Schwarzschild solution and related problems for slowly rotating Kerr. Aretakis \cite{AretakisExtremalKerr} obtained results on extremal Kerr spacetimes. Resonance expansions for scalar waves on Kerr--de~Sitter spacetimes were proved by Dyatlov \cite{DyatlovQNM,DyatlovQNMExtended,DyatlovAsymptoticDistribution} and Vasy~\cite{VasyMicroKerrdS}, following Bony--H\"afner \cite{BonyHaefnerDecay} and building on the work of S\'a Barreto--Zworski \cite{SaBarretoZworskiResonances}; see also \cite{MelroseSaBarretoVasySdS}. Many of these results rely on delicate separation of variables techniques at the trapped set; the work \cite{WunschZworskiNormHypResolvent} was the first to utilize the (stable under perturbations!) dynamical nature of the flow directly. For tensor-valued waves in the presence of trapping, see \cite{HintzPsdoInner,HintzVasyKdsFormResonances,HintzVasyKdSStability,HintzKNdSStability}.

The plan of the paper is as follows.
\begin{itemize}
\item In \S\ref{SM}, we prove the dynamical results on perturbations of dynamical systems with normally hyperbolic invariant sets, in particular proving Theorem~\ref{ThmIMfd}.
\item In \S\ref{SE}, we establish the microlocal propagation results, Theorem~\ref{ThmIEst} and the pseudodifferential generalization of Theorem~\ref{ThmIKerr}.
\item In \S\ref{SK}, we combine the first two parts, thus obtaining a description of trapping on asymptotically Kerr spacetimes, and the accompanying microlocal estimates of Theorem~\ref{ThmIKerr}.
\end{itemize}
We stress that \S\ref{SM} and \S\ref{SE} are independent and can be read in any order.

\subsection*{Acknowledgments}

This project grew out of an ongoing collaboration with Andr\'as Vasy, who sketched the proof of the estimate~\eqref{EqIEst}; I am very grateful to him for many useful discussions. I would also like to thank Semyon Dyatlov, Long Jin, and Maciej Zworski for helpful conversations and suggestions. This research was conducted during the time I served as a Clay Research Fellow, and I would like to thank the Clay Mathematics Institute for its support.

\section{Stable and unstable manifolds}
\label{SM}

In this first, dynamical, part of the paper, we closely follow the arguments and notation of \cite{HirschPughShubInvariantManifolds}.

\subsection{Function spaces and Lipschitz jets}

Denote by
\[
  0<\rho \in \CI(\R)
\]
a weight which is monotonically decreasing with $\lim_{t\to\infty}\rho(t)=0$, and so that
\begin{equation}
\label{EqMWeightCI}
  |\rho^{(k)}|\leq C_k\rho\quad \forall\,k\in\N_0.
\end{equation}
For example, one can take $\rho(t)=\la t\ra^{-\alpha}$ or $e^{-\alpha t}$ with $\alpha>0$ fixed. (These are natural choices for asymptotically Kerr, resp.\ Kerr--de~Sitter spaces.) Let $\cX$ denote a closed manifold\footnote{All our arguments below will be local near the closed positive codimensional submanifold $\Gamma$; the structure of $\cX$ away from $\Gamma$ will be irrelevant.} and put
\[
  \cM=\R_t\times\cX.
\]
Denoting by $\cC_b^0(\cM)$ the space of bounded continuous functions, we define
\begin{equation}
\label{EqMCBdd}
\begin{split}
  \cC_b^r(\cM) := \{ u\in \cC_b^0(\cM) \colon& V_1\ldots V_k u\in \cC_b^0(\cM)\ \ \forall\,1\leq k\leq r, \\
    &\qquad\qquad V_i\in\cV(\cX)\cup\{\pa_t\},\,i=1,\ldots,k \},
\end{split}
\end{equation}
and $\CI_b(\cM):=\bigcap_{r\in\N}\cC_b^r(\cM)$. The weighted analogues are
\[
  \rho\cC_b^r(\cM) := \{ u\colon\cM\to\R \colon u/\rho \in \cC_b^r(\cM) \},\quad r\in\N_0\cup\{\infty\}.
\]
Note that by~\eqref{EqMWeightCI}, differentiation along $\pa_t$ or $V\in\cV(\cX)$ gives continuous maps $\rho\cC_b^r\to\rho\cC_b^{r-1}$ for all $r\geq 1$. For open sets $\cU\subset\cM$, we denote the space of restrictions to $\cU$ by
\[
  \rho\cC_b^r(\cU)=\{u|_\cU\colon u\in\rho\cC_b^r(\cM)\}.
\]
If $E\to\cX$ is a vector bundle and
\begin{equation}
\label{EqMProjX}
  \pi_\cX\colon\cM\to\cX
\end{equation}
denotes the projection, we similarly define $\cC^\infty(\cM;\pi_\cX^*E)$ and $\rho\cC^\infty(\cM;\pi_\cX^*E)$ using local trivializations of $E$ or, equivalently, using a connection to differentiate sections of $E$ along vector fields on $\cX$.

We recall from \cite[\S3]{HirschPughShubInvariantManifolds} the useful notion of \emph{Lipschitz jets}:
\begin{definition}
\label{DefMLipJet}
  Let $(X,d_X)$ and $(Y,d_Y)$ denote two metric spaces, and fix a point $x\in X$.
  \begin{enumerate}
  \item We say that the continuous maps $g_1,g_2\colon X\to Y$, defined near $x$, are \emph{tangent at $x$} iff $g_1(x)=g_2(x)$ and
    \[
      d_x(g_1,g_2) := \limsup_{x'\to x} \frac{d_Y(g_1(x'),g_2(x'))}{d_X(x,x')} = 0.
    \]
  \item The \emph{Lipschitz jet} of a continuous function $g\colon X\to Y$ at $x$, denoted $J_x g$, is the equivalence class of $g$ modulo tangency at $x$, i.e.\ the set of all continuous maps defined near $x$ which are tangent to $g$ at $x$.
  \item The set of Lipschitz jets of maps carrying $x$ into $y$ is denoted
    \[
      J(X,x;Y,y) = \{ J_x g \colon g\ \text{is a continuous map}\ X\to Y,\ g(x)=y \}.
    \]
  \item For $j_1,j_2\in J(X,x;Y,y)$, we define $d(j_1,j_2):=d_x(g_1,g_2)$, where $g_i$ is a representative of $j_i$; this is independent of the choices of representatives.
  \item Denoting by $y$ the constant map $X\to Y$, $x'\mapsto y$, we define by
    \[
      L_x g := d(J_x g,J_x y) \in [0,\infty]
    \]
    the \emph{Lipschitz constant of $g$ at $x$}.
  \item We define the space of bounded Lipschitz jets by
    \[
      J^b(X,x;Y,y) := \{ j\in J(X,x;Y,y) \colon d(j,J_x y)<\infty \}.
    \]
    If $X,Y$ are differentiable manifolds, we set
    \[
      J^d(X,x;Y,y) := \{ j\in J^b(X,x;Y,y) \colon j\ \text{has a differentiable representative} \}.
    \]
  \end{enumerate}
\end{definition}

If $X,Y$ are Banach spaces, we recall from \cite[Theorem~(3.3)]{HirschPughShubInvariantManifolds} that $J^b(X,0;Y,0)$, equipped with the norm $|j|:=d(j,0)$, is a Banach space, and $J^d(X,0;Y,0)$ is a closed subspace.

We record that if $a\in\rho\cC^1_b(\R\times\R^n)$, then there exists $C>0$ such that
\begin{equation}
\label{EqMLipEst}
\begin{gathered}
  |a(t,x)-a(t',x')| \leq C\bigl(\rho_+(t,t')|t-t'|+\rho_-(t,t')|x-x'|\bigr), \\
   \rho_+(t,t'):=\max\{\rho(t),\rho(t')\},\quad
   \rho_-(t,t'):=\min\{\rho(t),\rho(t')\};
\end{gathered}
\end{equation}
in fact, we can take $C=\sup_{(t,x)}(L_{(t,x)}(a)/\rho(t))$.

\subsection{An invariant section theorem}
\label{SsInv}

As an illustration of the relevant techniques, and as a technical tool for later, we prove existence and higher regularity of invariant sections for fiber contractions; this is a version of \cite[Theorem~3.5]{HirschPughShubInvariantManifolds}.

\begin{thm}
\label{ThmMInv}
  Let $r\in\N$. Let $\cX$ be a closed Riemannian manifold, and let $\cX_1$ be an open subset. Let $\bar f\colon\ol{\cX_1}\to\cX$ be a smooth map which is a diffeomorphism onto its image, and suppose $\bar f$ overflows $\cX_1$: $\bar f(\cX_1)\supset\ol{\cX_1}$. Let moreover $\bar\cE\to\cX$ be a vector bundle equipped with a fiber metric, and let $\bar\ell\colon\bar\cE|_{\cX_1}\to\bar\cE$ be a fiber bundle map\footnote{That is, we do not require it to be fiber-linear.} covering $\bar f$. Let
  \[
    \alpha_x := \| D_{\bar f(x)}\bar f^{-1} \|, \quad
    k_x := \sup_{e\in\bar\cE_x}\|D_e(\ell|_{\bar\cE_x})\|,\qquad x\in\cX_1,
  \]
  denote the base contraction and fiber expansion, respectively, and suppose that
  \begin{equation}
  \label{EqMInvContr}
    \sup_{x\in\cX_1} k_x,\ \sup_{x\in\cX_1} k_x\alpha_x^r<1.
  \end{equation}
  Let $\bar\sigma\colon\cX_1\to\bar\cE|_{\cX_1}$ denote the unique $\bar\ell$-invariant section, i.e.\ $\bar\ell(\bar\sigma(\cX_1))\cap\bar\cE|_{\cX_1}=\bar\sigma(\cX_1)$, which is of class $\cC^r$.\footnote{This is the content of the first part of \cite[Theorem~3.5]{HirschPughShubInvariantManifolds}.}

  Let $\cM=\R_t\times\cX$ and $\cM_1=\R_t\times\cX_1$, and put $f_0(t,x)=(t-1,\bar f(x))$ for $(t,x)\in\cM_1$. With $\cE_0:=\pi_\cX^*\bar\cE$ denoting the pullback bundle, $\pi_\cX\colon\cM\to\cX$ being the projection, extend $\bar\ell$ to the map $\ell_0\colon(t,x,e)\mapsto(t-1,\bar\ell(x,e))$ which covers $f_0$. The section $\sigma_0(t,x)=(t,\bar\sigma(x))$ is a stationary and invariant section for $\ell_0$.

  Let next $f\colon\cM_1\to\cM$ and $\ell\colon\cE_0|_{\cM_1}\to\cE_0$ be $\rho\cC_b^r$-perturbations of $f_0$ and $\ell_0$, defined for $t>t_0$ with $t_0\in\R$ fixed, with $\ell$ covering $f$. That is, fix a finite cover of $\ol\cX_1$ by coordinate systems, and fix trivializations of $\bar\cE$ over these; write $\ell(t,x,e)=(f(t,x),\breve\ell(t,x,e))$ in local coordinates, with $\breve\ell(t,x,e)$ valued in the fibers of $\cE_0$, likewise for $\ell_0$; then
  \begin{alignat*}{3}
    f(t,x) - f_0(t,x) &= \wt f(t,x),&\quad& \wt f(t,x) \in \rho\cC_b^r, \\
    \breve\ell(t,x,e)-\breve\ell_0(t,x,e) &= \wt\ell(t,x,e), &\quad& \wt\ell(t,x,e) \in \rho\cC_b^r.
  \end{alignat*}
  Furthermore, assume $\pi_T\wt f(t,x)=0$ where $\pi_T\colon\R\times\cX\to\R$ projects onto the first factor.

  Then there exists an $\ell$-invariant section $\sigma_\ell\colon\cM_1\to\cE_0$, defined for $t>t_0$, that is, $\ell(\sigma_\ell(\cM_1\cap\{t>t_0+1\}))\cap\cE_0|_{\cM_1}=\sigma_\ell(\cM_1\cap\{t>t_0\})$, such that
  \[
    \sigma_\ell-\sigma_0\in\rho\cC_b^r(\cM_1;\cE_0).
  \]
  This section is unique among sections $\sigma_\ell$ with $\sigma_\ell-\sigma_0\in\rho\cC_b^0$.
\end{thm}

For a linear transformation $A$ between two normed vector spaces, we put
\[
  m(A) := \inf_{\|x\|=1} \|A x\|.
\]

\begin{proof}[Proof of Theorem~\usref{ThmMInv}]
  We shall follow \cite[\S3]{HirschPughShubInvariantManifolds} closely; the strategy is to show that the perturbation $\wt f$ does not destroy the contraction properties of the stationary model $f_0$. By replacing $\bar\cE$ with a vector bundle $\bar\cE\oplus\bar\cE'$ for suitable $\bar\cE'\to\cX$, we may assume that $\bar\cE$ is trivial with typical fiber denoted $E$; we extend $\bar\ell$ by mapping $(x,e\oplus e')\mapsto\bar\ell(x,e)\oplus 0$, similarly for $\ell_0$, $\ell$; note that an invariant section is necessarily valued in $\cE\oplus 0$. It suffices to construct $\sigma$ as a section over
  \[
    \cM_1(\eps):=\{(t,x)\in\cM_1\colon t>\eps^{-1}\}
  \]
  for any small $\eps>0$, as $\sigma$ in $t>t_0$ can be reconstructed from this by repeated application of $\ell$. At first, we will seek the invariant section in the space
  \[
    \Sigma^0(\eps) := \{ \text{sections}\ \sigma\colon\cM_1(\eps)\to\cE_0 \colon |\sigma(t,x)-\sigma_0(t,x)|\leq C_\Sigma\rho(t) \},
  \]
  where $C_\Sigma$ will be specified later, see~\eqref{EqMInvSectBd}. We then wish to consider the map $\ell_\sharp$ on $\Sigma^0(\eps)$, defined by
  \[
    \ell_\sharp\sigma(t,x) = \ell\sigma f^{-1}(t,x),\quad (t,x)\in\cM_1(\eps);
  \]
  this is a `graph transform' of $\sigma$: the graph of $\ell_\sharp\sigma$ is the image of the graph of $\sigma$ under $\ell$. We first need to check that $f^{-1}$ is well-defined and maps $\cM_1(\eps)$ into itself.

  \textit{\underline{Step 1:} control of $f^{-1}$.} First, we show (using $\wt f\in\rho\cC_b^1$) that $f|_{\cM_1(\eps)}$ is injective for small $\eps>0$. Indeed, $f(t,x)=f(t',x')$ requires $t'=t$; but then, in geodesic coordinates centered at $x\in\cX_1$ and $\bar f(x)$, working in a geodesic ball around $x$ with small radius $\eps_0$ (which we will take small, independently of $\eps$), and letting $\mu_x:=m(D_x\bar f)=\alpha_{\bar f^{-1}(x)}^{-1}>0$,
  \begin{align*}
    |f(t,x)-f(t,x')| &\geq |f_0(t,x)-f_0(t,x')| + |\wt f(t,x)-\wt f(t,x')| \\
      &\geq \bigl(\mu_x-o(1) - \wt C\rho(t)\bigr)|x-x'|, \qquad \eps_0\to 0,
  \end{align*}
  where $\wt C$ is the Lipschitz constant of $\rho^{-1}\wt f$. This implies injectivity of $f$ restricted to small geodesic balls in $\cX$ provided $t$ is sufficiently large; the injectivity of $\bar f$ on the compact set $\ol{\cX_1}$ then implies the injectivity of $f$ on $\cM_1(\eps)$ for $\eps>0$ small.
  
  Given $t$ large and $y\in\ol{\cX_1}$, we can then solve $f(t+1,x)=(t,y)$ with $x\in\cX_1$. Indeed, working in local coordinates, the map
  \begin{equation}
  \label{EqMInvFInv}
    x\mapsto \bar f^{-1}\bigl(y-\pi_\cX\wt f(t+1,x)\bigr),
  \end{equation}
  for $x$ in an $\eps_0$-neighborhood of $\bar f^{-1}(y)$, is well-defined for large $t$ (thus small $\wt f$) by the overflow assumption on $\bar f$, and is a contraction, again since $\wt f$ and its Lipschitz constant decay as $t\to\infty$. From this construction, we infer that the inverse map
  \[
    g\colon\cM_1(\eps)\to\cM_1(\eps),\quad
    g(t,y)=f^{-1}(t,y)=(t+1,x),
  \]
  satisfies Lipschitz bounds
  \begin{equation}
  \label{EqMInvLip}
  \begin{split}
    |\pi_\cX g(t,y)-\pi_\cX g(t,y')| &\leq (\alpha_{\bar f^{-1}(y)}+o(1)+\wt C\rho(t))|y-y'| \\
      &\leq (\alpha_{\bar f^{-1}(y)}+o(1))|y-y'|, \qquad \eps\to 0,
  \end{split}
  \end{equation}
  when $|y-y'|<\eps$ and $t>\eps^{-1}$. Writing $g_0(t,y)=f_0^{-1}(t,y)=(t+1,\bar f^{-1}(y))$, we also note that
  \begin{equation}
  \label{EqMInvBddInv}
    |g(t,y)-g_0(t,y)|\leq\wt C\rho(t)
  \end{equation}
  using the description of $\pi_\cX g(t,y)$ as the fixed point of~\eqref{EqMInvFInv}.

  \textit{\underline{Step 2:} existence and uniqueness in $\rho\cC_b^0$.} This uses the $r=1$ assumptions. Denote by $\pi_E\colon\cE_0\to E$ denote the projection onto the fiber, and write sections of $\cE_0$ as
  \[
    \sigma(t,x)=(t,x,\breve\sigma(t,x)),\quad \breve\sigma=\pi_E\sigma.
  \]
  Equip $\Sigma^0(\eps)$ with the (complete) $\cC^0$ metric. For $x=\pi_\cX g(t,y)$ and $x_0=\pi_\cX g_0(t,y)$, and using~\eqref{EqMInvBddInv}, we thus have
  \begin{align*}
    &|\ell_\sharp\sigma(t,y)-\sigma_0(t,y)| \\
    &\quad= |\ell_\sharp\sigma(t,y)-((\ell_0)_\sharp\sigma_0)(t,y)| \\
    &\quad =|\ell(g(t,y),\breve\sigma(g(t,y)))-\ell_0(g_0(t,y),\breve\sigma_0(g_0(t,y)))| \\
    &\quad\leq |(\breve\ell-\breve\ell_0)(g(t,y),\breve\sigma(g(t,y)))| + |\breve\ell_0(g(t,y),\breve\sigma(g(t,y)))-\breve\ell_0(g(t,y),\breve\sigma_0(g(t,y)))| \\
    &\qquad\qquad + |\breve\ell_0(g(t,y),\breve\sigma_0(g(t,y)))-\breve\ell_0(g_0(t,y),\breve\sigma_0(g(t,y)))| \\
    &\qquad\qquad + |\breve\ell_0(g_0(t,y),\breve\sigma_0(g(t,y)))-\breve\ell_0(g_0(t,y),\breve\sigma_0(g_0(t,y)))| \\
    &\quad\leq \wt C_1\rho(t) + k_x|\breve\sigma(g(t,y))-\breve\sigma_0(g(t,y))| + C_0\wt C\rho(t) + k_{x_0} C_0\wt C\rho(t) \\
    &\quad\leq C_\Sigma\rho(t),
  \end{align*}
  where $\wt C_1$, $\wt C$, resp.\ $C_0$, are upper bounds for $\wt\ell$, the Lipschitz constant of $\wt f$ (via~\eqref{EqMInvBddInv}), resp.\ $\ell_0$, $\breve\sigma_0$; the final inequality requires $\wt C_1+k_x C_\Sigma+C_0\wt C(1+k_{x_0})\leq C_\Sigma$, which can be arranged to hold for all $x$ by fixing
  \begin{equation}
  \label{EqMInvSectBd}
    C_\Sigma > \tfrac{\wt C_1+C_0\wt C(1+K)}{1-K},\quad K:=\sup_{x\in\cX_1}k_x.
  \end{equation}
  Thus $\ell_\sharp\sigma\in\Sigma^0(\eps)$. We further estimate for $\sigma,\sigma'\in\Sigma^0(\eps)$:
  \[
    |\ell_\sharp\sigma(t,y)-\ell_\sharp\sigma'(t,y)| \leq (k_y+o(1))|\sigma(t+1,y)-\sigma'(t+1,y)|, \qquad \eps\to 0.
  \]
  The contraction mapping principle implies the existence of a unique invariant section of $\ell_\sharp$,
  \[
    \sigma_\ell\in\Sigma^0(\eps) \subset \sigma_0+\rho\cC_b^0.
  \]

  \textit{\underline{Step 3:} Lipschitz regularity $(r=1)$.} Using the assumptions for $r=1$, we next prove that $\sigma_\ell$ is Lipschitz; this uses the estimate~\eqref{EqMInvLip}. Consider the space of sections
  \begin{equation}
  \label{EqMInvSectSpace}
    \Sigma(\eps) := \bigl\{ \sigma\in\Sigma^0(\eps) \colon L_{(t,x)}(\breve\sigma-\breve\sigma_0) \leq C_\Sigma\rho(t) \bigr\}.
  \end{equation}
  Here, the Lipschitz constant is defined using the triviality of $\bar\cE$ by
  \begin{equation}
  \label{EqMInvLipDef}
    L_{(t,x)}(\breve\sigma) := \limsup_{(s,y)\to(t,x)} \frac{|\breve\sigma(s,y)-\breve\sigma(t,x)|_{\bar\cE_x}}{|s-t|+d(y,x)},
  \end{equation}
  where $d$ is the Riemannian distance function on $\cX$. We contend that $\ell_\sharp(\sigma)\in\Sigma(\eps)$ for $\sigma\in\Sigma(\eps)$. Indeed, for $(t+1,x)=g(t,y)$, we have the bound
  \begin{align*}
    L_{(t,y)}(\pi_E(\ell_\sharp\sigma-\sigma_0)) & \leq L_{(t,y)}(\pi_E(\ell\sigma-\ell_0\sigma)g) + L_{(t,y)}(\pi_E(\ell_0\sigma-\ell_0\sigma_0)g) \\
      &\leq \bigl(L_{(t,y)}(\wt\ell)L_{(t+1,x)}(\breve\sigma) + L_{(t,y)}(\pi_E\ell_0)L_{(t+1,x)}(\breve\sigma-\breve\sigma_0)\bigr) L_{(t,x)}(g) \\
      &\leq \bigl(\wt C\rho(t)(C_0+C_\Sigma\rho(t)) + k_{\bar f^{-1}(y)} C_\Sigma\rho(t+1)\bigr)(\alpha_{\bar f^{-1}(y)}+o(1)) \\
      &\leq (k_x\alpha_x+o(1))C_\Sigma\rho(t), \qquad \eps\to 0
  \end{align*}
  where in the last step we used $|\bar f^{-1}(y)-x|\leq\wt C\rho(t)$ from~\eqref{EqMInvBddInv}, and also that $\wt C C_0+k_x C_\Sigma<C_\Sigma$ by our choice of $\Sigma$. In view of assumption~\eqref{EqMInvContr}, this proves our contention. Applying the contraction mapping principle produces a unique $\ell_\sharp$-invariant section in $\Sigma(\eps)$, which must be equal to that constructed in the previous step. Thus,
  \[
    \sigma_\ell\in\Sigma(\eps).
  \]
  
  \textit{\underline{Step 4:} pointwise differentiability $(r=1)$.} As in \cite[\S3]{HirschPughShubInvariantManifolds}, consider the bundle $J^b\to f(\cM_1(\eps))=:\cM_2(\eps)$, where the base is equipped with the discrete topology;\footnote{This circumvents difficulties in giving $J^b$ the structure of a Banach bundle over $\cM_2(\eps)$ with its usual topology; see also the proof of \cite[Theorem~(3.5)]{HirschPughShubInvariantManifolds}.} the fibers are
  \[
    J^b_{(t,x)} = \{ J_{(t,x)}\sigma \in J^b(\cM_2(\eps),(t,x);\cE_0,\sigma_\ell(t,x)) \colon \sigma\in\Sigma(\eps) \},
  \]
  where we define $\sigma_\ell=\sigma$ on $\cM_1(\eps)$ and $\sigma_\ell(t,x)=\ell(\sigma(g(t,x)))$ for $(t,x)\in\cM_2(\eps)\setminus\cM_1(\eps)$, and membership in $\Sigma(\eps)$ near such $(t,x)$ is defined by the bounds in~\eqref{EqMInvSectSpace}. The bundle map $\ell$ induces a bundle map $J\ell$ on $J^b|_{\cM_1(\eps)}$, covering $f$, by
  \[
    J\ell(J_{(t,x)}\sigma) = J_{f(t,x)}(\ell_\sharp\sigma).
  \]
  This is well-defined by the estimates from the previous step. It is also a fiber contraction:
  \begin{align*}
    |J\ell(J_{(t,x)}\sigma)-J\ell(J_{(t,x)}\sigma')| &= L_{f(t,x)}(\pi_E\ell\sigma g-\pi_E\ell\sigma' g) \\
      &\leq L_{\sigma_\ell(t,x)}(\pi_E\ell)L_{(t,x)}(\sigma-\sigma')L_{f(t,x)}(g) \\
      &\leq (k_x+\wt C\rho(t)) (\alpha_{\bar f^{-1}\pi_\cX f(t,x)}+o(1)) L_{(t,x)}(\sigma-\sigma') \\
      &\leq (k_x\alpha_x+o(1))L_{(t,x)}(\sigma-\sigma'),\qquad \eps\to 0,
  \end{align*}
  by~\eqref{EqMInvLip}. The contraction mapping principle produces a unique bounded $J\ell$-invariant section $\sigma_{J\ell}$ of $J$. On the other hand, $J\sigma_\ell$ is a bounded section of $J^b$ which by construction is $J\ell$-invariant, hence
  \[
    J\sigma_\ell=\sigma_{J\ell}.
  \]
  Now, $J\ell$ preserves the closed subbundle $J^d\to\cM_2(\eps)$ of Lipschitz jets of \emph{differentiable} sections, as $g$ is differentiable whenever $f$ is. Thus, $\sigma_{J\ell}$ is necessarily a section of $J^d$, proving the pointwise differentiability of $\sigma_\ell$.

  \textit{\underline{Step 5:} $\rho\cC_b^1$-regularity.} Consider the bundle $L\to\cM_2(\eps)$, where the base carries its standard topology again, and the fibers of $L$ are spaces of linear maps:
  \[
    L_{(t,x)} = \cL(T_{(t,x)}\cM, (\cE_0)_{(t,x)}).
  \]
  We shall work in the subbundle
  \[
    \cB_{(t,x)} = \{ P\in L_{(t,x)} \colon \|P-D_{(t,x)}\sigma_0\|\leq C_\Sigma\rho(t) \},
  \]
  on which we have a bundle map $L\ell$ which acts on $P\in L_{(t,x)}$ by
  \[
    {\rm graph}(L\ell(P)) = (D_{\sigma(t,x)}\ell)({\rm graph}(P))
  \]
  where on the right, we write, using the triviality of $\cE_0$,
  \[
    {\rm graph}(P) := \{ \xi + P\xi \colon \xi\in T_{(t,x)}\cM \} \subset T_{(t,x,\breve\sigma(t,x))}\cE_0,
  \]
  similarly on the left. The previous estimates show that $L\ell$, which covers the base map $f$, contracts the fibers by $k_x\alpha_x+o(1)$, hence, using the contraction mapping principle as before, there exists a unique bounded invariant section of $L\ell$, denoted
  \[
    \sigma_{L\ell} \colon \cM_1 \to \cB.
  \]
  Since $\ell$ is $\cC^1$, so $L\ell$ is $\cC^0$, the section $\sigma_{L\ell}$ is necessarily equal to the unique \emph{continuous} invariant section. The (a priori discontinuous) section $D\sigma_\ell$ of $\cB$ must be equal to $\sigma_{L\ell}$ and is therefore continuous, which gives 
  \[
    \sigma_\ell-\sigma_0 \in \rho\cC_b^1
  \]
  under the $r=1$ assumptions.

  \textit{\underline{Step 6:} $\rho\cC_b^r$-regularity.} We argue inductively. Thus, assume $\sigma_\ell-\sigma_0\in\rho\cC_b^{r-1}$, $r\geq 2$. Consider again the map $L\ell$ from the previous step: under the $\rho\cC_b^r$-assumptions of the theorem, $L\ell$ is a $\rho\cC_b^{r-1}$-perturbation of its stationary analogue $L\ell_0$, and it contracts fibers by $k_x\alpha_x+o(1)$. Therefore, it satisfies the $\rho\cC_b^{r-1}$-assumptions of the theorem, which implies that $\sigma_{L\ell}-\sigma_{L\ell_0}\in\rho\cC_b^{r-1}$. Therefore, $\sigma_\ell-\sigma_0\in\rho\cC_b^r$, finishing the inductive step.
\end{proof}

\subsection{Stable/unstable manifold theorem for maps}

We now turn to the main dynamical result of this paper. Let $\cX$ denote a closed $n$-dimensional ($n\geq 1$) manifold. We make the following assumptions:
\begin{enumerate}[label=(I.\arabic*),ref=I.\arabic*]
\item\label{ItMGamma} $\Gamma$ is a closed $\CI$ submanifold of $\cX$;
\item\label{ItMMap} $\bar f\colon\cX\to\cX$ is a $\CI$ map so that $\bar f(\Gamma)=\Gamma$, and there exists an open neighborhood $\cU$ of $\Gamma$ so that $\bar f\colon\cU\to\bar f(\cU)$ is a diffeomorphism;
\item\label{ItMSplit} there is a $\CI$ bundle splitting
  \begin{equation}
  \label{EqMSplit}
    T_\Gamma\cX = T\Gamma \oplus \bar N^u \oplus \bar N^s
  \end{equation}
  which is preserved by the linearization $D\bar f$ at $\Gamma$. Denote
  \[
    \Gamma_p\bar f:=D_p\bar f|_{T_p\Gamma},\quad
    \bar N^{u/s}_p\bar f:=D_p\bar f|_{\bar N_p^{u/s}};
  \]
\item\label{ItMNormHyp} for all $r\in\N$, the map $\bar f$ is \emph{immediately relatively $r$-normally hyperbolic} at $\Gamma$ in the sense of \cite[Definition~2]{HirschPughShubInvariantManifolds}. That is, for all $r$, there exist fiber metrics on the summands in~\eqref{EqMSplit} with respect to which
  \begin{equation}
  \label{EqMNormHyp}
    m(\bar N^u_p\bar f) > \|\Gamma_p\bar f\|^k,\quad 
    \|\bar N^s_p\bar f\| < m(\Gamma_p\bar f)^k,\qquad \forall\,p\in\Gamma,\ 0\leq k\leq r.
  \end{equation}
\end{enumerate}

By \cite[Theorem~4.1]{HirschPughShubInvariantManifolds}, these assumptions imply that in a neighborhood of $\Gamma$, there exist stable ($s$) and unstable ($u$) manifolds
\[
  \bar\Gamma^s,\,\bar\Gamma^u \subset \cX, \quad T_\Gamma\bar\Gamma^{u/s}=T\Gamma\oplus\bar N^{u/s},
\]
which are $\CI$ and locally invariant under $\bar f$, namely $\bar f(\bar\Gamma^s)\subset\bar\Gamma^s$ and $\bar f(\bar\Gamma^u)\supset\bar\Gamma^u$.

On the spacetime
\begin{equation}
\label{EqMSpacetime}
  \cM := \R_t \times \cX,
\end{equation}
define then the \emph{stationary model} $f_0\colon\cM\to\cM$ by\footnote{One can equally well consider time translations $t\mapsto t+1$ instead, as done in~\S\ref{SI}, in which case unstable and stable manifolds exchange roles in Theorem~\ref{ThmMMap} below. We consider~\eqref{EqMStatMod} for simpler comparison with~\cite[\S4]{HirschPughShubInvariantManifolds} since in this case the unstable manifold is canonical, and it is the unstable manifold which was explicitly discussed in the reference.}
\begin{equation}
\label{EqMStatMod}
  f_0(t,x) := (t-1,\bar f(x)).
\end{equation}

As perturbations of $f_0$, we consider smooth maps
\[
  f\colon \cM \to \cM,
\]
defined for $t>t_0\in\R$, with the following properties:
\begin{enumerate}[label=(II.\arabic*),ref=II.\arabic*]
\item\label{ItMPertTime} let $\pi_T\colon\cM\to\R_t$, $(t,x)\mapsto t$, denote the projection, then $\pi_T(f(t,x))=t-1$;
\item\label{ItMPertClose} $f$ is a $\rho\cC_b^r$-perturbation of $f_0$ for every $r$. That is, fixing a Riemannian metric on $\cX$ and denoting by $\exp$ its exponential map, we have
  \[
    f(t,x)=\bigl(t-1,\exp_{\bar f(x)}\wt V(t,\bar f(x))\bigr),\quad x\in\cU,
  \]
  where $\wt V(t,-)\in\cV(\cX)$, $\wt V\in\rho\CI_b(\{t>t_0\};T\cX)$, is a $t$-dependent vector field on $\cX$.
\end{enumerate}

Let us denote by
\begin{equation}
\label{EqMMapSpExt}
  \Gamma^{u/s}_0 := \R_t \times \bar\Gamma^{u/s}
\end{equation}
the stationary spacetime extension of the unstable and stable manifolds, and $\Gamma_0:=\R_t\times\Gamma$.

\begin{thm}
\label{ThmMMap}
  Under the assumptions~\eqref{ItMGamma}--\eqref{ItMNormHyp} and \eqref{ItMPertTime}--\eqref{ItMPertClose}, there exists a submanifold $\Gamma^u\subset\cM$ with the following properties:
  \begin{enumerate}
  \item\label{ItMMapUInv} $\Gamma^u$ is $f$-invariant in the sense that
    \begin{equation}
    \label{EqMMapUNiv}
      f(\Gamma^u) \cap \{t>t_0\} = \Gamma^u \cap \{t>t_0\};
    \end{equation}
  \item\label{ItMMapU} $\Gamma^u$ approaches $\Gamma^u_0$ as $t\to\infty$ in a $\rho\CI_b$ sense. This means: let $\bar\Gamma^u(\eps)$ denote an $\eps$-neighborhood of $\Gamma$ within $\bar\Gamma^u$, and fix a $\CI$ tubular neighborhood of $\bar\Gamma^u(\eps)$ in $\cX$. Extend this $t$-independently to a tubular neighborhood of $\Gamma^u_0(\eps)$ in $\cM$. Then, for small $\eps>0$, the unstable manifold $\Gamma^u$ is the graph of a function in the space $\rho\CI_b(\Gamma^u_0(\eps);N\Gamma^u_0)$;
  \item\label{ItMMapUUniq} $\Gamma^u$ is the unique manifold satisfying~\eqref{EqMMapUNiv} within the class of manifolds approaching $\Gamma^u_0$ in a $\rho\cC_b^1$ sense.
  \ctrsave
  \end{enumerate}

  Furthermore:
  \begin{enumerate}
  \ctrset
  \item\label{ItMMapSExt} there exists a (non-unique) manifold $\Gamma^s\subset\cM$ approaching $\Gamma^s_0$ as $t\to\infty$ in a $\rho\CI_b$ sense, that is, $\Gamma^s$ is the graph of a function in the space $\rho\CI_b(\Gamma^s_0(\eps);N\Gamma^s_0)$ (defined using the stationary extension of a tubular neighborhood of $\bar\Gamma^s$ analogously to~\eqref{ItMMapU}) so that $f(\Gamma^s)\cap\{t>t_0\}\subset\Gamma^s$.
  \end{enumerate}
  In these statements, the regularity of $\Gamma^u$ and $\Gamma^s$ is $\rho\cC_b^r$ if we relax assumption~\eqref{ItMPertClose} by only assuming $\rho\cC_b^r$-regularity, for some fixed $r\geq 1$, of $\wt V$.\footnote{One can likewise relax assumptions~\eqref{ItMSplit} and~\eqref{ItMNormHyp} by assuming that the bundle splitting~\eqref{EqMSplit} is merely continuous, and only assuming $r$-normal hyperbolicity for $r$ fixed, in which case $\Gamma$, $\bar\Gamma^{u/s}$ are merely $\cC^r$. This adds only minor technical complications to the proof, which can be handled by smoothing techniques as in~\cite[\S4]{HirschPughShubInvariantManifolds}. We do not state the theorem in this generality, as it will not be needed in our application.}
\end{thm}

\begin{figure}[!ht]
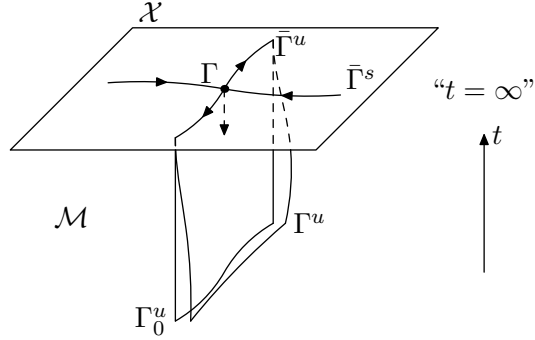

  \centering
  \inclfig{FigMMap}
  \caption{Illustration of Theorem~\ref{ThmMMap}, parts~\eqref{ItMMapUInv} and \eqref{ItMMapU}. The stationary model $f_0$ (or $\bar f$) should be thought of living at ``$t=\infty$'', see Remark~\ref{RmkMCptExp}. Its invariant ($\Gamma$), stable ($\Gamma^s$), and unstable ($\bar\Gamma^u$) manifolds are shown, as well as the stationary extension $\Gamma^u_0$. The unstable manifold for $f$ is a $\rho\CI_b$-graph over $\Gamma^u_0$. Also drawn is the arrow indicating the action of $f$ in the time variable.}
  \label{FigMMap}
\end{figure}

Assuming that $\bar\Gamma^u$ is orientable, an equivalent formulation of~\eqref{ItMMapU} is the following:
\begin{enumerate}[label=(\ref*{ItMMapU}'),ref=\ref*{ItMMapU}']
\item\label{ItMMapUalt} let $\bar\varphi^u\in\CI(\cX)$ denote a defining function of $\bar\Gamma^u$, that is, $(\bar\varphi^u)^{-1}(0)=\bar\Gamma^u$, and $d\bar\varphi^u\neq 0$ on $\bar\Gamma^u$. Then there exists a function $\wt\varphi^u\in\rho\CI_b(\cM)$ such that
  \begin{equation}
  \label{EqMMapUalt}
    \varphi^u(t,x):=\bar\varphi^u(x)+\wt\varphi^u(t,x)
  \end{equation}
  is a defining function of $\Gamma^u$.
\end{enumerate}

Similarly, claim~\eqref{ItMMapSExt} can be restated using defining functions when $\bar\Gamma^s$ is orientable:
\begin{enumerate}[label=(\ref*{ItMMapSExt}'),ref=\ref*{ItMMapSExt}']
\item\label{ItMMapSExtalt} let $\bar\varphi^s\in\CI(\cX)$ denote a defining function of $\bar\Gamma^s$. There exists a function $\wt\varphi^s\in\rho\CI_b(\cM)$ such that
  \begin{equation}
  \label{EqMMapSExtalt}
    \varphi^s(t,x):=\bar\varphi^s(x)+\wt\varphi^s(t,x)
  \end{equation}
  is a defining function of a manifold $\Gamma^s$ as in~\eqref{ItMMapSExt}.
\end{enumerate}

\begin{rmk}
\label{RmkMCptExp}
  Introducing a coordinate $\tau_e:=e^{-t}$ as in the introduction, the action of $f_0$ in the time variable is given by $\tau_e\mapsto e\tau_e$, hence $f_0$ induces a map on $[0,\infty)_{\tau_e}\times\cX$ for which $\{0\}\times\bar\Gamma^s$ is the stable manifold and $[0,\infty)\times\bar\Gamma^u$ the unstable manifold. (While $\Gamma^u_0$ has an invariant interpretation as the interior of the unstable manifold of $\{0\}\times\bar\Gamma$, the manifold $\Gamma^s_0$ has no such interpretation, and we only introduce it for convenience.) The perturbations considered here, which are very singular in $\tau_e$ (for $\rho(t)=\la t\ra^{-\alpha}$ of size $\la\log\tau_{\rm\exp}\ra^{-\alpha}$), can be analyzed because of the particular $t$-dependence of $f$. Namely, $\pi_T f$ is a family (depending on $t$) of smooth small perturbations of $\bar f$; one can show that this is already sufficient to guarantee the existence of a family of manifolds $\Gamma^u(t)\subset\cX$ which remain close to $\bar\Gamma^u$ for all $t$ and are invariant in the sense that $f(t,\Gamma^u(t))\subset\Gamma^u(t-1)$. The point here is that under our stronger assumptions on $f$, we can obtain more precise control on $\Gamma^u$.
\end{rmk}

\begin{proof}[Proof of Theorem~\usref{ThmMMap}]
We closely follow the constructions and arguments of~\cite[\S4]{HirschPughShubInvariantManifolds} and explain the required extensions and modifications. The first, main, part of the proof concerns claims~\eqref{ItMMapUInv}--\eqref{ItMMapUUniq} about $\Gamma^u$; in the second part, we prove claim~\eqref{ItMMapSExt} using a simple adaptation of the methods from the first part. As in the reference, the construction of $\Gamma^u$ proceeds in several steps:
  \begin{itemize}
  \item After some preliminary simplifications, we will find $\Gamma^u$ as the graph of a section $\sigma$ via a fixed point argument using a graph transform \`a la Hadamard, which was already used in~\S\ref{SsInv}. Here, starting e.g.\ with the candidate $\Gamma^u_0$, we replace the current candidate by its image under $f$, which stretches it out along $\Gamma^u_0$ and flattens it in the stable directions. (The perturbative part of $f$ is damped by the main part, $f_0$, in the stable directions.) In steps 1 and 2, we show that this is a well-defined procedure on candidate sections which are small sections of the normal bundle $\Gamma^u_0$.
  \item In step 3, we prove that this graph transform is a contraction on a space of Lipschitz sections decaying (together with its pointwise Lipschitz constants) at rate $\rho$, thus furnishing $\Gamma^u$ as a Lipschitz submanifold $\rho$-close to $\Gamma^u_0$.\footnote{As a simple toy example, one may keep the following in mind: $\Gamma$ is a point, $\bar\Gamma^u$ and $\bar\Gamma^s$ (with coordinates $x^U$ and $s$, respectively) are 1-dimensional, $\bar f(x^U,s)=(2 x^U,\tfrac12 s)$, so $f_0(t,x^U,s)=(t-1,2 x^U,\tfrac12 s)$, and the perturbed map is $f(t,x^U,s)=(t-1,2 x^U,\tfrac12 s+\rho(t))$. The graph transform, denoted $f_\sharp$ in~\eqref{EqMMapContr}, maps the section $(t,x^U)\mapsto\sigma(t,x^U)=(t,x^U,\breve\sigma(t,x^U))$ into $f_\sharp\sigma(t,x^U)=\bigl(t,x^U,\tfrac12\breve\sigma(t+1,x^U/2)+\rho(t+1)\bigr)$. Acting on $\cO(\rho)$-sections, the unique fixed point of this transform is the section $\sigma_f(t,x^U) = \sum_{j=0}^\infty 2^{-j}\rho(t+1+j)$.}
  \item In step 4, we improve this to pointwise differentiability using the contraction mapping principle applied to a graph transform on a priori discontinuous families of Lipschitz jets. The key point, as in the proof of Theorem~\ref{ThmMInv}, is that the graph transform preserves jets of differentiable sections, which implies the pointwise differentiability of $\Gamma^u$.
  \item In step 5, we show in a similar manner (considering the graph transform acting on tangent planes based at points of $\Gamma^u$, on the one hand on a priori discontinuous, on the other hand on continuous sections) that the tangent distribution of $\Gamma^u$ must be continuous.
  \item Higher regularity is proved in step 6 using an inductive argument in which one improves the regularity of the tangent bundle of $\Gamma^u$.
  \end{itemize}

  To start, fix a Riemannian metric on $\cX$ which restricts to that of assumption~\eqref{ItMNormHyp} (with $r=1$) at $\Gamma$. For $\eps>0$ small, denote by $\bar\Gamma^u(\eps)$ the $\eps$-neighborhood of $\Gamma$ within $\bar\Gamma^u$. Extend the bundle $\bar N^s\to\Gamma$ of stable tangent directions and its fiber metric to a $\CI$ vector bundle $\bar\cS'\to\bar\Gamma^u(\eps)$. Using the exponential map on $\cX$, embed an $\eps$-neighborhood, $\bar\cS'(\eps)$, of the zero section diffeomorphically into $\cX$. For convenience, we fix a vector bundle $\bar\cS''\to\bar\Gamma^u(\eps)$ such that $\bar\cS:=\bar\cS'\oplus\bar\cS''$ is trivial, i.e.\ $\bar\cS\cong\bar\Gamma^u(\eps)\times S$ for some fixed vector space $S$. Choose the fiber metric on $\bar\cS$ to be the direct sum of that on $\bar\cS'$ and any fixed metric on $\bar\cS''$. We extend $\bar f$ to a map on $\bar\cS(\eps)$ by setting $\bar f(s'\oplus s''):=\bar f(s')\oplus 0$; we likewise extend $f$ to a map on $\R_t\times\bar\cS(\eps)$ by setting
  \begin{equation}
  \label{EqMMapTrivExt}
    f(t,s'\oplus s''):=f(t,s')\oplus 0,
  \end{equation}
  similarly for $f_0$. We let $\cS:=\R_t\times\bar\cS$.

  Let $C_\Sigma>0$ denote a constant, which will be specified in the course of the proof and only depends on Lipschitz properties of $f$, see~\eqref{EqMMapCS}. For small $\eps>0$, we set
  \begin{equation}
  \label{EqMMapTint}
  \begin{gathered}
    T(\eps) := \{ t\in\R\colon\rho(t)<\eps^2 \}, \\
    \Gamma^u_0(\eps) := T(\eps) \times \bar\Gamma^u(\eps), \quad
    \cS(\eps) := T(\eps) \times \bar\cS(\eps),
  \end{gathered}
  \end{equation}
  in particular $\cS(\eps)\cong\Gamma^u_0(\eps)\times S$ is trivial; we denote the projection by
  \[
    \pi^U_0 \colon \cS(\eps) \to \Gamma^u_0(\eps),
  \]
  and $\bar\pi^U:=\pi_\cX\pi^U_0$, where $\pi_\cX\colon\cM\to\cX$ is the projection. (We use the capital letter `$U$' to emphasize that this projects onto more than the unstable tangent directions $\bar N^u$, which is only \emph{part} of the tangent bundle of $\bar\Gamma^u$ as it does not include $T\Gamma$.) Denote points on $\bar\Gamma^u$ by $x^U$.
  
  \textit{\textbf{\underline{Part 1:}} proof of  claims~\eqref{ItMMapUInv}--\eqref{ItMMapUUniq}.} We shall find $\Gamma^u\subset\cS(\eps)$ as the image of a section in the space
  \begin{equation}
  \label{EqMMapSect}
    \Sigma(\eps) := \bigl\{ \text{sections}\ \sigma\colon\Gamma^u_0(\eps)\to\cS(\eps)\colon |\sigma(t,x^U)|\leq C_\Sigma\rho(t), \ L_{(t,x^U)}(\breve\sigma)\leq C_\Sigma\rho(t) \bigr\},
  \end{equation}
  where we define the fiber value map $\breve\sigma\colon\Gamma^u_0(\eps)\to S$ by $\sigma(t,x^U)=(t,x^U,\breve\sigma(t,x^U))$; the Lipschitz constant is defined analogously to~\eqref{EqMInvLipDef} using the triviality of $\bar\cS$. The space $\Sigma(\eps)$, equipped with the $\cC_b^0$ metric, is complete. The idea is to define a graph transform on $\Sigma(\eps)$ by mapping $\sigma$ to a section $f_\sharp\sigma$ whose graph is the image of the graph of $\sigma$ under $f$: we wish to take
  \begin{equation}
  \label{EqMMapContr}
    f_\sharp\sigma := f\sigma g,
  \end{equation}
  where $g\colon\bar\Gamma^u(\eps)\to\bar\Gamma^u(\eps)$ is a right inverse of $\pi^U_0 f\sigma$. We shall see that this is a contraction for two reasons: first, $f$ expands, thus $g$ contracts, the base $\Gamma^u_0(\eps)$, and second, $f$ contracts the fibers. Here, the behavior of $f$ is dominated by $f_0$, while the perturbation only affects these statements by a decaying (as $t\to\infty$) amount.

  \textit{\underline{Step 1:} construction of $g$.} We first show that the right inverse $g$ is well-defined when $\eps>0$ is sufficiently small. We introduce local coordinates on $\bar\Gamma^u(\eps_0)$, $\eps_0>0$ small and fixed, using the Riemannian exponential map $\exp^{\bar\Gamma^u}$ as follows: for $p\in\Gamma$ and small $v^U=(v,n)\in T_p\bar\Gamma^u=T_p\Gamma\oplus\bar N^u_p$, we let
  \begin{equation}
  \label{EqMMapChi}
    \bar\chi_p(v^U) := \exp^{\bar\Gamma^u}_p(v^U+h_p(v^U)),
  \end{equation}
  where $h_p\colon T_p\bar\Gamma^u\to T_p\bar\Gamma^u$ vanishes quadratically at $0$ and is chosen such that for $v\in T_p\Gamma$, we have $\exp^{\bar\Gamma^u}_p(v+h_p(v))=\exp_p^\Gamma(v)\in\Gamma$, thus straightening out $\Gamma$. Using the triviality of $\bar\cS$, this induces local coordinates $(v^U,s)\in T_p\bar\Gamma^u\oplus S$ on $\bar\cS$ via
  \begin{equation}
  \label{EqMMapE}
    \bar e_p(v^U,s) := (\bar\chi_p(v^U), s).
  \end{equation}
  The local coordinate representation of $\bar f$,
  \[
    \bar f_p := \bar e_{\bar f(p)}^{-1} \bar f \bar e_p \colon T_p\bar\Gamma^u \times S \to T_{\bar f(p)}\bar\Gamma^u \times S,
  \]
  can be written as $\bar f_p=D_p\bar f + \bar r_p$, where $\bar r_p$ vanishes quadratically at $0$. Since $\bar f$ preserves $\Gamma$ as well as $\bar\Gamma^u$, $\bar f_p$ preserves the origin of $\bar N^u\oplus S$ as well as of $S$ (considered as subbundles of $T\Gamma\oplus\bar N^u\oplus S$). Moreover,
  \begin{equation}
  \label{EqMMapTf}
    D_p\bar f = \begin{pmatrix} \Gamma_p\bar f\oplus\bar N^u_p\bar f & 0 \\ 0 & D_p\bar f|_{\bar\cS_p} \end{pmatrix}.
  \end{equation}
  The charts~\eqref{EqMMapChi}--\eqref{EqMMapE} have natural spacetime extensions $\chi_p$ (giving coordinates on $\Gamma^u_0$) and $e_p$ (giving coordinates on $\cS$):
  \[
    \chi_p(t,v^U) := (t,\bar\chi_p(v^U)), \quad
    e_p(t,v^U,s) := (t,\bar e_p(v^U,s)).
  \]
  Writing $(f_0)_p=e_{\bar f(p)}^{-1}f_0 e_p$ and $f_p=e_{\bar f(p)}^{-1}f e_p$ (note that we use the chart centered at $\bar f(p)$ even for the perturbed map $f$), we have $(f_0)_p(t,v^U,s)=(t-1,\bar f_p(v^U,s))$, while the perturbation
  \[
    \wt f_p := f_p - (f_0)_p,\quad
    \pi_T\wt f_p(t,v^U,s)\equiv 0,
  \]
  satisfies
  \[
    \wt f_p\in\rho\CI_b
  \]
  as a map $\{t>t_0\}\times\cW\to\{0\}\times T_{\bar f(p)}\bar\Gamma^u\times S$ where $\cW$ is a neighborhood of the origin in $T_p\bar\Gamma^u\times S$, as follows from assumptions~\eqref{ItMPertTime}--\eqref{ItMPertClose} (increasing $t_0$ if necessary to stay within the coordinate patches $e_p$ and $e_{\bar f(p)}$).

  Let now $\sigma\in\Sigma(\eps)$. Write $\sigma_p:=e_p^{-1}\sigma\chi_p\colon T(\eps)\times T_p\bar\Gamma^u\to T(\eps)\times T_p\bar\Gamma^u\times S$, which has the form $\sigma_p(t,v^U)=(t,v^U,\breve\sigma_p(t,v^U))$, and satisfies the bounds in~\eqref{EqMMapSect} (with respect to the fixed fiber metric on $S$ coming from $\bar\cS_p$), with the right hand sides multiplied by $1+\cO(\eps)$ (since the norm on the fibers of $\cS$ may change smoothly away from $p$). We then have
  \[
    \chi_{\bar f(p)}^{-1}\pi^U_0 f\sigma\chi_p = \pi^U_0 e_{\bar f(p)}^{-1}f e_p e_p^{-1}\sigma\chi_p = \pi^U_0 f_p\sigma_p = L+(0,\bar\pi^U\wt L),
  \]
  where the main and remainder terms (for present purposes) are, respectively,
  \begin{equation}
  \label{EqMMapFpSigmap}
    L(t,v^U)=(t-1,D_p\bar f(v^U)), \quad
    \wt L(t,v^U)=D_p\bar f\circ(0,\breve\sigma_p(t,v^U))+\bar r_p(v^U,\breve\sigma_p)+\wt f_p\sigma_p(t,v^U).
  \end{equation}
  (Thus, $L+(0,\wt L)$ is the local coordinate representation of $f\sigma$.) Fix $\lambda,\mu$ and $\omega,\omega'$ (depending on $p$) so that
  \[
    0<\mu<\ul\mu:=m(\Gamma_p\bar f), \quad
    1<\lambda<\ul\lambda:=m(\bar N^u_p\bar f), \quad
    0<\omega<\omega'<1,\quad \lambda\omega>1.
  \]
  We claim that, for $\eps_0>0$ fixed and small, we have, for all sufficiently small $\eps>0$:
  \begin{gather}
  \label{EqMMapInj}
    \chi_{\bar f(p)}^{-1}\pi_0^U f\sigma\chi_p\ \ \text{is injective on}\ \ T(\eps)\times(T_p\Gamma(\eps_0)\oplus\bar N^u_p(\omega'\eps)), \\
  \label{EqMMapOverflow}
    \chi_{\bar f(p)}^{-1}\pi_0^U f\sigma\chi_p\bigl(T(\eps)\times(T_p\Gamma(\omega\eps)\oplus\bar N^u_p(\omega\eps))\bigr) \supset T(\eps)\times(T_{\bar f(p)}\Gamma(\mu\omega\eps)\oplus\bar N^u_{\bar f(p)}(\lambda\omega\eps)).
  \end{gather}
  (Note that if $\wt L$ were identically zero, this would be clear.) Now, only points $(t,v^U)$ with the same $t$-values can possibly map to the same point under $L+(0,\bar\pi^U\wt L)$, hence~\eqref{EqMMapInj} follows from~\eqref{EqMMapTf} (which implies that $\bar\pi^U$ annihilates the first term of $\wt L$) and
  \begin{equation}
  \label{EqMMapInjEst}
  \begin{split}
    &\big|\chi_{\bar f(p)}^{-1}\pi_0^U f\sigma\chi_p(t,v^U)-\chi_{\bar f(p)}^{-1}\pi_0^U f\sigma\chi_p(t,w^U)\big| \\
    &\qquad \geq |D_p\bar f(v^U-w^U)| - |\bar r_p(v^U,\breve\sigma_p(t,v^U))-\bar r_p(w^U,\breve\sigma_p(t,w^U))| \\
    &\hspace{8em} -|\bar\pi^U\wt f_p(t,v^U,\breve\sigma_p(t,v^U))-\bar\pi^U\wt f_p(t,w^U,\breve\sigma_p(t,w^U))| \\
    &\qquad \geq \bigl(\ul\mu-C_r(\eps_0+\eps)(1+C_\Sigma)-\wt C\rho(t)(C_\Sigma\rho(t)+1)\bigr)(1-\cO(\eps_0+\eps))|v^U-w^U| \\
    &\qquad \geq (\ul\mu-o(1))|v^U-w^U|,\qquad \eps_0,\eps\to 0,
  \end{split}
  \end{equation}
  where $\wt C$ is greater than the Lipschitz constant of $\rho^{-1}\wt f_p$, and $C_r$ bounds the $\cC^2$ norm of $\bar r_p$; we may fix $\wt C,C_r$ uniformly for $p\in\Gamma$. The factor $1-\cO(\eps_0+\eps)$ accounts for the fact that the norms used in the definition of Lipschitz constants may vary away from $p$.

  To prove~\eqref{EqMMapOverflow}, note that $|\bar\pi^U\wt L(t,v^U)|\leq C_r\eps^2+\wt C\rho(t)$ for $(t,v^U)\in T(\eps)\times(T_p\Gamma(\omega\eps)\oplus\bar N^u_p(\omega\eps))$. Let now $t\in T(\eps)$, and consider $y^U=(y,n)\in T_{\bar f(p)}\Gamma\oplus\bar N^u_{\bar f(p)}$ satisfying (slightly increasing $C_r$ and $\wt C$ to accommodate the $(1-\cO(\eps_0+\eps))$ factor in~\eqref{EqMMapInjEst}, now with $\eps_0=\eps$)
  \begin{equation}
  \label{EqMMapOverflowGen}
    |y| \leq (\ul\mu-\delta)\omega\eps-\ul\mu(C_r\eps^2+\wt C\rho(t+1)), \quad
    |n| \leq (\ul\lambda-\delta)\omega\eps-\ul\lambda(C_r\eps^2+\wt C\rho(t+1)),
  \end{equation}
  where $0<\delta<\min(\ul\mu-\mu,\ul\lambda-\lambda)$. Using the form~\eqref{EqMMapTf} of $D_p\bar f$, we see that
  \begin{equation}
  \label{EqMMapIt}
    G \colon v^U \mapsto (D_p\bar f)^{-1}\bigl(y^U - \ol\pi^U\wt L(t+1,v^U)\bigr)
  \end{equation}
  maps $T_p\Gamma(\omega\eps)\oplus\bar N^u_p(\omega\eps)$ into itself for small $\eps>0$ since $\|(\Gamma_p\bar f)^{-1}\|\leq\ul\mu^{-1}$ and $\|(\bar N^u_p\bar f)^{-1}\|\leq\ul\lambda^{-1}$. It also is a contraction, since, similarly to~\eqref{EqMMapInjEst},
  \[
    |G(v^U)-G(w^U)| \leq \ul\mu^{-1}\bigl(2\eps C_r+\wt C\rho(t+1)(C_\Sigma\rho(t+1)+1)\bigr)|v^U-w^U| = o(1)|v^U-w^U|
  \]
  as $\eps\to 0$. Letting $v^U$ denote the unique fixed point of $G$, we have
  \begin{equation}
  \label{EqMMapItSolv}
    \chi_{\bar f(p)}^{-1}\pi^U_0 f\sigma\chi_p(t+1,v^U)=(t,y^U).
  \end{equation}
  Choosing $\eps$ sufficiently small and recalling from~\eqref{EqMMapTint} that $\rho(t+1)<\eps^2$, the second summand on the right in~\eqref{EqMMapOverflow} is contained in the set of $(t,y,n)$ satisfying~\eqref{EqMMapOverflowGen} and $t\in T(\eps)$, hence we have proved~\eqref{EqMMapOverflow}.
  
  We define
  \[
    g \colon \cS(\eps) \to T(\eps) \times \bigcup_{p\in\Gamma} \chi_p\bigl(T_p\Gamma(\omega\eps)\oplus\bar N^u_p(\omega\eps)\bigr)
  \]
  to be the map which in the $\chi_p$, $\chi_{\bar f(p)}$ charts is given by $(t,y^U)\mapsto(t+1,v^U)$ with $v^U$ as in~\eqref{EqMMapItSolv}. In these charts, we record (using~\eqref{EqMMapIt}) that for $(t+1,v^U)=g(t,y^U)$ and $(t'+1,v^{\prime U})=g(t',y^{\prime U})$, and $\rho_\pm=\rho_\pm(t+1,t'+1)$,
  \begin{align*}
    |v^U-v^{\prime U}| &\leq \ul\mu^{-1}\Bigl( |y^U-y^{\prime U}| + \bigl(2\eps C_r+\wt C\rho_-(C_\Sigma\rho_-+1)\bigr)|v^U-v^{\prime U}| \\
      &\qquad\hspace{8em} + \bigl(2\eps C_r C_\Sigma + \wt C(C_\Sigma\rho_++1)\bigr)\rho_+|t-t'|\Bigr);
  \end{align*}
  this uses the Lipschitz estimate~\eqref{EqMLipEst} applied to $\wt f_p$ as well as the Lipschitz bounds on $\sigma$. Hence
  \begin{equation}
  \label{EqMMapInvEst}
    |g(t,y^U)-g(t',y^{\prime U})| \leq (\ul\mu^{-1}+o(1))|y^U-y^{\prime U}| + (1+o(1))|t-t'|,\qquad \eps\to 0.
  \end{equation}

  \textit{\underline{Step 2:} mapping properties of $f_\sharp$.} For $\sigma\in\Sigma(\eps)$, we can now define $f_\sharp\sigma=f\sigma g$ as a section $\Gamma^u_0(\eps)\to\cS$. We proceed to check that $f_\sharp\sigma\in\Sigma(\eps)$. Let us work in local coordinates as above, and let $(t+1,v^U)=g(t,y^U)$. Using~\eqref{EqMMapFpSigmap}, and writing $\pi_S$ for the projection onto the fiber, the length of $(\pi_S f_\sharp\sigma)(t,y^U)\in\cS_{(t,y^U)}=S$ is bounded by
  \begin{equation}
  \label{EqMMapLinfty}
  \begin{split}
    |(\pi_S f_\sharp\sigma)(t,y^U)| &\leq |\pi_S D_p\bar f(v^U,\breve\sigma_p(t+1,v^U))| + |\pi_S\bar r_p(v^U,\breve\sigma_p(t+1,v^U))| \\
      &\qquad + |\pi_S\wt f_p(t+1,v^U,\breve\sigma_p(t+1,v^U))| \\
      &\leq \|\bar N^s_p\bar f\| C_\Sigma \rho(t+1) + \eps C_r C_\Sigma\rho(t+1) + \wt C\rho(t+1),
  \end{split}
  \end{equation}
  where we used crucially that $\bar f$ preserves $\bar\Gamma^u$ in the estimate of the term involving $\bar r_p$, as this gives $\pi_S\bar r_p(v^U,0)=0$. Since $\|\bar N^s_p\bar f\|<1$, we may fix
  \begin{equation}
  \label{EqMMapCS}
    C_\Sigma > \frac{\wt C}{\inf_{p\in\Gamma}\bigl(\min(1,m(\Gamma_p\bar f))-\|\bar N^s_p\bar f\|\bigr)}.
  \end{equation}
  For small $\eps>0$, this implies the desired estimate
  \begin{equation}
  \label{EqMMapBounded}
    |\pi_S f_\sharp\sigma|\leq C_\Sigma\rho.
  \end{equation}

  In order to verify the Lipschitz condition of $\Sigma(\eps)$ for $f_\sharp\sigma$, we estimate, using~\eqref{EqMMapInvEst},
  \begin{equation}
  \label{EqMMapLipschitz}
  \begin{split}
    &L_{(t,y^U)}(\pi_S f_\sharp\sigma) \\
    &\quad\leq L_{(t+1,v^U)}(\pi_S f\sigma) L_{(t,y^U)}(g) \\
    &\quad\leq \bigl(\|\bar N^s_{v^U}\bar f\| C_\Sigma\rho(t+1) + \eps C_r(C_\Sigma\rho(t+1)+1) + \wt C\rho(t+1)(C_\Sigma\rho(t+1)+1)\bigr) \\
    &\quad\qquad \times \max(\ul\mu^{-1},1)(1+o(1)) \\
    &\quad \leq C_\Sigma\rho(t)
  \end{split}
  \end{equation}
  for small $\eps>0$ due to our choice~\eqref{EqMMapCS}.
  
  \textit{\underline{Step 3:} $f_\sharp$ is a contraction.} We next show that $f_\sharp$ is a contraction on $\Sigma(\eps)$ equipped with the (complete!) $\cC_b^0$ metric. To this end, let $\sigma,\sigma'\in\Sigma(\eps)$ and denote by $g,g'$ the right inverses of $\pi^U_0 f\sigma$, $\pi^U_0 f\sigma'$. Recalling the definition of $\wt L$ from~\eqref{EqMMapFpSigmap}, let us make the dependence on $\sigma$ explicit by writing $\wt L_\sigma$. Using the construction of $g$, resp.\ $g'$, via the fixed point argument involving~\eqref{EqMMapIt}, with $\wt L$ replaced by $\wt L_\sigma$, resp.\ $\wt L_{\sigma'}$, we estimate at a point $(t,y^U)$, with $y^U$ $\eps$-close to $p\in\Gamma$, and using the local coordinates as above:
  \begin{equation}
  \label{EqMMapDiffG1}
  \begin{split}
    |g-g'|&\leq \ul\mu^{-1}|\bar\pi^U\wt L_\sigma(g)-\bar\pi^U\wt L_{\sigma'}(g')| \\
    &\leq |\bar\pi^U\bar r_p(\bar\pi^U g,\breve\sigma_p(t+1,\bar\pi^U g))-\bar\pi^U\bar r_p(\bar\pi^U g',\breve\sigma'_p(t+1,\bar\pi^U g'))| \\
    &\quad + |\bar\pi^U\wt f_p(t+1,\bar\pi^U g,\breve\sigma_p(t+1,\bar\pi^U g))-\bar\pi^U\wt f_p(t+1,\bar\pi^U g',\breve\sigma'_p(t+1,\bar\pi^U g'))| \\
    &\leq \bigl(\eps C_r+\wt C\rho(t+1)(1+C_\Sigma\rho(t+1))\bigr)|\bar\pi^U g-\bar\pi^U g'| \\
    &\quad + (\eps C_r+\wt C\rho(t+1))|\breve\sigma_p(t+1,\bar\pi^U g)-\breve\sigma'_p(t+1,\bar\pi^U g)|;
  \end{split}
  \end{equation}
  the first term on the right can be absorbed into the left hand side, and we thus obtain, in our local coordinates,
  \begin{equation}
  \label{EqMMapDiffG2}
    |g-g'|\leq o(1)|\sigma-\sigma'|, \qquad \eps\to 0.
  \end{equation}
  We use this to estimate at $(t,y^U)$, using the Lipschitz estimate~\eqref{EqMMapLipschitz}:
  \begin{align*}
    |\pi_S f_\sharp\sigma-\pi_S f_\sharp\sigma'| &\leq |\pi_S f\sigma g-\pi_S f\sigma' g|+|\pi_S f\sigma' g-\pi_S f\sigma' g'| \\
      &\leq \bigl(\|\bar N^s_p\bar f\|+\eps C_r+\wt C\rho(t+1)\bigr)|\sigma-\sigma'| + C_\Sigma\rho(t+1)\cdot o(1)|\sigma-\sigma'| \\
      &\leq \theta|\sigma-\sigma'|
  \end{align*}
  where we fix $\theta$ such that $\sup_{p\in\Gamma}\|\bar N^s_p\bar f\|<\theta<1$, and use $\eps=\eps_1$ with $\eps_1>0$ sufficiently small. By the contraction mapping principle, this implies the existence of an $f_\sharp$-invariant section
  \[
    \sigma_f \in \Sigma(\eps_1),\quad f_\sharp\sigma_f=\sigma_f.
  \]
  Note that $\sigma_f$ automatically takes values in $\cS'\oplus 0$ in view of the definition~\eqref{EqMMapTrivExt}, hence the image of $\sigma_f$,
  \[
    \Gamma^u := \sigma_f(\Gamma^u_0(\eps_1)),
  \]
  is a Lipschitz submanifold of $\cM$ which is $\rho$-close to $\Gamma^u_0$.

  The uniqueness claim~\eqref{ItMMapUUniq} is an immediate consequence of our construction and the contraction mapping principle.
  
  \textit{\underline{Step 4:} pointwise differentiability.} We next improve the Lipschitz regularity of $\Gamma^u$ to the pointwise existence of tangent planes. We again argue using Lipschitz jets. Thus, consider $\eps\in(0,\eps_1)$ such that $\Gamma^u_0(\eps_1)\supset\pi^U_0 f\sigma_f\Gamma^u_0(\eps)$, and equip $\Gamma^u_0(\eps_1)$ with the discrete topology. Define the fiber bundle $\cD^b\to\Gamma^u_0(\eps_1)$ whose fiber at $(t,x^U)$ is
  \begin{equation}
  \label{EqMMapJetBdl}
    \cD^b_{(t,x^U)} = \bigl\{ J_{(t,x^U)}\sigma \in J^b\bigl(\Gamma^u_0(\eps_1),(t,x^U);\cS,\sigma_f(t,x^U)\bigr) \colon \sigma\in\Sigma(\eps_1) \bigr\}.
  \end{equation}
  This is a subbundle of the vector bundle $J^b$ whose fiber over $(t,x^U)$ is equal to all of $J^b(\Gamma_0^u(\eps_1),(t,x^U);\cS,\sigma_f(t,x^U))$. The map $f$ induces a natural bundle map $J f\colon\cD^b|_{\Gamma^u_0(\eps)}\to\cD^b$ covering the map $(t,x^U)\mapsto (t-1,x^U_1)=\pi^U_0 f\sigma_f(t,x^U)$, defined by
  \[
    J f(J_{(t,x^U)}\sigma) := J_{(t-1,x^U_1)}(f_\sharp\sigma) = J_{(t-1,x^U_1)}(f\sigma g),
  \]
  where $g$ is the right inverse, defined near $(t-1,x^U_1)$, of $\pi^U_0 f\sigma$ constructed above. Note that since $\sigma$ is tangent to $\sigma_f$ at $(t,x^U)$, we have $g(t-1,x^U_1)=(t,x^U)$, so the Lipschitz jet is well-defined. The membership in $\cD^b_{(t-1,x^U_1)}$ follows from the estimates~\eqref{EqMMapBounded} and \eqref{EqMMapLipschitz}.

  We contend that $J f$ is a fiber contraction. Let $\sigma,\sigma'\in\Sigma(\eps_1)$ denote two local sections near $(t,x^U)$ with $\sigma(t,x^U)=\sigma'(t,x^U)=\sigma_f(t,x^U)$, and let $g,g'$ denote the local right inverses of $\pi^U_0 f\sigma$, $\pi^U_0 f\sigma'$, defined near $(t-1,x^U_1)$. Then in local coordinates as above, and with $\breve\sigma=\pi_S\sigma$, $\breve\sigma'=\pi_S\sigma'$, we estimate using~\eqref{EqMMapInvEst}:
  \begin{equation}
  \label{EqMMapLipContr}
  \begin{split}
    &L_{(t-1,x^U_1)}(\pi_S f\sigma g-\pi_S f\sigma' g') \\
    &\quad\leq L_{(t-1,x^U_1)}(\pi_S f\sigma g-\pi_S f\sigma' g) + L_{(t-1,x^U_1)}(\pi_S f\sigma' g-\pi_S f\sigma'g') \\
    &\quad\leq L_{\sigma_f(t,x^U)}(\pi_S f)\bigl(L_{(t,x^U)}(\breve\sigma-\breve\sigma')L_{(t-1,x^U_1)}(g) + L_{(t,x^U)}(\breve\sigma')L_{(t-1,x^U_1)}(g-g')\bigr) \\
    &\quad\leq (\|\bar N^s_{\bar f(x^U)}\bar f\|+o(1))\bigl(L_{(t,x^U)}(\breve\sigma-\breve\sigma')(\max(\ul\mu^{-1},1)+o(1)) + o(1)\cdot o(1)L_{(t,x^U)}(\breve\sigma-\breve\sigma')\bigr)
  \end{split}
  \end{equation}
  as $\eps\to 0$; the estimate on $L_{(t-1,x^U_1)}(g-g')$ follows similarly to the estimates~\eqref{EqMMapDiffG1}--\eqref{EqMMapDiffG2}. By the normal hyperbolicity of $\bar f$, we have $\|\bar N^s_{f(x^U)}\bar f\|<\min(1,\ul\mu)$, hence this proves our contention. Consider then the map $J_\sharp f$ on sections of $\cD^b|_{\Gamma^u_0(\eps)}$, defined by mapping a section $\{J_{(t,x^U)}\sigma^{(t,x^U)}\}$ to $(t,x^U)\mapsto J f(J_{g(t,x^U)}\sigma^{g(t,x^U)})$; this is the analogue of the map $f_\sharp$ above.\footnote{We could have worked with $J_\sharp f$ from the beginning of this step of the proof; however, the present arguments are more easily extended to the inductive argument for proving higher regularity below.} This is a contraction and hence has a unique fixed point, the section $\sigma_{J f}$ of $\cD^b|_{\Gamma^u_0(\eps)}$.

  Since $J f$ (meaning: its composition with restriction of sections of $\cD^b$ to $\Gamma^u_0(\eps)$) has the invariant bounded, but a priori discontinuous, section $J\sigma_f$, we have
  \[
    J\sigma_f = \sigma_{J f}.
  \]
  On the other hand, $J f$ maps \emph{differentiable} jets into differentiable jets since $f$ is differentiable, and $g$ is differentiable when $\sigma$ is. That is, $J f$ preserves the closed subbundle $\cD^d$, defined like~\eqref{EqMMapJetBdl} but using $J^d$ and differentiable sections $\sigma$; the unique bounded $J f$-invariant section thus necessarily lies in $J^d$, which proves the pointwise differentiability of $\sigma_f$. (Strictly speaking, this only holds over $\Gamma^u_0(\eps)$ rather than $\Gamma^u_0(\eps_1)$; but the image $\sigma_f(\Gamma^u_0(\eps_1))$ is obtained from $\sigma_f(\Gamma^u_0(\eps))$ by repeated application of $f$, hence the former inherits the regularity of the latter.)
  
  \textit{\underline{Step 5:} $\rho\cC_b^1$-regularity.} We improve the pointwise differentiability to continuous differentiability: we shall show that $\sigma_f\in\rho\cC_b^1(\Gamma^u_0;S)$, using arguments similar to those employed in the proof of Theorem~\ref{ThmMInv}. To wit, consider the bundle $L\to\Gamma^u_0$ (with the base equipped with its standard topology coming from $\cM$) with fibers equal to spaces of linear maps,
  \[
    L_{(t,x^U)} = \cL(T_{(t,x^U)}\Gamma^u_0, \cS_{(t,x^U)}).
  \]
  Consider then the smooth disc subbundle $\cB\to\Gamma^u_0(\eps)$ with fibers
  \[
    \cB_{(t,x^U)} = \{ P \in L_{(t,x^U)} \colon \|P\| \leq C_\Sigma\rho(t) \}.
  \]
  Since $\cS\to\Gamma^u_0$ is trivial, we can identify $T_{(t,x^U,s)}\cS=T_{(t,x^U)}\Gamma^u_0\oplus S$. Define then the bundle map $L f$ on $\cB$ as mapping $L_{(t,x^U)}\ni P\mapsto L f(P)\in L_{\pi^U_0 f\sigma_f(t,x^U)}$ such that
  \begin{equation}
  \label{EqMMapLf}
    {\rm graph}(L f(P)) = (D_{\sigma_f(t,x^U)}f)({\rm graph}(P)),
  \end{equation}
  where on the right,
  \[
    {\rm graph}(P) := \{ \xi + P\xi \colon \xi\in T_{(t,x^U)}\Gamma^u_0 \} \subset T_{(t,x^U,\breve\sigma_f(t,x^U))}\cS,
  \]
  similarly on the left.
  
  By the same estimates as in the previous step, $L f$ is a well-defined fiber contraction. One invariant section of $L f$ is given by the (a priori discontinuous) tangent bundle of $\Gamma^u$, i.e.\ the pointwise derivative $(t,x^U)\mapsto D_{(t,x^U)}\sigma_f$. On the other hand, $L f$ is continuous since $\sigma_f$ is, hence it preserves the space of continuous sections (which is $L^\infty$-closed in the space of sections of $\cB$, which are bounded by definition). Thus, $L f$ has a continuous invariant section, which is the unique bounded section, and therefore must agree with $D\sigma_f$. This gives
  \[
    \sigma_f \in \rho\cC_b^1.
  \]

  \textit{\underline{Step 6:} $\rho\cC_b^r$-regularity.} To prove higher regularity, we proceed inductively. Thus, assume that $\sigma_f\in\rho\cC_b^{r-1}$, $r\geq 2$. Consider again the map $L f$ in~\eqref{EqMMapLf}, acting on the $\CI$ fiber bundle $\cB\to\Gamma^u_0$. In view of assumption~\eqref{ItMPertClose},\footnote{The $\rho\cC_b^r$-regularity version is sufficient at this point.} the map $L f$ is $\rho\cC_b^{r-1}$-close to $L f_0$, where the latter is defined by~\eqref{EqMMapLf} with $f$ replaced by $f_0$, and $\sigma_{f_0}(t,x^U)=(t,x^U,0)$. Indeed, the image of $\sigma_{f_0}$ is the unstable manifold $\Gamma^u_0$ for $f_0$; note similarly that $D\sigma_{f_0}$ is the invariant section of $L f_0$, and is smooth and stationary, i.e.\ $t$-invariant. Furthermore, $L f_0$ is covered by the base map $\pi^U_0 f_0\sigma_{f_0}$, for whose inverse $g_{f_0}$ the Lipschitz constant of $x^U\mapsto\bar\pi^U g_{f_0}(t=0,x^U)$ is $m(\Gamma_p\bar f)^{-1}+o(1)$, $\eps\to 0$, in an $\eps$-neighborhood of $\bar f(p)$ by (a simplified version of)~\eqref{EqMMapInvEst}. On the other hand, $L_0 f$ contracts fibers by the factor $\sup_{p\in\Gamma}\|\bar N_p^s\bar f\|\max(m(\Gamma_p\bar f)^{-1},1)+o(1)$ by~\eqref{EqMMapLipContr}. In view of the $r$-normal hyperbolicity of $\bar f$, $L f$ therefore satisfies the $\rho\cC_b^{r-1}$-assumptions of Theorem~\ref{ThmMInv}. We conclude that $D\sigma_f\in\rho\cC_b^{r-1}$, hence $\sigma_f\in\rho\cC_b^r$, finishing the inductive step.

   This finishes the proof of claims~\eqref{ItMMapUInv}--\eqref{ItMMapUUniq} of the theorem.

  \textit{\textbf{\underline{Part 2:}} proof of claim~\eqref{ItMMapSExt} of the theorem.} This proceeds along somewhat similar lines, with the roles of $\Gamma^u_0$, $\bar N^s$, $f$ now being played by $\Gamma^s_0$, $\bar N^u$, $f^{-1}$. Note here that the inverse map $f^{-1}$ satisfies analogous conditions to $f$, namely, it is $\rho\cC_b^r$-close to $f_0^{-1}$ for all $r$, though now $\pi_T f^{-1}(t,x)=t+1$ steps forward in time.
  
  The main difference to the unstable manifold theorem is that there is no unique $f$-invariant spacetime extension $\Gamma^s$ of the stable manifold $\bar\Gamma^s$ at $\Gamma$. We construct one possible $\Gamma^s$ as follows. For $I\subset\R$, let $\Gamma_0^s(I):=\Gamma_0^s\cap t^{-1}(I)$; let further $t_0(\eps)$ denote the value of $t$ for which $\rho(t)=\eps$. For $\eps$ small, define the $\CI$ submanifolds
  \begin{align*}
    \Gamma^s(t_0(\eps),t_0(\eps)+\tfrac12)&:=\Gamma^s_0(t_0(\eps),t_0(\eps)+\tfrac12), \\
    \Gamma^s(t_0(\eps)+1,t_0(\eps)+\tfrac32)&:=f^{-1}\bigl(\Gamma^s(t_0(\eps),t_0(\eps)+\tfrac12)\bigr),
  \end{align*}
  which are graphs of smooth local sections (uniformly bounded in $\rho\cC_b^r$ for all $r$ as $\eps\to 0$) of the vector bundle $U$, which we define to be an extension of $\bar N^u$ from $\Gamma$ to $\bar\Gamma^s$ and then, by stationarity, to $\Gamma^s_0$; if necessary, we take the direct sum with another stationary vector bundle to obtain a trivial vector bundle, as done above for $\cS$. We shall henceforth assume that $U$ is trivial. Define $\Gamma^s_\bullet$ to be the graph of a local section $\sigma_\bullet$ of $U$ over $\Gamma^s_0((t_0(\eps),t_0(\eps)+\tfrac32))$ such that it is equal to $\Gamma^s(t_0(\eps),t_0(\eps)+\tfrac12)$ and $\Gamma^s(t_0(\eps)+1,t_0(\eps)+\tfrac32)$ over the respective slabs in $\Gamma^s_0$; we may arrange that $\sigma_\bullet$ has $\rho\cC_b^r$ bounds for all $r$, with the bound depending on the $\rho\cC_b^r$ bounds on the perturbation of $f_0$ in assumption~\eqref{ItMPertClose}, and that moreover the $\rho\cC_b^r$ norms of $\sigma_\bullet$ are uniformly bounded as $\eps\to 0$. See Figure~\ref{FigMMapSExt}.

  \begin{figure}[!ht]
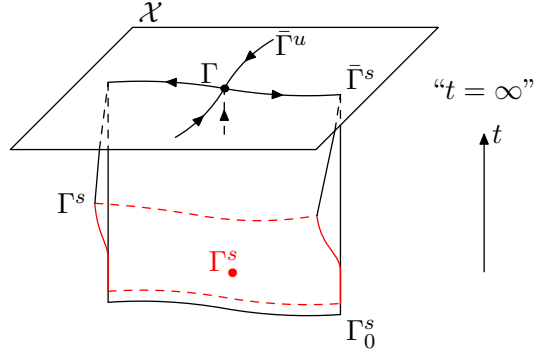

    \centering
    \inclfig{FigMMapSExt}
    \caption{Construction of a spacetime extension of the stable manifold: starting from a prescribed part $\Gamma^s_\bullet$, we iteratively apply $f^{-1}$ to obtain $\Gamma^s$. The arrows at ``$t=\infty$'' indicate the behavior of the stationary model $f_0^{-1}$.}
    \label{FigMMapSExt}
  \end{figure}

  We then define
  \[
    \Gamma^s := \bigcup_{n\in\N_0} f^{-n}\Gamma^s_\bullet,
  \]
  where $f^{-n}=f^{-1}\circ\ldots\circ f^{-1}$ ($n$-fold composition); by construction of $\Gamma^s_\bullet$, this is the graph of a smooth section $\sigma\colon\Gamma^s_0\to U$ defined in $t>t_0(\eps)$. Our goal is to show that $\sigma\in\rho\CI_b$. It is reasonable to expect this to be the case: the map $\bar f^{-1}$ expands in the direction of $\bar\Gamma^s$ and contracts in the fiber directions, thus `flattening out' any reasonable initial piece $\Gamma^s_\bullet$ to $\bar\Gamma^s$ as $t\to\infty$. To prove this rigorously, we shall describe $\Gamma^s$ more indirectly in a way analogous to our construction of $\Gamma^u$.
  
  Denote by $\bar\Gamma^s(\eps)$ an $\eps$-neighborhood of $\Gamma$ inside of $\cX$; with $T(\eps)$ as in~\eqref{EqMMapTint}, put then $\Gamma^s_0(\eps)=T(\eps)\times\bar\Gamma^s(\eps)$. Let also $U(\eps)$ denote the ball bundle of radius $\eps$ inside of $U\to\Gamma^s_0$. Write points on $\Gamma^s_0$ as $(t,x^S)$, $x^S\in\bar\Gamma^s$. Consider then the space of sections
  \begin{align*}
    \Sigma(\eps) &= \bigl\{ \text{sections}\ \sigma\colon\Gamma^s_0(\eps)\to U(\eps) \colon \\
      &\qquad\quad \sigma|_{\Gamma^s_0(\eps)\cap\Gamma_0^s\bigl((t_0(\eps),t_0(\eps)+\tfrac32)\bigr)}=\sigma_\bullet,\ |\sigma|\leq C_\Sigma\rho(t),\ L_{(t,x^S)}(\breve\sigma)\leq C_\Sigma\rho(t) \bigr\}
  \end{align*}
  which initially agree with $\Gamma^s_\bullet$ (we omit the $\eps$-dependence of the latter from the notation); we write $\breve\sigma=\pi_U\sigma$, where $\pi_U$ projects from $U$ onto the typical fiber of the trivial vector bundle $U$. Here, $C_\Sigma$ will be fixed later; it is in particular chosen to be larger than a constant $\wt C$, depending only on the perturbation of $f_0$, so that $\sigma_\bullet$ satisfies the $L^\infty$ and Lipschitz bounds with constant $\wt C$. When $\eps$ is sufficiently small, we can define a map $f^\sharp\colon\Sigma(\eps)\to\Sigma(\eps)$ by
  \[
    f^\sharp\sigma(t,x^S):=
      \begin{cases}
        \sigma_\bullet(t,x^S), & t_0(\eps)<t<t_0(\eps)+\tfrac32, \\
        f^{-1}\sigma g^-(t,x^S), & t>t_0(\eps)+1,
      \end{cases}
  \]
  where $g^-$ is a right inverse of $\pi^S_0 f^{-1}\sigma$, with $\pi^S_0\colon U\to\Gamma^s_0$ denoting projection to the base. Note that by definition of $\sigma_\bullet$, the two expressions on the right hand side agree for $t_0(\eps)+1<t<t_0(\eps)+\tfrac32$.
  
  One can now proceed as in the case of $\Gamma^u$ to infer that for small $\eps>0$, $f^\sharp$ is a well-defined contraction on $\Sigma(\eps)$, equipped with the complete $\cC^0$ metric. Indeed, up to additive $o(1)$ errors as $\eps\to 0$, the map $g^-$ has Lipschitz constants $\max(\|\Gamma_p\bar f\|,1)$, while $f^{-1}$ contracts the fibers of $U$ by no less than $m(\bar N^u_p f)^{-1}$. The product of these two quantities, which arises in estimates on $f^\sharp\sigma$ similarly to before, is less than unity by $1$-normal hyperbolicity. Moreover, the mapping properties of $f^\sharp$ rely on the fact that $\bar f$ preserves $\bar\Gamma^s$; see~\eqref{EqMMapLinfty} for the analogous estimate for $\Gamma^u$.
  
  Letting $\sigma^s$ denote the unique fixed point of $f^\sharp$, the graph of $\sigma^s$ is by definition equal to $\Gamma^s\cap U(\eps)$. We then proceed as before, improve the regularity of $\sigma^s$ to pointwise differentiability, then to $\rho\cC_b^1$, and then inductively, using Theorem~\ref{ThmMInv}, to $\sigma^s\in\rho\cC_b^r$ under $\rho\cC_b^r$-assumptions on the perturbation of $f_0$. This finishes the proof of part~\eqref{ItMMapSExt}, and thus of the theorem.
\end{proof}

This theorem remains true if we relax assumption~\eqref{ItMNormHyp} to eventual relative $r$-normal hyperbolicity as in \cite[Definition~3]{HirschPughShubInvariantManifolds}, that is: for some constants $0<\mu<1<\lambda<\infty$ and $C>1$, we have
\begin{equation}
\label{EqMMapEvRel}
  m(\bar N^u_p\bar f^n) \geq C^{-1}\lambda^n\|\Gamma_p\bar f^n\|^k, \quad
  \|\bar N^s_p\bar f^n\| \leq C\mu^n m(\Gamma_p\bar f^n)^k,\qquad \forall\,p\in\Gamma,\ 0\leq k\leq r,\ n\geq 0.
\end{equation}
Indeed, for all $n\geq N_0$, with $N_0$ sufficiently large, the map $\bar f^n$ is then immediately relatively $r$-normally hyperbolic, and $f^n$ satisfies the assumptions of Theorem~\ref{ThmMMap} upon rescaling time by a factor of $n$. Then, denoting by $\Gamma^u$ the corresponding unstable manifold, note that
\begin{equation}
\label{EqMMapEvRelInv}
  f^n(f(\Gamma^u))=f(f^n(\Gamma^u))=f(\Gamma^u),
\end{equation}
so $f(\Gamma^u)$ is another $\rho\cC_b^1$-perturbation of $\Gamma^u_0$ which is invariant under $f^n$ and approaches $f_0(\Gamma^u_0)=\Gamma^u_0$ as $t\to\infty$; by uniqueness, $\Gamma^u$ is therefore $f$-invariant as well. Part~\eqref{ItMMapSExt} of Theorem~\ref{ThmMMap} applies as well to $f^n$; to ensure that $\Gamma^s$ is $f$-invariant, we work with initial sections---which now need to have $t$-size at least $n$---which are $f$-invariant; in the notation of the above proof, we can take this to be $\bigcup_{j=0}^{n-1} f^{-j}\Gamma_\bullet^s$.

We end this section by noting that other notions of regularity carry over from the perturbation to the invariant manifolds. We only explicitly state one case of interest. Namely, working in $t\geq 1$, let $\cC_{b,\bop}^0\equiv\cC_b^0$, and define the space $\cC_{b,\bop}^r$ exactly like $\cC^r_b$ in~\eqref{EqMCBdd}, but in addition allowing any number of the $V_i$ to be equal to $t\pa_t$.\footnote{The reason for the notation is that the vector fields on $\cX$ together with $t\pa_t$ span the space of b-vector fields on the compactification~\eqref{EqICpt} of $\cM$ at future infinity.} We strengthen the assumption~\eqref{EqMWeightCI} correspondingly by requiring
\[
  |(t\pa_t)^k\rho|\leq C_k\rho\quad \forall\,k\in\N_0.
\]
This is satisfied for polynomial weights $\rho(t)=\la t\ra^{-\alpha}$, $\alpha>0$.

\begin{thm}
\label{ThmMMapb}
  Under the assumptions~\eqref{ItMGamma}--\eqref{ItMNormHyp} and \eqref{ItMPertTime}, and assuming in assumption~\eqref{ItMPertClose} that $\wt V\in\rho\cC_{b,\bop}^\infty$, the unstable manifold $\Gamma^u\subset\cM$ is, resp.\ the stable manifold $\Gamma^s$ can be chosen to be, the graph over $\Gamma^u_0$, resp.\ $\Gamma^s_0$, of a function in $\rho\cC_{b,\bop}^\infty$. This remains true upon replacing $\cC_{b,\bop}^\infty$ by $\cC_{b,\bop}^r$ for $r\in\N$ fixed.
\end{thm}
\begin{proof}
  Introducing the new time coordinate $\tilde t:=\log t$, we have $\pa_{\tilde t}=t\pa_t$, and we need to prove $\rho\CI_b$-regularity of the (un)stable manifold over $\R_{\tilde t}\times\cX$. But this follows by a minor adaptation of the proof of Theorem~\ref{ThmMMap}. Indeed, the only fact about the $t$-behavior of $f$ and $f_0$ used in the proof is that $\pi_T f(\tilde t,x)=\tilde t-\delta(\tilde t)=\pi_T f_0(\tilde t,x)$ change $\tilde t$ by the same amount for every fixed $\tilde t$; for the Lipschitz estimate~\eqref{EqMMapInvEst}, it suffices to assume $(\tilde t-\delta(\tilde t))-(\tilde t'-\delta(\tilde t'))\geq(1-o(1))(\tilde t-\tilde t')$ as $\tilde t,\tilde t'\to\infty$. Note that $t\mapsto t-1$ means $\tilde t\mapsto\log(e^{\tilde t}-1)=\tilde t-\delta(\tilde t)$ with $\delta(\tilde t)=\log(1+\tfrac{1}{e^{\tilde t}-1})$, which fits this assumption.
\end{proof}

\subsection{Stable/unstable manifold theorem for flows}
\label{SsMF}

For the analogous theorem for flows on the closed manifold $\cX$ and the spacetime $\cM=\R_t\times\cX$, we make the following assumptions:

\begin{enumerate}[label=(I$_{\rm F}$.\arabic*),ref=I$_{\rm F}$.\arabic*]
\item\label{ItMFMfd} $\Gamma$ is a closed $\CI$ submanifold of $\cX$;
\item\label{ItMFVect} $\bar V\in\cV(\cX)$ is a smooth vector field tangent to $\Gamma$;
\item\label{ItMFTimeOne} for all $r\in\N$, the time one flow $\bar f:=e^{\bar V}$ is immediately relatively $r$-normally hyperbolic at $\Gamma$, see~\eqref{ItMSplit}--\eqref{ItMNormHyp}, or in fact just eventually relatively $r$-normally hyperbolic in the sense of~\eqref{EqMMapEvRel}.
\end{enumerate}

Denote by $V_0:=-\pa_t+\bar V\in\cV(\cM)$ the spacetime extension which moves at speed $1$ in the $t$-direction. We consider the following class of perturbations $V\in\cV(\cM)$:
\begin{enumerate}[label=(II$_{\rm F}$.\arabic*),ref=II$_{\rm F}$.\arabic*]
\item\label{ItMFTime} we have $V t=V_0 t=-1$;
\item\label{ItMFPert} $V$ is a $\rho\cC_b^r$-perturbation of $V_0$ for every $r$. That is,
  \[
    \wt V := V-V_0 \in \rho\CI_b(\cM;T\cX).
  \]
\end{enumerate}

\begin{thm}
\label{ThmMFlow}
  Under these assumptions, and using the notation used in the statement of Theorem~\usref{ThmMMap}, in particular~\eqref{EqMMapSpExt}, there exists a submanifold $\Gamma^u\subset M$ such that:
  \begin{enumerate}
  \item\label{ItMFlowTgt} $V$ is tangent to $\Gamma^u$;
  \item $\Gamma^u$ approaches $\Gamma^u_0$ as $t\to\infty$ in a $\rho\CI_b$ sense;
  \item $\Gamma^u$ is the unique manifold satisfying~\eqref{ItMFlowTgt} within the class of manifolds approaching $\Gamma^u_0$ in a $\rho\cC_b^1$ sense.
  \ctrsave
  \end{enumerate}
  Furthermore:
  \begin{enumerate}
  \ctrset
  \item there exists a (non-unique) manifold $\Gamma^s\subset\cM$ approaching $\Gamma^s_0$ as $t\to\infty$ in a $\rho\CI_b$ sense and such that $V$ is tangent to $\Gamma^s$.
  \end{enumerate}
  In these statements, the regularity of $\Gamma^u$ and $\Gamma^s$ is $\rho\cC_b^r$ if we relax assumption~\eqref{ItMFPert} by only assuming $\rho\cC_b^r$-regularity of $\wt V$. Moreover, if $\wt V\in\rho\cC_{b,\bop}^\infty$, then $\Gamma^u$ is, resp.\ $\Gamma^s$ can be chosen to be, a $\rho\cC_{b,\bop}^\infty$-graph over $\Gamma^u_0$, resp.\ $\Gamma^s_0$. This remains true upon replacing $\cC_{b,\bop}^\infty$ by $\cC_{b,\bop}^r$ for $r\in\N$ fixed.
\end{thm}
\begin{proof}
  The time one flow, denoted $f$, of $V$ satisfies the assumptions of Theorem~\ref{ThmMMap} (in the relaxed form using eventual hyperbolicity as around~\eqref{EqMMapEvRel} under the assumption on eventual relative normal hyperbolicity) and, for the final part, those of Theorem~\ref{ThmMMapb}. An argument analogous to~\eqref{EqMMapEvRelInv} shows that the time $t$ flow $e^{t V}$ preserves $\Gamma^u$ for all $t$, in particular as $t\to 0$, so $V$ is tangent to $\Gamma^u$. For the construction of $\Gamma^s$, and using the notation of the proof of Theorem~\ref{ThmMMap}, we need to define $\Gamma^s_\bullet$ slightly more carefully; for instance, we can take
  \[
    \Gamma^s_\bullet = \bigcup_{t=0}^T e^{-t V}\Gamma_0^s(\{t_0(\eps)\})
  \]
  for large and fixed $T$ depending on the constants in the eventual hyperbolicity of $f$.
\end{proof}

\section{Microlocal estimates at normally hyperbolic trapping}
\label{SE}

Our microlocal estimates for PDEs with Hamilton flow having a normally hyperbolic trapped set require certain non-degeneracy assumptions which go beyond the purely dynamical results of \S\ref{SM}. Thus, for clarity and simplicity, rather than indirectly controlling the (un)stable manifold via dynamical assumptions, we state the assumptions on them as well as the non-degeneracy requirements directly. This section can therefore be read independently of the previous one.

We first sketch the proof of microlocal estimates in the technically simpler setting of closed manifolds in~\S\ref{SsES}. A detailed proof in the setting of asymptotically stationary spacetimes, with microlocal analysis taking place at future infinity, is given in~\S\ref{SsEC}.

\subsection{Estimates on closed manifolds}
\label{SsES}

We first consider microlocal estimates at normally hyperbolic trapping on closed manifolds without boundary and for the simplest class of operators, in order to present the main ideas without the technical complications caused by microlocal analysis on manifolds with boundary.

Thus, let $X$ denote a closed manifold with a fixed volume density, and let $P\in\Psi^m(X)$ be a self-adjoint classical pseudodifferential operator of order $m\in\R$, with (real) principal symbol $p$. The Hamilton vector field $H_p\in\cV(T^*X)$ is homogeneous of degree $m-1$; if $\wh\rho\in\CI(T^*X\setminus o)$ denotes a positive function which is homogeneous of degree $-1$, the rescaled vector field
\[
  \sfH_p := \wh\rho^{m-1}H_p \in \cV(S^*X)
\]
induces a vector field on the cosphere bundle $S^*X=(T^*X\setminus o)/\R_+$. We will occasionally identify subsets of $S^*X$ with their conic extensions into $T^*X\setminus o$. Let
\[
  \Sigma := p^{-1}(0)\setminus o \subset T^*X\setminus o
\]
denote the characteristic set of $P$. Suppose then that $\sfH_p$ is tangent to a closed $\CI$ submanifold $\Gamma\subset\Sigma$, and $d p\neq 0$ in a neighborhood of $\Gamma$, so $\Sigma$ has codimension $1$ there. Assume that $\wh\rho$ is such that
\begin{equation}
\label{EqESReg}
  H_p\wh\rho \equiv 0 \ \ \text{near}\ \ \Gamma;
\end{equation}
this will allow us to shift the regularity in our estimate at will.\footnote{This can always be arranged \emph{locally} along the $H_p$ flow; here however we make this assumption in a full neighborhood of $\Gamma$.} Suppose there are local orientable manifolds $\Gamma^{u/s}\subset\Sigma$ near $\Gamma$ to which $\sfH_p$ is tangent; we assume they are of class $\CI$, orientable, and non-trivial, i.e.\ $\Gamma^{u/s}\supsetneq\Gamma$, and with $\codim_\Sigma\Gamma^{u/s}=1$. Let $\phi^{u/s}\in\CI(S^*X)$ be such that $\phi^{u/s}$ are defining functions for $\Gamma^{u/s}$ within $\Sigma\cap S^*X$;\footnote{The case of higher codimension can easily be treated by using vector-valued $\phi^{u/s}$ here and in the proof below.} we extend them to functions on $T^*X\setminus o$ by homogeneity of degree $0$. We make the hyperbolicity and non-degeneracy assumptions
\begin{gather}
\label{EqESUS}
  \sfH_p\phi^u=-w^u\phi^u,\quad \sfH_p\phi^s=w^s\phi^s,\quad w^{u/s}>0, \\\
\label{EqESPoisson}
  \wh\rho^{-1}H_{\phi^u}\phi^s = \wh\rho^{-1}\{\phi^u,\phi^s\} > 0,
\end{gather}
on $\Gamma$.\footnote{Condition~\eqref{EqESPoisson} is independent of the choice of defining functions, up to an overall sign. It holds in the following important situation: $\Gamma$ can be extended off $\Sigma$ to a conic symplectic submanifold of $T^*X$ (of codimension $2$), see also \cite[\S5.1]{DyatlovResonanceProjectors}. In condition~\eqref{EqESUS}, $w^{u/s}$ can be controlled precisely for carefully chosen $\phi^{u/s}$ if the flow of $\sfH_p$ is absolutely $1$-normally hyperbolic, see \cite[Lemma~5.1]{DyatlovResonanceProjectors} and also \cite[Theorem~2.2]{HirschPughShubInvariantManifolds}.} Note that $H_{\phi^u}\phi^u=0$ and $H_{\phi^u}p=-H_p\phi^u=0$ on $\Gamma^u$; arguing similarly for $H_{\phi^s}$ shows that $H_{\phi^{u/s}}|_{\Gamma^{u/s}}\in T\Gamma^{u/s}$. Assumption~\eqref{EqESPoisson} on the other hand ensures that $H_{\phi^{u/s}}$ is transversal to $\Gamma^{s/u}$ within $\Sigma$.

\begin{thm}
\label{ThmES}
  There exist operators $B_0,B_1,G\in\Psi^0(X)$, with $\WF'(B_0)$, $\WF'(B_1)$, $\WF'(G)$ contained in any fixed neighborhood of $\Gamma$, such that $B_0$ is elliptic at $\Gamma$, while $\WF'(B_1)\cap\Gamma^u=\emptyset$, and $G$ is elliptic near $\Gamma$, such that for $s,N\in\R$ and some $C>0$,
  \begin{equation}
  \label{EqESEst}
    \|B_0 v\|_{H^s} \leq C\bigl(\|B_1 v\|_{H^{s+1}} + \|G P v\|_{H^{s-m+2}} + \|v\|_{H^{-N}}\bigr).
  \end{equation}
  This holds in the strong sense that if all norms on the right are finite, then so is the norm on the left, and the inequality holds.
\end{thm}

The geometric control condition for~\eqref{EqESEst} to hold is that $\WF'(B_0)\subset\Ell(G)$, and all backward null-bicharacteristics from $\WF'(B_0)$ either reach $\Ell(B_1)$ in finite time or tend to $\Gamma$, all while remaining in $\Ell(G)$. See Figure~\ref{FigIEst}.

\begin{rmk}
\label{RmkESExt}
  It suffices to assume that $P$ has real scalar principal symbol, as long as the imaginary part $\wh\rho^{m-1}\sigma(\frac{1}{2 i}(P-P^*))$ of the subprincipal symbol is not too large relative to $w^s$ and $w^u$, see Remark~\ref{RmkESSkew}. One can also let $P$ be a principally scalar operator acting on sections of a vector bundle. We shall consider these more general cases in \S\ref{SsEC} as they are crucial for applications in relativity.
\end{rmk}

\begin{rmk}
\label{RmkESSaddle}
  Theorem~\ref{ThmES} is different from (and in a certain sense more degenerate than) the saddle point radial estimates considered in~\cite[\S2]{VasyMicroKerrdS}, see also \cite[\S{E.5.2}]{DyatlovZworskiBook}. Indeed, in the simplest situation where $\Gamma\subset T^*X\setminus o$ is a single ray, assumption~\eqref{EqESReg} above implies that the Hamilton vector field $H_p$ \emph{vanishes} at $\Gamma\subset T^*X\setminus o$, hence weights in the fiber such as $\wh\rho^{-2 s+m-1}$ cannot give positivity in positive commutator estimates. On the other hand, if in this situation $H_p$ were non-trivially radial at $\Gamma$, i.e.\ $\wh\rho^{-1}\sfH_p\wh\rho\neq 0$ at $\Gamma$, one would have the estimate~\eqref{EqESEst} with the $H^s$, resp.\ $H^{s-m+1}$ norm on $B_1 v$, resp.\ $G P v$ on the right for $s$ below or above (depending on the sign of $\wh\rho^{-1}\sfH_p\wh\rho$) a certain threshold, by considering commutants of the form $\chi(\phi^u)\chi(\phi^s)\wh\rho^{-2 s+m-1}$.
\end{rmk}

\begin{rmk}
\label{RmkESDefect}
  For $v\in\CI(X)$, the estimate~\eqref{EqESEst} can be proved by means of an argument by contradiction using defect measures, as done by Dyatlov~\cite{DyatlovSpectralGaps} in the semiclassical setting. A simple approximation argument proves~\eqref{EqESEst} more generally for all $v\in H^s(X)$ for which the norms on the right are finite, namely, one applies the estimate to smoothed versions $J_\eps v$ of $v$, where $J_\eps=\Op((1+\eps\wh\rho^{-1})^{-2})$ is uniformly bounded in $\Psi^0(X)$ and tends to $I$ in $\Psi^\delta(X)$ for all $\delta\to 0$, and thus strongly on $H^\sigma(X)$ for all $\sigma$, and uses that $\|[G P,J_\eps]v\|_{H^{s-m+2}}\leq C\|v\|_{H^s}$ remains uniformly bounded in view of~\eqref{EqESReg}; one thus obtains the boundedness of $B_0 J_\eps v$ in $H^s$, and thus by a weak compactness argument the membership $B_0 v\in H^s$ with the estimate~\eqref{EqESEst}. What one \emph{cannot} obtain in this fashion is the strong version of Theorem~\ref{ThmES}, which allows us to \emph{conclude} $H^s$-membership of $v$ at $\Gamma$; this is proved by regularizing the (positive commutator) argument itself.
\end{rmk}

\begin{example}
\label{ExESEx}
  We demonstrate that operators $P$ as above exist: take $X$ to be the 3-torus $X=\Sph^1_x\times\Sph^1_y\times\Sph^1_z$, and let
  \[
    P=D_y+(\sin x)D_x.
  \]
  Writing covectors as $\xi\,d x+\eta\,d y+\zeta\,d z$, the principal symbol of $P$ is $p=\eta+(\sin x)\xi$, whose differential is non-vanishing, so $\Sigma=p^{-1}(0)\setminus o$ is smooth. The Hamilton vector field is $H_p=\pa_y+(\sin x)\pa_x-(\cos x)\xi\pa_\xi$, the conic submanifold
  \[
    \Gamma = \{ x = 0,\ \xi = 0 \} \cap \Sigma = \{ (x=0,y,z;\ \xi=0,\eta=0,\zeta) \colon \zeta\neq 0 \}
  \]
  is invariant under the $H_p$-flow, and $\wh\rho=|\zeta|^{-1}$, defined in a conic neighborhood of $\Gamma$, Poisson-commutes with $p$. Let $\phi^u=|\zeta|^{-1}\xi$ and $\phi^s=x$, then the manifolds
  \begin{alignat*}{3}
    \Gamma^u &= (\phi^u)^{-1}(0) \cap \Sigma &\ =&\ \{ (x,y,z;\ 0,0,\zeta) \colon \zeta\neq 0 \}, \\
    \Gamma^s &= (\phi^s)^{-1}(0) \cap \Sigma &\ =&\ \{ (0,y,z;\ \xi,0,\zeta) \colon (\xi,\zeta)\neq(0,0) \}
  \end{alignat*}
  are $H_p$-invariant; moreover, $\wh\rho^{-1}\{\phi^u,\phi^s\}=1$, and~\eqref{EqESUS} holds with $w^{u/s}=1$ at $\Gamma$. Thus, Theorem~\ref{ThmES} applies. (Note that the trapping at $\Gamma$ is not the only delicate structure of $P$: for example, over $x=0$ there is also a radial set at $\{x=0,\eta=0,\zeta=0\}\cap\Sigma=N^*\{x=0\}\setminus o$, which however is disjoint from $\Gamma$.)
\end{example}

The key ingredient for the proof is a quantitative propagation estimate for classical operators $L\in\Psi^m$ with real principal symbol $\ell$ and subprincipal part $\ell_1=\wh\rho^{m-1}\sigma(\frac{1}{2 i}(L-L^*))$. Working on $S^*X$, suppose $\sfH_\ell=\wh\rho^{m-1}H_\ell$ is transversal to a hypersurface $Z$, and let $K\subset Z$ be compact. For intervals $I\subset\R$, we write, schematically, $\|u\|_{H^s(I)}$ for the ``$H^s$ norm of $u$ on $\{\exp(x\sfH_\ell)z\colon z\in K,\ x\in I\}$'', by which we mean, roughly, the $H^s$ norm of the microlocalization of $u$ to a small neighborhood of this set; this will be made rigorous later. Then for $T_1,T_2>0$ and $\delta_0>0$, and assuming $\ell_1\leq 0$, we have (up to numerical constants independent of $T_1,T_2,\delta_0$, and using a slight enlargement of $K$ in the norms on the right)
\begin{equation}
\label{EqESSkProp}
\begin{split}
  \|v\|_{H^s([0,T_2])} &\leq \sqrt{\frac{T_2}{T_1}}\|v\|_{H^s([-T_1,0])} + (T_2+\sqrt{T_1 T_2})\|L v\|_{H^{s-m+1}([-T_1,T_2+\delta_0])} \\
  &\qquad + C\|v\|_{H^{s-1/2}([-T_1,T_2+\delta_0])}.
\end{split}
\end{equation}
For general $\ell_1$, one can apply this to $L_\gamma=e^{-\gamma s}L e^{\gamma s}$ and $v_\gamma=e^{-\gamma s}v$ for $\gamma\geq\max\{0,\sup\ell_1\}$---thus $L_\gamma$ has non-positive subprincipal part---and obtain the same estimate but with an overall factor of $e^{\gamma(T_1+T_2)}$ on the right. Thus, if $L$ is symmetric, we get a quantitative propagation estimate depending on the lengths of the control and conclusion regions; if $L$ has a subprincipal part, the prototypical example being $L=D_x+i\ell_1$ on $\R_x$, $\ell_1\in\R$, which in particular annihilates $v=e^{\ell_1 x}$, the constants scale accordingly. We sketch the proof of~\eqref{EqESSkProp} at the end of this section.

We shall only sketch the main steps of the proof of Theorem~\ref{ThmES}, leaving the details to the reader; we give a detailed proof in the spacetime setting, which is of primary interest, in \S\ref{SsEC}. We work in a fixed neighborhood $|\phi^u|<\eps_0$ of $\Gamma^u$; by $\|\cdot\|_{H^s(I)}$, we shall mean the $H^s$ norm on the subdomain where $\phi^s\in I$. We shall also drop localizers to the characteristic set, as well as irrelevant constants. Recall that in~\eqref{EqESEst} we are effectively assuming a priori $H^s$-control on $v$ in a punctured neighborhood of $\Gamma^u$.
\begin{enumerate}
\item\label{ItESSkEq} Let $\Phi^u$ denote a quantization of $\phi^u$. The weak localization $\Phi^u v$ off the unstable manifold satisfies a better equation due to~\eqref{EqESUS},
  \[
    (P-i W^u)\Phi^u v = \Phi^u P v + R v,
  \]
  with $R\in\Psi^{m-2}$, and $W^u=(W^u)^*\in\Psi^{m-1}$ quantizing $\wh\rho^{-m+1}w^u$. A simple commutator argument exploiting the positivity of $w^u$ gives, for fixed small $\eps_0>0$,
  \begin{equation}
  \label{EqESSkEq}
    \|\Phi^u v\|_{H^{s+1}([-\eps_0,\eps_0])} \leq \|P v\|_{H^{s-m+2}} + \|v\|_{H^s([-2\eps_0,2\eps_0])} + \|v\|_{H^{s+1}(\{|\phi^u|>\eps_0\})};
  \end{equation}
  the second term comes from $R$, and the last one is an error term arising from the localization near $\Gamma^u$.
\item\label{ItESSkEqC} The quantitative estimate~\eqref{EqESSkProp} applied to $L=\Phi^u$ allows us to control, for $\delta>0$ to be chosen later,
  \begin{equation}
  \label{EqESSkEqC0}
    \|v\|_{H^s([-2\delta,2\delta])} \leq \delta^{1/2} ( \|v\|_{H^s(\pm[2\delta,\eps_0])} + \|\Phi^u v\|_{H^{s+1}([-\eps_0,\eps_0])} ) + \|v\|_{H^{s-1/2}};
  \end{equation}
  see Figure~\ref{FigESSkEqC}. This uses assumption~\eqref{EqESPoisson}, which implies that $\phi^s$ is comparable to the affine parameter along the integral curves of $\wh\rho^{-1}H_{\phi^u}$.
  \begin{figure}[!ht]
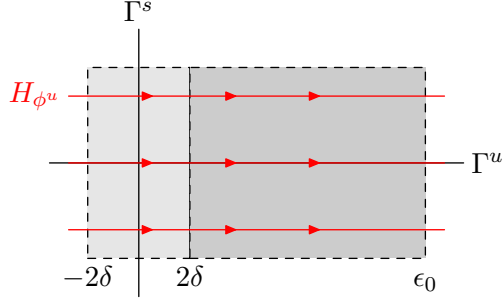

    \centering
    \inclfig{FigESSkEqC}
    \caption{Quantitative propagation estimate for $v$ in terms of $\Phi^u v$: the norm of $v$ in the light gray box is small, $\cO(\delta^{1/2})$, compared to that in the dark gray box due its small size in the $H_{\phi^u}$ direction (red). The labels indicate the $\phi^s$-variable. Also shown is the $H_{\phi^u}$ flow along which we propagate.}
    \label{FigESSkEqC}
  \end{figure}

  For small $\delta>0$, combined with~\eqref{EqESSkEq} and absorbing the piece $\|v\|_{H^s{[-2\delta,2\delta])}}$ of the second term on the right in~\eqref{EqESSkEq} (which gets the small prefactor $\delta^{1/2}$ from~\eqref{EqESSkEqC0}) into the left hand side of~\eqref{EqESSkEqC0}, this controls $v$ close to $\Gamma$ by a \emph{small} constant times $v$ away from $\Gamma$:
  \begin{equation}
  \label{EqESSkEqC}
  \begin{split}
    \|v\|_{H^s([-2\delta,2\delta])} &\leq \delta^{1/2}(\|v\|_{H^s(-[2\delta,2\eps_0])}+\|v\|_{H^s([2\delta,2\eps_0])} + \|P v\|_{H^{s-m+2}}) \\
      &\qquad\qquad + \|v\|_{H^{s+1}(\{|\phi^u|>\eps_0\})} + \|v\|_{H^{s-1/2}}.
  \end{split}
  \end{equation}
  This is analogous to the Lipschitz estimate \cite[Lemma~3.2]{DyatlovSpectralGaps} on the defect measure in the semiclassical setting.
\item\label{ItESSkLog} Using a quantitative propagation estimate~\eqref{EqESSkProp} for $P$, we can conversely estimate $v$ in $\phi^s\in\pm[2\delta,2\eps_0]$ by $v$ in $\pm[\delta,2\delta]$; this requires time $\sim\log\delta^{-1}$ propagation along $\sfH_p$ by the unstable dynamics of $\sfH_p$ within $\Gamma^u$.\footnote{To control $v$ in the fixed $\eps_0$-neighborhood of $\Gamma^u$, we again need to use the a priori control away from $\Gamma^u$, propagated along $\sfH_p$, but the constants for this part of the estimates do not matter.} Thus,
  \begin{equation}
  \label{EqESSkLog}
  \begin{split}
    \|v\|_{H^s(\pm[2\delta,2\eps_0])} &\leq (\log\delta^{-1})^{1/2} \|v\|_{H^s(\pm[\delta,2\delta])} + (\log\delta^{-1})\|P v\|_{H^{s-m+1}} \\
      &\qquad\qquad + \|v\|_{H^s(\{|\phi^u|>\eps_0\})} + \|v\|_{H^{s-1/2}};
  \end{split}
  \end{equation}
  see Figure~\ref{FigESSkEqLog}.
  \begin{figure}[!ht]
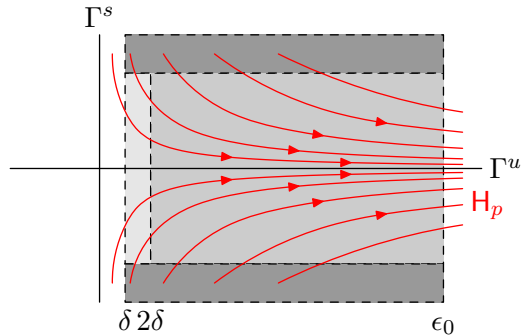

    \centering
    \inclfig{FigESSkEqLog}
    \caption{Control of $v$ for $\phi^s\in[2\delta,\eps_0]$ (medium gray) by $v$ on $\phi^s\in[\delta,2\delta]$ (light gray) and $v$ away from $\Gamma^u$ (dark gray) by time $\sim\log\delta^{-1}$ propagation along $\sfH_p$ (red). There is an analogous, symmetric, picture for $\phi^s<0$.}
    \label{FigESSkEqLog}
  \end{figure}
\item\label{ItESSkFin} Plugging~\eqref{EqESSkLog} into~\eqref{EqESSkEqC}, we can absorb $\|v\|_{H^s(\pm[\delta,2\delta])}$ into the left hand side of~\eqref{EqESSkEqC}, giving unconditional control on $\|v\|_{H^s([-2\delta,2\delta])}$. Plugged back into~\eqref{EqESSkLog}, this gives control on $v$ near $\Gamma^u$, resulting in the estimate~\eqref{EqESEst} with $N=s-\half$. To obtain the better error term, one proceeds iteratively, applying this estimate to a microlocalized version $B_0' v$ of $v$, where $B_0'\in\Psi^0(X)$ is elliptic at $\Gamma$, and take $B_0,B_1$ to have $\WF'(B_0),\WF'(B_1)\subset\Ell(B_0')$, and $\WF'(I-B_0')\cap\WF'(G)=\emptyset$; one can then estimate the error term $\|B_0' v\|_{H^{s-1/2}}$ by the analogous estimate with $s$ reduced by $\half$ and with slightly enlarged elliptic sets of $B_1,G$.
\end{enumerate}

\begin{rmk}
\label{RmkESSkew}
  If $p_1:=\wh\rho^{m-1}\sigma(\frac{1}{2 i}(P-P^*))$ was non-zero, one would get an extra factor $\sim\delta^{-\gamma/w^s}$ on the right in~\eqref{EqESSkLog}, with $\gamma\geq\max\{0,\sup p_1\}$, which can still be absorbed in~\eqref{EqESSkEqC} for $\gamma<\half w^s$. (On the other hand, step~\eqref{ItESSkEq} requires $\gamma<w^u$.)
\end{rmk}

We finally indicate the proof of the model estimate~\eqref{EqESSkProp}; the key tool is G\aa{}rding's inequality, which allows one to translate symbolic bounds into operator bounds. Namely, dropping localizers to $\Sigma$ and to $\Gamma^u$ (the latter can be taken to be $\sfH_p$-invariant, identically one in a small neighborhood of $K$) and choosing coordinates in which $\sfH_p=\pa_x$, we consider commutants $a=\wh\rho^{-2 s+m-1}\wt a$, where we design $\wt a=\wt a(x)$ with support in $[-T_1,T_2+\delta_0]$ such that $0\leq\wt a\leq 1$, and $\sfH_p\wt a\leq\frac{2}{T_1}$ on $[-T_1,0]$ and $\sfH_p\wt a\leq -\frac{1}{2 T_2}$ on $[0,T_2]$. The latter implies $\wt a\leq 2 T_2(-\sfH_p\wt a)$ on $[0,T_2]$, which we arrange on all of $[0,T_2+\delta_0]$. See Figure~\ref{FigESComm}.

\begin{figure}[!ht]
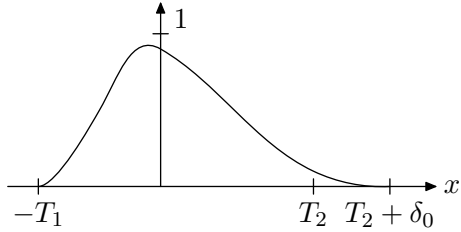

  \centering
  \inclfig{FigESComm}
  \caption{The commutant $\wt a=\wt a(x)$ used in the quantitative propagation estimate; the slope of $\wt a$ is controlled from above on $[-T_1,0]$, and has definite decay on $[0,T_2]$.}
  \label{FigESComm}
\end{figure}

We can arrange that $\wt a=\wt a_1^2+\wt a_2^2$, where $\wt a_2=\sqrt{\wt a}$ on $[0,T_2+\delta_0]$, and $0\leq\wt a_1\leq 1$ is supported in $[-T_1,0)$, and so that for $a_j=\wh\rho^{-s+(m-1)/2}\wt a_j$, $j=1,2$,
\[
  H_p a = -b^2 - \frac{1}{4 T_2}\wh\rho^{-m+1}a_2^2 + e,
\]
where $b$ is non-negative, with lower bound $b\geq\frac{1}{2\sqrt T_2}\wh\rho^{-s}$ on $[0,T_2]$, while $e$ is supported in $[-T_1,0]$ with $|e|\leq\frac{2}{T_1}\wh\rho^{-2 s}$. For $A=A_1^*A_1+A_2^*A_2$, $B$, $E$ denoting quantizations of $a$ (and $a_1,a_2$), $b$, $e$, one gets by means of G\aa{}rding's inequality
\begin{align*}
  &2\Im\la P u,A u\ra = \la i[P,A]u,u\ra \\
  &\qquad\leq -\|B u\|_{L^2}^2 - \frac{1}{4 T_2}\|A_2 u\|_{H^{-(m-1)/2}}^2 + |\la E u,u\ra| + C\|u\|_{H^{s-1/2}}^2 \\
  &\qquad\leq -\frac{1}{4 T_2}\|u\|_{H^s([0,T_2])}^2 - \frac{1}{4 T_2}\|A_2 u\|_{H^{-(m-1)/2}}^2 + \frac{2}{T_1}\|u\|_{H^s([-T_1,0])}^2 + C\|u\|_{H^{s-1/2}}^2.
\end{align*}
The left hand side can be estimated by
\begin{align*}
  2|\la P u,A u\ra| &\geq -2\|A_1 u\|_{H^{-(m-1)/2}}\|A_1 P u\|_{H^{(m-1)/2}}-2\|A_2 u\|_{H^{-(m-1)/2}}\|A_2 P u\|_{H^{(m-1)/2}} \\
    &\geq -\frac{2}{T_1}\|u\|_{H^s([-T_1,0])}^2 - \frac{1}{4 T_2}\|A_2 u\|_{H^{-(m-1)/2}}^2 \\
    &\quad\qquad\qquad - \Bigl(\frac{T_1}{2}+4 T_2\Bigr)\|P u\|_{H^{s-m+1}([-T_1,T_2+\delta_0])}^2.
\end{align*}
Combining the two gives the desired estimate
\[
  \|u\|_{H^s([0,T_2])}^2 \leq \frac{16 T_2}{T_1}\|u\|_{H^s([-T_1,0])}^2 + (2 T_1 T_2+16 T_2^2)\|P u\|^2_{H^{s-m+1}([-T_1,T_2+\delta_0])} + C\|u\|_{H^{s-1/2}}^2.
\]
When making this precise (in particular the $H^s$ norms here) in the cusp setting in~\S\ref{SsEC}, we shall simply keep writing $4 T_2\|B u\|^2$ for the term on the left. The basic estimate then is:
\begin{lemma}
\label{LemmaESBound}
  Let $B,B'\in\Psi^s(X)$ with $\WF'(B')\subset\Ell(B)$, and suppose that their rescaled symbols $b=\wh\rho^s\sigma(B)$, $b'=\wh\rho^s\sigma(B')$ satisfy $|b'|\leq b$ on $\WF'(B')$. Then for any $\delta>0$, there exists a constant $C$ such that
  \[
    \|B' u\|_{H^s} \leq (1+\delta)\|B u\|_{H^s} + C\|u\|_{H^{s-1/2}}.
  \]
\end{lemma}
\begin{proof}
  The principal symbol of $(1+\delta)^2 B^*B-(B')^*B'$ has a smooth real square root $e\in S^s(T^*X)$: near $\WF'(B')$, this follows from $(1+\delta)b>|b'|$, while away from $\WF'(B')$, we simply have $e=(1+\delta)\sigma(B)$. Thus, $(1+\delta)^2 B^*B-(B')^*B'-E^*E\in\Psi^{2 s-1}(X)$. Applying this to $u$ and then pairing with $u$ gives the desired result.\footnote{The sharp G\aa{}rding inequality allows one to take $\delta=0$. On the other hand, for $\delta>0$, one can replace the error term by $C\|u\|_{H^{-N}}$ for any fixed $N\in\R$ by improving the square root construction in the proof: one takes $e$ to be the \emph{full} symbol of $(1+\delta)B$ in a fixed small neighborhood of $\WF'(B')$, hence one only needs to solve away the $\Psi^{2 s-1}(X)$-error term near $\WF'(B')$, where $e$ is positive, so the usual square root construction applies. Neither improvement will be needed in our application, so we settle for the simplest version here.}
\end{proof}

\subsection{Estimates in the spacetime setting}
\label{SsEC}

Let $X$ denote a closed $(n-1)$-dimensional manifold, and let
\[
  M^\circ=\R_t\times X.
\]
We describe the (partial) compactification $M$ of $M^\circ$ at future infinity and the relevant operator algebra and function spaces---(complete) cusp pseudodifferential operators and weighted cusp Sobolev spaces---in~\S\ref{SssECC}, before stating and proving the microlocal estimate at trapped sets lying over $\pa M$ in~\S\ref{SssECP}. While our main applications concern wave equations on Lorentzian spacetimes, our arguments in the present section are entirely \emph{microlocal}, and thus apply to any pseudodifferential operator with suitable null-bicharacteristic flow and subprincipal symbol conditions.

\subsubsection{Compactification and the cusp algebra}
\label{SssECC}

The cusp ps.d.o.\ algebra was introduced by Mazzeo--Melrose \cite{MazzeoMelroseFibred} in more generality than needed here, hence we give a simplified account adapted to present interests. As a further reference, we refer the reader to the lecture notes \cite{VasyMinicourse}. We partially compactify $M^\circ$ by introducing
\begin{equation}
\label{EqECCMfd}
  M = \bigl(M^\circ \sqcup ([0,\infty)_\tau \times X)\bigr) / \sim,\quad (t,x)\sim(\tau=t^{-1},x),\ t>0.
\end{equation}
The manifold $M$ has a natural smooth structure, with functions near $\tau=0$ being smooth if and only if they are smooth functions of $(\tau,x)$. Consider then the Lie algebra of \emph{cusp vector fields},
\[
  \Vcu(M) := \{ V\in\cV(M) \colon V\tau \in \tau^2\CI(M) \}.
\]
If $x^1,\ldots,x^{n-1}$ denotes local coordinates on $X$, $\Vcu(M)$ is generated over $\CI(M)$ by $\tau^2\pa_\tau=-\pa_t$ and $\pa_{x^1},\ldots,\pa_{x^{n-1}}$; hence $\Vcu(M)$ is a uniform version of $\cV(M^\circ)$ as $t\to\infty$, encoding the stationary nature ($t$-invariance) of $M^\circ$.\footnote{In previous works, see in particular \cite[\S3.3]{HintzVasyKdSStability}, we encoded stationarity by compactifying using $\tau'=e^{-t}$ as the defining function of future infinity and the b-structure corresponding to the natural smooth structure generated by $\CI(X)$ and smooth functions of $\tau'$. The relationship between the two structures is $\tau'=e^{-1/\tau}$, $\tau'\pa_{\tau'}=\tau^2\pa_\tau$. Since we are mainly interested in \emph{polynomial} weights in $t$, as appropriate for perturbations of the Kerr spacetime, the cusp setting is more natural; but we could equally well work in the b-setting with logarithmic weights $|\log\tau'|^\alpha$.} This also shows directly that $\Vcu(M)$ is the space of smooth sections of a natural vector bundle $\Tcu M$, called \emph{cusp tangent bundle}, with local frame
\begin{equation}
\label{EqECCTFrame}
  \tau^2\pa_\tau,\pa_{x^1},\ldots,\pa_{x^{n-1}};
\end{equation}
we emphasize that $\tau^2\pa_\tau$ is non-vanishing as a cusp vector field \emph{down to $\tau=0$}.\footnote{For general manifolds $M$, not necessarily arising from our specific construction, $\Tcu M$ is well-defined once one fixes the equivalence class of a defining function $\tau\in\CI(M)$ of $\pa M$ modulo $\tau^2\CI(M)$ (note that the latter space is independent of $\tau$).} The dual bundle $\Tcu^*M$, called \emph{cusp cotangent bundle}, thus has local sections $\frac{d\tau}{\tau^2},d x^1,\ldots,d x^{n-1}$, with $\frac{d\tau}{\tau^2}=-d t$ non-singular as a cusp 1-form down to $\tau=0$. We note that there are natural smooth maps $\Tcu M\to T M$ and $T^*M\to\Tcu^*M$ which are isomorphisms over $M^\circ$, but not over $\pa M$, where they have 1-dimensional kernel and cokernel.

Of particular importance in applications is the symmetric second tensor power $S^2\,\Tcu^*M$. Indeed, stationary metrics are smooth non-degenerate sections of $S^2\,\Tcu^*M$ \emph{down to $\tau=0$}; conversely, smooth non-degenerate sections $g\in\CI(M;S^2\,\Tcu^*M)$ are `asymptotically stationary', in that the components of $g-g|_{\pa M}$ in the frame~\eqref{EqECCTFrame} decay at a polynomial rate as $t\to\infty$. We shall also make use of the cusp density bundle whose sections are of the form $a|\frac{d\tau}{\tau^2} d x^1\cdots d x^{n-1}|=a|d t d x^1\cdots d x^{n-1}|$ with $a\in\CI(M)$; by a \emph{cusp volume density}, we mean a \emph{positive} such density, i.e.\ $a>0$.

Recall the filtered algebra of \emph{cusp differential operators}
\[
  \Diffcu(M) = \bigcup_{m\geq 0}\Diffcu^m(M),
\]
where $\Diffcu^m(M)$ is generated over $\CI(M)$ by up to $m$-fold products of cusp vector fields; for $m=0$, we set $\Diffcu^0(M):=\CI(M)$. If $E,F\to M$ are smooth vector bundles, we define $\Diffcu(M;E,F)$ in local trivializations of $E,F$ as $\rank F\times\rank E$ matrices of elements of $\Diffcu(M)$. Thus, elements of $\Diffcu(M;E,F)$ define continuous linear maps $\CI(M;E)\to\CI(M;F)$ and $\CIdot(M;E)\to\CIdot(M;F)$; here $\CIdot$ consists of smooth functions or sections vanishing to infinite order at $\pa M$.

The (in our setup canonical) cusp vector field $\tau^2\pa_\tau$ induces a natural fiber-linear function\footnote{For general $M$ and with $\tau$ fixed modulo $\tau^2\CI(M)$, this function is well-defined at $\Tcu_{\pa M}^*M$.}
\begin{equation}
\label{EqECCTimeMomentum}
  \sigma(\varpi) := -\varpi(\tau^2\pa_\tau), \quad \varpi\in\Tcu^*M.
\end{equation}
This is the same as the coordinate $\sigma$ in the set of local fiber coordinates on $\Tcu^*M$ defined by writing cusp 1-forms as
\begin{equation}
\label{EqECCCoord}
  -\sigma\tfrac{d\tau}{\tau^2} + \xi_i d x^i = \sigma\,d t+\xi_i d x^i.
\end{equation}
Given an operator
\[
  P = \sum_{j+|\alpha|\leq m} a_{j\alpha}(\tau,x)(-\tau^2 D_\tau)^j D_x^\alpha \in\Diffcu^m(M),\quad a_{j\alpha}\in\CI(M),
\]
we define its principal symbol in the coordinates~\eqref{EqECCCoord} by
\[
  \sigma_\cuop^m(P) := \sum_{j+|\alpha|=m} a_{j\alpha}\sigma^j\xi^\alpha;
\]
we often simply write $\sigma(P)$ for brevity. The symbol has the usual properties
\begin{equation}
\label{EqECCSymb}
  \sigma_\cuop^{m_1+m_2}(P_1\circ P_2) = p_1 p_2, \quad
  \sigma_\cuop^{m_1+m_2-1}(i[P_1,P_2]) = H_{p_1}p_2
\end{equation}
for $P_j\in\Diffcu^{m_j}(M)$, $p_j=\sigma_\cuop^{m_j}(P_j)$, $j=1,2$, where for a function $p\in\CI(\Tcu^*M)$ its Hamilton vector field $H_p$ is given by
\[
  H_p = -(\pa_\sigma p)\tau^2\pa_\tau + (\pa_{\xi_i}p)\pa_{x^i} + (\tau^2\pa_\tau p)\pa_\sigma - (\pa_{x^i}p)\pa_{\xi_i} \in \Vcu(\Tcu^*M).
\]
In particular, $H_p t=\pa_\sigma p$ is a smooth function down to $\tau=0$.

The algebra $\Diffcu(M)$ can be microlocalized, which amounts to allowing as symbols more general functions than polynomials in the fibers. Thus, denote by $S^m(\Tcu^*M)$ the space of symbols of order $m$, i.e.\ functions $a\in\CI(\Tcu^*M)$ satisfying the bounds
\begin{equation}
\label{EqECCCuspSymb}
  |\pa_\tau^j\pa_x^\alpha\pa_\sigma^k\pa_\xi^\beta a(\tau,x,\sigma,\xi)| \leq C_{j\alpha k\beta}(1+|\sigma|+|\xi|)^{m-(k+|\beta|)}
\end{equation}
in local coordinates; the quantization $\Op(a)$ of $a\in S^m$ acts on smooth functions $u\in\CIdot(M)$ supported in a coordinate chart near $\tau=0$ via the usual formula\footnote{The Schwartz kernel of $\Op(a)$ in $(t,x)$-variables is super-polynomially decreasing away from the diagonal, which is sufficient for $\Op(a)$ to act on the polynomially weighted function spaces below. Thus, unlike in the b-setting \cite[\S2]{HintzVasySemilinear}, we do not need to insert a cutoff for $t$ near $t'$.}
\[
  \Op(a)u(t,x) = (2\pi)^{-n}\iiiint e^{i((t-t')\sigma+(x-x')\xi)} a(\tau,x,\sigma,\xi)u(t',x')\,d t'\,d x'\,d\sigma\,d\xi,
\]
For general $u$, define $\Op(a)u$ using a partition of unity. The space $\Psicu^m(M)$ then consists of all operators of the form $\Op(a)$. The principal symbol map extends to $\sigma_\cuop^m\colon\Psicu^m(M)\to S^m/S^{m-1}(\Tcu^*M)$ enjoying the properties~\eqref{EqECCSymb}.

We further define the operator wave front set $\WFcu'(A)\subset\Tcu^*M\setminus o$ of $A=\Op(a)$ as the essential support $\esssupp a$ of the full symbol $a$, which is the complement of all $\varpi\in\Tcu^*M\setminus o$ such that $a$ is of order $-\infty$ in a conic neighborhood of $\varpi$. This is a conic set by definition, hence can be identified with its image in the quotient space
\[
  \Scu^*M := (\Tcu^*M\setminus o) / \R_+,
\]
the cusp cosphere bundle of $M$. It is often convenient to radially compactify $\Tcu^*M$, namely, in a local trivialization $\cU\times\R^n$ of $\Tcu^*M$, we set
\[
  \ol{\Tcu^*_\cU M} := \cU \times \ol{\R^n},
\]
where $\ol{\R^n}$ is the radial compactification
\[
  \ol{\R^n} := \Bigl(\R^n \sqcup \bigl([0,\infty)_\rho\times\Sph^{n-1}\bigr) \Bigr) / \sim,
  \quad
  \R^n\setminus\{0\}=r\omega\sim(\rho=r^{-1},\omega).
\]
We can then identify $\Scu^*M$ with the new boundary `at fiber infinity' of $\ol{\Tcu^*}M$. Thus, functions, resp.\ vector fields, which are homogeneous of degree $0$ are identified with functions, resp.\ induce vector fields, on $S^*M$. (In the case of vector fields, the restriction $\Tb(\ol{\Tcu^*}M)|_{\Scu^*M}\to T(\Scu^*M)$ to $S^*M$ has a non-trivial kernel generated by the fiber-radial vector field.) We shall identify conic subsets of $\Tcu^*M\setminus o$ with their closures in $\ol{\Tcu^*M}\setminus o$ and also with their boundary at fiber at infinity, or equivalently their quotient in $\Scu^*M$.

Working on $M^\circ$, we can define another closely related class of operators, $\Psi_\infty^m(M)$: given a \emph{uniform symbol} $a\in S_\infty^m(T^*M^\circ)$ satisfying
\begin{equation}
\label{EqECCUnifSymb}
  |\pa_t^j\pa_x^\alpha \pa_\sigma^k\pa_\xi^\beta a(t,x,\sigma,\xi)| \leq C_{j\alpha k\beta}(1+|\sigma|+|\xi|)^{m-(k+|\beta|)}
\end{equation}
(the difference to $S^m$ being that we use $\pa_t$, not $\pa_\tau$), we define the quantization $\Op(a)$ of $a$ by the same formula as above; its Schwartz kernel is again super-polynomially decaying off the diagonal. This defines the \emph{uniform algebra} $\Psi_\infty(M)$ which generalizes $\Psicu(M)$ in that the symbols are no longer required to be smooth down to $\tau=0$, so $\Psicu(M)\subset\Psi_\infty(M)$; on the other hand, uniform symbols do not have well-defined limits at $\tau=0$. Since the regularity of $a\in S^m_\infty$ in the base is precisely boundedness with respect to iterative application of cusp vector fields, we shall think of elements of $\Psi_\infty(M)$ as cusp ps.d.o.s with coefficients having infinite cusp regularity. They form a filtered algebra, $\Psi_\infty^{m_1}(M)\circ\Psi_\infty^{m_2}(M)\subset\Psi_\infty^{m_1+m_2}(M)$, and the symbolic properties~\eqref{EqECCSymb} continue to hold in $\tau>0$.

Both classes of operators are invariant under conjugation by polynomial weights in $t$ or equivalently $\tau=t^{-1}$, which allows us to define bi-filtered algebras $\tau^\alpha\Psicu^m(M)$ and $\tau^\alpha\Psi_\infty^m(M)$, $\alpha\in\R$, with both orders additive upon composition of operators. We point out that if for $P\in\tau^\alpha\Psicu^m(M)$ one has $\sigma(P)=0\in\tau^\alpha S^m/\tau^\alpha S^{m-1}(\Tcu^*M)$, then one can only conclude that $P\in\tau^\alpha\Psicu^{m-1}(M)$, likewise for the uniform algebra; thus, the principal symbol captures only the leading order behavior of $P$ in the sense of differential order, but not in the sense of decay at $\tau=0$.

Closely related to uniform symbols~\eqref{EqECCUnifSymb} are cusp-conormal functions
\begin{equation}
\label{EqECCCuspConormal}
  \cA_\cuop(M) := \{ u\in L^\infty(M) \colon A u\in L^\infty(M)\ \forall\,A\in\Diffcu(M) \}, \quad
  \cA_\cuop^\alpha(M) := \tau^\alpha\cA_\cuop(M).
\end{equation}
Indeed, $\cA_\cuop(M)=\cA_\cuop^0(M)\subset S_\infty^0(T^*M^\circ)$. In the notation of~\S\ref{SM}, we have (restricting supports to  $t\geq 1$)
\[
  \cA_\cuop^\alpha(M)=t^{-\alpha}\CI_b(M^\circ).
\]
Standard conormal functions are defined using b- instead of cusp operators, to wit
\begin{equation}
\label{EqECCBConormal}
  \cA_\bop(M) := \{ u\in L^\infty(M) \colon A u\in L^\infty(M)\ \forall\,A\in\Diffb(M) \}, \quad
  \cA_\bop^\alpha(M) := \tau^\alpha\cA_\bop(M).
\end{equation}
Thus, $\cA_\bop^\alpha(M)=t^{-\alpha}\CI_{b,\bop}(M^\circ)\subset\cA_\cuop^\alpha(M)$. Conormal functions valued in vector bundles are defined using local trivializations, or equivalently using connections to define derivatives of sections.

Finally, we define the natural $L^2$-based function spaces for cusp analysis. Being interested in uniform analysis near $\tau=0$, we shall implicitly assume that all functions have support in, say, $\tau\leq 1$ (i.e.\ $t\geq 1$). Fixing any cusp volume density\footnote{All choices lead to the same function spaces with equivalent norms since we are restricting supports to the fixed \emph{compact} subset $\tau\leq 1$ of $M$.} $\nu$, we define the $L^2$-space
\[
  \Hcu^0(M) \equiv L^2_\cuop(M) := L^2(M;\nu),
\]
as well as its weighted analogue $\tau^\alpha L^2_\cuop(M)\equiv\Hcu^{0,\alpha}(M)$, $\alpha\in\R$, consisting of all $u\in L^2_\loc(M)$ for which $\tau^{-\alpha}u\in L^2_\cuop(M)$. For integers $k\geq\N_0$, we define the \emph{weighted cusp Sobolev space}
\[
  \Hcu^{k,\alpha}(M) := \{ u\in \tau^\alpha L^2_\cuop(M) \colon P u\in \tau^\alpha L^2_\cuop(M)\ \forall\,P\in\Diffcu^k(M) \}.
\]
This becomes a Hilbert space with the norm $\|u\|_{\Hcu^{k,\alpha}(M)}^2=\sum_j \|P_j u\|_{\tau^\alpha L^2_\cuop(M)}^2$, where $\{P_j\}\subset\Diffcu^k(M)$ is a finite collection of operators spanning $\Diffcu^k(M)$ over $\CI(M)$. For general $k\in\R$, $\Hcu^{k,\alpha}(M)$ can be defined by duality and interpolation. Any operator $P\in\tau^{-\beta}\Psicu^m(M)$ or $P\in\tau^{-\beta}\Psi_\infty^m(M)$ then defines a continuous linear map
\[
  P \colon \Hcu^{s,\alpha}(M) \to \Hcu^{s-m,\alpha-\beta}(M),\quad s,\alpha\in\R.
\]
Using local trivializations and partitions of unity, one can similarly define $\Psicu^m(M;E,F)$, $\Psi_\infty^m(M;E,F)$ for vector bundles $E,F\to M$, as well as function spaces $\Hcu^{s,\alpha}(M;E)$.

Lastly, we remark that the basic lemma~\ref{LemmaESBound} is proved using a purely symbolic argument, hence applies to cusp ps.d.o.s and the uniform algebra as well.

\subsubsection{Setup and proof of the microlocal estimate}
\label{SssECP}

Define $M$ by~\eqref{EqECCMfd}, and let $E\to M$ denote a smooth vector bundle such as $E=S^2\,\Tcu^*M$. Let
\[
  P\in\Psicu^m(M;E)
\]
denote a principally scalar operator with principal symbol $p=\sigma_\cuop^m(P)\in\CI(\Tcu^*M)$. Its characteristic set is denoted
\[
  \Sigma := p^{-1}(0) \subset \Tcu^*M\setminus o.
\]
We define the stationary model at $\pa M$ as follows: denote the restriction of the Hamilton vector field $H_p\in\Vcu(\Tcu^*M)$ to $\Tcu_{\pa M}^*M$ \emph{as a cusp vector field} by $H_{p_0}$; in local coordinates, this simply means
\[
  H_{p_0} = -(\pa_\sigma p_0)\tau^2\pa_\tau + (\pa_{\xi_i}p_0)\pa_{x^i} - (\pa_{x^i}p_0)\pa_{\xi_i},\quad p_0:=p|_{\tau=0},
\]
the point being that we keep the $\tau^2\pa_\tau$-term. We can identify this with the stationary vector field (i.e.\ it commutes with $\cL_{\pa_t}$)
\[
  (\pa_\sigma p_0)\pa_t + (\pa_{\xi_i}p_0)\pa_{x^i} - (\pa_{x^i}p_0)\pa_{\xi_i} \in \cV(T^*M^\circ),
\]
which is indeed the same as the Hamilton vector field $H_{p_0}$ (and thus justifies the notation) if we define $p_0\in\CI(T^*M^\circ)$ by stationary extension, to wit, $p_0(t,x,\sigma,\xi)\equiv p_0(x,\sigma,\xi)$.

We make the following assumptions near the compact $\CI$ submanifold $\Gamma\subset\Sigma\cap\Scu^*_{\pa M}M$:
\begin{enumerate}[label=(P.\arabic*),ref=P.\arabic*]
\item\label{ItECPNondeg} we have $d p_0\neq 0$ on $\Sigma$ near $\Gamma$;
\item\label{ItECPWeight} there exists a function $\wh\rho\in\CI(\Tcu^*M\setminus o)$, $\wh\rho>0$, which is homogeneous of degree $-1$ such that $H_{p_0}\wh\rho=0$ near $\Gamma$;
\item\label{ItECPHam} the rescaled Hamilton vector field for the stationary model, $\sfH_{p_0}:=\wh\rho^{m-1}H_{p_0}\in\Vcu(\Scu^*M)$, is tangent to $\Gamma$, and satisfies
\[
  \inf_\Gamma \sfH_{p_0}t > 0.
\]
\item\label{ItECPMfdsBdy} there are $\CI$ orientable submanifolds $\bar\Gamma^{u/s}\subset\Sigma\cap\Scu^*_{\pa M}M$ near $\Gamma$ intersecting transversally at $\Gamma$ which have codimension\footnote{As in~\S\ref{SsEC}, the case of higher codimension can be treated similarly, but is not needed for our application.} $1$ inside of $\Sigma\cap\Scu^*_{\pa M}M$ and to which $\sfH_{p_0}$ is tangent;
\item\label{ItECPDefBdy} there exist defining functions $\bar\phi^{u/s}\in\CI(\Scu^*_{\pa M}M)$ of $\bar\Gamma^{u/s}$ inside of $\Sigma\cap\Scu^*_{\pa M}M$, defined locally near $\Gamma$, such that in a neighborhood of $\Gamma$ inside $\Sigma$,
  \begin{gather}
  \label{EqECPDefBdyHyp}
    \sfH_{p_0}\bar\phi^u = -\bar w^u\bar\phi^u,\quad \sfH_{p_0}\bar\phi^s = \bar w^s\bar\phi^s, \\
  \label{EqECPDefBdySymp}
    \wh\rho^{-1}H_{\bar\phi^u}\bar\phi^s = \wh\rho^{-1}\{\bar\phi^u,\bar\phi^s\} > 0,
  \end{gather}
  where we denote the homogeneous degree $0$ and $\tau$-independent extensions of $\bar\phi^{u/s}$ to $\Tcu^*M\setminus o$ by the same letters; we assume that
  \begin{equation}
  \label{EqECPDefBdyNumin}
     \nu_{\rm min} := \min\Bigl\{\inf_\Gamma\bar w^u,\inf_\Gamma\bar w^s\Bigr\} > 0.
  \end{equation}
\item\label{ItECPMfdsInt} there exist $\CI$ submanifolds $\Gamma^{u/s}\subset\Sigma$ such that $\Gamma^{u/s}\cap\Scu^*_{\pa M}M=\bar\Gamma^{u/s}$, and so that $\sfH_p:=\wh\rho^{m-1}H_p$ is tangent to $\Gamma^{u/s}$. There exist defining functions $\phi^{u/s}\in\CI(\Scu^*M)$ of $\Gamma^{u/s}$ inside of $\Sigma$ so that $\phi^{u/s}-\bar\phi^{u/s}\in\tau\CI(\Scu^*M)$.
\end{enumerate}

\begin{rmk}
\label{RmkECPMfdsDef}
  The conditions in assumptions~\eqref{ItECPDefBdy}--\eqref{ItECPMfdsInt} only depend on the restrictions of $\bar\phi^{u/s}$, resp.\ $\phi^{u/s}$, to $\Sigma\cap\Scu^*_{\pa M}M$, resp.\ $\Sigma\cap\Scu^*M$.
\end{rmk}

It is important for our application to relax the regularity of $P$ and $\Gamma^{u/s}$. Thus, we shall allow
\begin{equation}
\label{EqECPNonsmooth}
  P\in\Psicu^m(M;E)+\tau^\alpha\Psi_\infty^m(M;E),\quad \alpha>0.
\end{equation}
(The characteristic set of such $P$ is still smooth, and will be described precisely in Lemma~\ref{LemmaKChar} below.) We then relax assumption~\eqref{ItECPMfdsInt} to:
\begin{enumerate}[label=(\ref*{ItECPMfdsInt}'),ref=\ref*{ItECPMfdsInt}']
\item\label{ItECPMfdsInt2} there exist subsets $\Gamma^{u/s}\subset\Sigma$, which are $\CI$ submanifolds of $S^*M^\circ$ in $\tau>0$, such that $\Gamma^{u/s}\cap\Scu^*_{\pa M}M=\bar\Gamma^{u/s}$, and so that $\sfH_p:=\wh\rho^{m-1}H_p$ is tangent to $\Gamma^{u/s}$ in $\tau>0$. There exist functions $\phi^{u/s}\in\CI(S^*M^\circ)$ such that $\Gamma^{u/s}=\Sigma\cap(\phi^{u/s})^{-1}(0)$ near $\Gamma$, and so that
  \[
    \phi^{u/s} - \bar\phi^{u/s} \in \cA_\cuop^\alpha(\Scu^*M);
  \]
  see also Figures~\ref{FigMMap} and \ref{FigMMapSExt}.
\end{enumerate}

The weight in~\eqref{ItECPWeight} can be taken to be $\tau$-independent by replacing it with $(\tau,x)\mapsto\wh\rho(0,x)$. Assumption~\eqref{ItECPHam} implies a uniform lower bound for $\sfH_p t$ near $\Gamma$, and thus ensures that null-bicharacteristics of $P$ on $\Gamma^u$ tend to $\pa M$. A crucial consequence of assumption~\eqref{ItECPMfdsInt2} is that
\begin{equation}
\label{EqECPHamPhius}
  \sfH_p\phi^u = -w^u\phi^u,\quad \sfH_p\phi^s = w^s\phi^s,
\end{equation}
in $\Sigma$, where $w^{u/s}-\bar w^{u/s}\in \cA_\cuop^\alpha(\Scu^*M)$ (or $\CI(\Scu^*M)$ in the smooth case). We can now state the main microlocal theorem of this paper:

\begin{thm}
\label{ThmECP}
  Suppose $E$ is equipped with a positive definite fiber metric so that
  \begin{equation}
  \label{EqECPSubpr}
    \sup_\Gamma\sfp_1 < \half\nu_{\rm min}, \quad \sfp_1 := \wh\rho^{m-1}\sigma_\cuop^{m-1}\Bigl(\frac{1}{2 i}(P-P^*)\Bigr).
  \end{equation}
  Then there exist operators $B_0,B_1,G\in\Psicu^0(M)$, with $\WFcu'(B_0),\WFcu'(B_1),\WFcu'(G)$ contained in any fixed neighborhood of $\Gamma$, such that $B_0$ is elliptic at $\Gamma$, while $\WFcu'(B_1)\cap\bar\Gamma^u=\emptyset$, and $G$ is elliptic near $\Gamma$, such that for all $s,r,N\in\R$, we have
  \begin{equation}
  \label{EqECP}
    \|B_0 v\|_{\Hcu^{s,r}} \leq C\bigl(\|B_1 v\|_{\Hcu^{s+1,r}} + \|G P v\|_{\Hcu^{s-m+2,r}} + \|v\|_{\Hcu^{-N,r}}\bigr)
  \end{equation}
  for some constant $C>0$. This holds in the strong sense that if $v\in\Hcu^{-\infty,r}$ and the quantities on the right are finite, then so is the quantity on the left, and the inequality holds.

  This estimate also holds for $P^*$ in place of $P$ and for suitable choices of $B_0,B_1,G$, where the assumption on $B_1$ is now $\WFcu'(B_1)\cap\Gamma^s=\emptyset$.
\end{thm}

Thus, we can propagate $\Hcu^{s,r}$-control from $\Sigma\setminus\bar\Gamma^u$ into a neighborhood of $\Gamma$, and from there by standard real principal type propagation, out along $\bar\Gamma^u$. The last statement concerns propagation in the opposite direction for the adjoint problem: control from $\Sigma\setminus\Gamma^s$ can be propagated into $\Gamma$ and out along $\Gamma^s$ (in particular, into $\tau>0$). See Figure~\ref{FigECP}.

\begin{figure}[!ht]
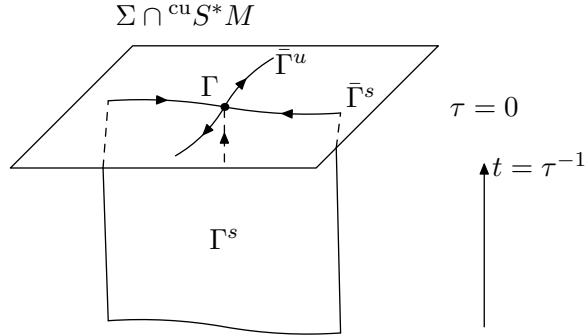

\centering
  \inclfig{FigECP}
  \caption{Illustration of the setup of Theorem~\ref{ThmECP}.}
\label{FigECP}
\end{figure}

\begin{rmk}
\label{RmkECPNhyptr}
  One can replace the weighted $\Hcu$-norms in the theorem by weighted $\Hb$-norms, with the same indices, on the exponential compactification
  \begin{equation}
  \label{EqECPNhyptrExp}
    M_e=[0,\infty)_{\tau_e}\times X,\quad \tau_e=e^{-t},
  \end{equation}
  of $M$, provided $P\in\Psib^m(M_e)$. Note that then $r$ is the order of an \emph{exponential} weight. One can reduce to the case $r=0$ (for which $\Hcu^{s,0}(M)=\Hb^{s,0}(M_e)$) by conjugating $P$ by the exponential $e^{-r t}$; this changes $\sfp_1$, and we therefore only obtain an estimate at $\Gamma$ on $\Hb^{s,r}$ when $r<r_0$ for some threshold $r_0\in\R$; in fact $r_0$ is positive when $P$ is symmetric to leading and subleading order, i.e.\ modulo $\Psib^{m-2}$. In this way, Theorem~\ref{ThmECP} extends \cite[Theorem~3.2]{HintzVasyNormHyp}, which takes place in the setting $P-P^*\in\Psib^{m-2}$, to a wider range of weights, in particular providing direct estimates on exponentially decaying function spaces which however become more lossy once one crosses the weight $r=0$.
\end{rmk}

\begin{rmk}
\label{RmkECPTame}
  The analogues of the observations in Remark~\ref{RmkESDefect} apply also here. Furthermore, with nonlinear applications in mind, we recall that positive commutator arguments, as used in the proof below, can be generalized easily to the case that $P$ has coefficients with high, but only finite regularity, with quantitative control of $C$ in~\eqref{EqECP} on their $H^s$-norm (roughly speaking); in fact, one can prove more precise \emph{tame estimates} as in~\cite{HintzQuasilinearDS,HintzVasyQuasilinearKdS}. Details will be discussed elsewhere.
\end{rmk}

\begin{proof}[Proof of Theorem~\usref{ThmECP}]
  \emph{We drop the bundle $E$ from the notation.} We shall first consider the case that $P\in\Psicu^m$ and work under the stronger assumption~\eqref{ItECPMfdsInt}. Denote by $\Phi^{u/s}=(\Phi^{u/s})^*\in\Psicu^0(M)$ quantizations of $\phi^{u/s}$. In order to not overburden the notation, we will re-use symbols for microlocal cutoffs and adjust their size to our needs as we proceed.

  Steps 1--4 of the proof follow the outline~\eqref{ItESSkEq}--\eqref{ItESSkFin} in the closed manifold setting. We first prove the estimate~\eqref{EqECP} for $v\in\CIdot(M)$, leaving the regularization argument giving the stronger statement to step 5. In step 6, we consider the general case~\eqref{ItECPMfdsInt2}; in step 7, we prove the adjoint estimate.

  \textit{\underline{Step 1:} equation and estimate for $\Phi^u v$.} Let $W^{u/s}\in\Psicu^{m-1}(M)$ denote quantizations of $\wh\rho^{-m+1}w^{u/s}$; then by~\eqref{EqECPHamPhius}, we have
  \[
    i[P,\Phi^u] = -W^u\Phi^u + i R_1 P + i R_2,\quad R_1\in\Psicu^{-1},\ R_2\in\Psicu^{m-2};
  \]
  the term $R_1 P$ arises because~\eqref{EqECPHamPhius} only controls the principal symbol of the commutator on the characteristic set, see also equation~\eqref{EqECPHpPhius} below. Therefore,
  \begin{equation}
  \label{EqECPPhiuv}
    P'v'=f',\quad P':=P-i W^u,\ v':=\Phi^u v,\ f' := (\Phi^u+R_1)P v + R_2 v.
  \end{equation}
  We record that $f'$ satisfies the estimate
  \begin{equation}
  \label{EqECPPhiuvForc}
    \|G f'\|_{\Hcu^{s-m+2,r}} \leq C\bigl(\|\wt G P v\|_{\Hcu^{s-m+2,r}} + \|\wt G v\|_{\Hcu^{s,r}} + \|v\|_{\Hcu^{s-1,r}}\bigr)
  \end{equation}
  for $G,\wt G\in\Psicu^0(M)$ with $\WFcu'(G)\subset\Ellcu(\wt G)$.

  We point out that this is the step which requires $\Gamma^u$---even though in the dynamical setting of Theorem~\ref{ThmMMap} and \ref{ThmMFlow} (where we consider backwards propagation, hence the roles of $\Gamma^s$ and $\Gamma^u$ are reversed relative to the present discussion) it is \emph{not} canonically defined---to be preserved by the $\sfH_p$ flow. Indeed, without arranging $\sfH_p\phi^u=0$ on $\Gamma^u$, one would get an error term from the commutator $[P,\Phi^u]$ which on $\Sigma$ does not factor through $\Phi^u$ and which thus, while decaying, has too high order in the differential sense to be absorbed later on: $f'$ would get an extra term $\tau\Psicu^{m-1}(M)v$, but since we need to estimate the $\Hcu^{s-m+2,r}$ norm of $f'$, this would necessitate control of the $\Hcu^{s+1,r-1}$ norm of $v$ near $\Gamma$, which is of higher order in the sense of differential orders than the $\Hcu^{s,r}$ control we will be able to get (having a weaker weight $r-1$ does not affect this issue).

  Noting that $\sfp_1':=\wh\rho^{m-1}\sigma\bigl(\frac{1}{2 i}(P'-(P')^*)\bigr)=\sfp_1-w^u$ on $\Sigma$, the assumption~\eqref{EqECPSubpr} implies that $\sfp_1'<-\half\nu_{\rm min}$ has a favorable sign at $\Gamma$. Quantitatively, letting $\sfp=\wh\rho^m p$, define the commutant
  \begin{equation}
  \label{EqECPPhiuComm}
    a = \check a^2,\quad \check a=\tau^{-r}\wh\rho^{-s-1+(m-1)/2}\chi^u\bigl((\phi^u)^2\bigr)\chi^s\bigl((\phi^s)^2\bigr)\chi_T(\tau)\chi_\Sigma(\sfp),
  \end{equation}
  where $\chi^{u/s},\chi_T,\chi_\Sigma\in\CIc(\R)$ have non-positive derivatives on $[0,\infty)$, and $\sqrt{-\chi^{u/s}(\chi^{u/s})'}$, $\sqrt{-\chi_T\chi_T'}\in\CI([0,\infty))$. We can choose the supports of the cutoffs so small that $\supp a$ is localized in any fixed neighborhood of $\Gamma$ inside $\Scu^*M$.
  
  With $\check A\in\tau^{-r}\Psicu^{s+1-(m-1)/2}(M)$ and $A=\check A^*\check A$ denoting quantizations of $\check a$ and $a$, we then consider the $L^2_\cuop$ pairing
  \begin{equation}
  \label{EqECPPhiuPair}
    \Im\la f',A v'\ra = \half i\bigl(\la A v',P'v'\ra-\la P'v',A v'\ra\bigr) = \Big\la\bigl(\half i[P',A]+\tfrac{P'-(P')^*}{2 i}A\bigr)v',v'\Big\ra.
  \end{equation}
  Writing
  \begin{equation}
  \label{EqECPHpPhius}
    \sfH_p\phi^u=-w^u\phi^u+\sfr^u\sfp,\quad
    \sfH_p\phi^s=w^s\phi^s+\sfr^s\sfp,
  \end{equation}
  and fixing $0<c<-\inf_\Gamma\sfp_1'$, we compute the principal symbol of the operator on the right as
  \begin{equation}
  \label{EqECPPhiuCommSymb}
  \begin{split}
    \check a H_p\check a + \wh\rho^{-m+1}\sfp_1'\check a^2 &= -c\wh\rho^{-m+1}\check a^2-(\wh\rho^{-s-1}b_0)^2-(\wh\rho^{-s-1}b^s)^2 \\
      &\hspace{8em} + (\wh\rho^{-s-1}b^u)^2 + (\wh\rho^{-s-1}b_T)^2 + h p,
  \end{split}
  \end{equation}
  where, writing $\chi=\chi^u\chi^s\chi_T\chi_\Sigma$, and $\wh{\chi^u}=\chi^s\chi_T\chi_\Sigma$ etc,
  \begin{align*}
    b_0 &= \tau^{-r}\chi\bigl(-\sfp_1'-c+r\tau(\tau^{-2}\sfH_p\tau) + (s+1-(m-1)/2)(\wh\rho^{-1}\sfH_p\wh\rho)\bigr)^{1/2}, \\
    b^{u/s} &= \tau^{-r}\wh{\chi^{u/s}}\phi^{u/s}\sqrt{-2 w^{u/s}\chi^{u/s}(\chi^{u/s})'}, \\
    b_T &= \tau^{-r}\wh{\chi_T}\sqrt{\tau\chi_T\chi_T'(\tau^{-2}\sfH_p\tau)}, \\
    h &= 2\tau^{-2 r}\wh\rho^{-2 s-2+m}\chi\bigl(\wh{\chi^u}(\chi^u)'\phi^u\sfr^u+\wh{\chi^s}(\chi^s)'\phi^s\sfr^s + m\wh{\chi_\Sigma}\chi_\Sigma'(\wh\rho^{-1}\sfH_p\wh\rho)\bigr);
  \end{align*}
  note here that $\tau^{-2}\sfH_p\tau=-\sfH_p t<0$ on $\Gamma$ by assumption~\eqref{ItECPHam}, and $\wh\rho^{-1}\sfH_p\wh\rho=o(1)$ as $\tau\to 0$ by assumption~\eqref{ItECPWeight}, hence all square roots are well-defined and smooth provided $\supp a$ is sufficiently close to $\Gamma$. Writing $B_0=\Op(b_0)$ etc, we thus find that
  \begin{align*}
    &c\|\check A v'\|_{\Hcu^{(m-1)/2}}^2  + \|B_0 v'\|_{\Hcu^{s+1,r}}^2 + \|B^s v'\|_{\Hcu^{s+1,r}}^2 \\
    &\qquad\leq \|B^u v'\|_{\Hcu^{s+1,r}}^2 + \|B_T v'\|_{\Hcu^{s+1,r}}^2 + |\la P'v',H v'\ra| \\
    &\qquad\qquad + \bigl(\|G v'\|_{\Hcu^{s+1/2,r}}^2 + C\|v'\|_{\Hcu^{s-1/2,r}}^2\bigr) + c\|\check A v'\|_{\Hcu^{(m-1)/2}}^2 + \tfrac{1}{2 c}\|\check A P' v'\|_{\Hcu^{-(m-1)/2}}
  \end{align*}
  for some constant $C\in\R$, where the terms in parentheses on the last line are due to the equality~\eqref{EqECPPhiuCommSymb} only concerning the \emph{principal} symbols; we keep a microlocalization $G\in\Psicu^0(M)$ near $\supp a$, which is estimated in a norm which is $\half$ weaker than the control afforded by $B_0$, to set up an iterative argument, given momentarily, for improving the error term to $\|v'\|_{\Hcu^{s-1/2,r}}^2$. Absorbing the second to last term into the left hand side, and simplifying by dropping the last term on the left, we obtain
  \begin{equation}
  \label{EqECPPhiuEst0}
    \|B_0\Phi^u v\|_{\Hcu^{s+1,r}} \leq \|B_1\Phi^u v\|_{\Hcu^{s+1,r}} + \|G f'\|_{\Hcu^{s-m+2,r}} + \|G\Phi^u v\|_{\Hcu^{s+1/2,r}} + C\|\Phi^u v\|_{\Hcu^{s-1/2,r}};
  \end{equation}
  here, $B_1\in\Psicu^0(M)$ is a suitable operator, microlocalized near $\Gamma$, which is elliptic on $\Ellcu(B^u)\cup\Ellcu(B_T)$ and satisfies $\WFcu'(B_1)\cap\bar\Gamma^u=\emptyset$. This estimate can be iterated for the term $G\Phi^u v$, i.e.\ replacing $B_0$ by $G$, and replacing the control terms on the right by ones with slightly larger microsupport. (Recall that our arguments can be localized arbitrarily closely to $\Gamma$.) Iterating twice, we can combine the resulting term $\|G''\Phi^u v\|_{\Hcu^{s+1/2-2\cdot 1/2,r}}$, with $\WFcu'(G'')$ near $\Gamma$, into the final term; estimating $G f'$ using~\eqref{EqECPPhiuvForc} gives
  \begin{equation}
  \label{EqECPPhiuEst}
    \|B_0\Phi^u v\|_{\Hcu^{s+1,r}} \leq C\bigl(\|B_1\Phi^u v\|_{\Hcu^{s+1,r}} + \|\wt G P v\|_{\Hcu^{s-m+2,r}} + \|\wt G v\|_{\Hcu^{s,r}} + \|v\|_{\Hcu^{s-1/2,r}}\bigr)
  \end{equation}
  upon slightly re-defining $B_0,B_1,\wt G$; their properties are: for some small fixed $\eps_0>0$,
  \[
    \cU_{3\eps_0} := \{ \tau,\,|\phi^u|,\,|\phi^s|,\,|\sfp|<3\eps_0 \} \subset \Ellcu(B_0),
  \]
  while $\WFcu'(B_1)$ is disjoint from an $\eps_0$-neighborhood of $\bar\Gamma^u$, and $\wt G,B_0,B_1$ are all microsupported in $\cU_{4\eps_0}$.

  \textit{\underline{Step 2:} quantitative propagation for $\Phi^u$.} We now estimate $v$ close to $\bar\Gamma^s$ as the solution of the equation $\Phi^u v=v'$ via a propagation estimate. We choose a new commutant
  \begin{equation}
  \label{EqECPPhiuPropSymb}
    a = \check a^2,\quad \check a=\tau^{-r}\wh\rho^{-s-1/2}\psi(\phi^s)\chi^u\bigl((\phi^u)^2\bigr)\chi_T(\tau)\chi_\Sigma(\sfp),
  \end{equation}
  where $\chi^u,\chi_T,\chi_\Sigma$ as before, identically $1$ in $[-\eps_0,\eps_0]$, and with their supports and $\supp\psi$ chosen small enough so that $\supp a\subset\cU_{3\eps_0}$. Concretely, for $0<\delta<\frac{\eps_0}{3}$ to be chosen later, we arrange the smooth function $\psi$ to satisfy $\supp\psi\subset[-4\delta,2\eps_0]$, $0\leq\psi\leq 1$, with $\psi'\geq 0$ in $\phi^s\leq 0$, $\psi(-3\delta)\geq\tfrac{1}{4}$, and $\psi'\geq\frac{1}{12\delta}$ on $[-3\delta,3\delta]$, while $\psi'\geq-\frac{1}{\eps_0}$ in $\phi^s\geq 0$. In $[-3\delta,3\delta]$, this implies $\psi\leq 12\delta\psi'$; we arrange for this to hold in all of $[-4\delta,3\delta]$.
  
  Split $\psi^2=\psi_-^2+\psi_+^2$, where $\psi_-=\psi$ in $\phi^s\leq 3\delta$ and $\supp\psi_-\subset[-4\delta,4\delta]$, and $0\leq\psi_+\leq 1$ has $\supp\psi_+\subset(3\delta,2\eps_0]$. We can then write
  \[
    \psi\psi' = \tfrac{1}{24\delta}\psi_-^2 + \tfrac{1}{24\delta}\wt b_\delta^2 - \wt e,
  \]
  where $\supp\wt e\subset(3\delta,2\eps_0]$, $|\wt e|\leq\frac{1}{\eps_0}$, and $\wt b_\delta\geq 0$, $\supp\wt b_\delta\subset[-4\delta,4\delta]$, and, crucially, $\wt b_\delta\geq\tfrac14$ in $[-3\delta,3\delta]$.\footnote{Smoothness of $\wt b_\delta$ is only an issue near the left boundary of its support (while near its right boundary we may simply cut $\wt b_\delta$ off smoothly in $[3\delta,4\delta]$); but this can be ensured by taking $\psi$ to be of the form $e^{-1/(\phi^s+4\delta)}$ in $\phi^s>-4\delta$ near $\phi^s=-4\delta$. This relies on the fact that for $f(x)=e^{-1/x}$ ($x>0$), one has $f'=x^{-2}f$, so e.g.\ $f f'-C f^2=(x^{-1}f\sqrt{1-C x^2})^2$ indeed has a smooth square root near $x=0$ since $x^{-1}f$ is smooth.} We define $\check a_\pm$ by replacing $\psi$ by $\psi_\pm$ in the definition~\eqref{EqECPPhiuPropSymb} of $\check a$, so $\check a^2=\check a_-^2+\check a_+^2$.

  By assumption~\eqref{EqECPDefBdySymp}, and shrinking $\eps_0$ if necessary (also in the construction of the commutant), we can choose $C_\phi>1$ so that
  \[
    C_\phi^{-1} \leq \sfH_{\phi^u}\phi^s \leq C_\phi\ \ \text{on}\ \ \cU_{3\eps_0},\qquad \sfH_{\phi^u}:=\wh\rho^{-1}H_{\phi^u}.
  \]
  Therefore,
  \begin{equation}
  \label{EqECPPhiuPropComm}
    \check a H_{\phi^u}\check a = \tfrac{1}{48 C_\phi\delta}\wh\rho\,\check a_-^2 + (\tau^{-r}\wh\rho^{-s}b_-)^2 + \tfrac{1}{24 C_\phi\delta}(\tau^{-r}\wh\rho^{-s}b_\delta)^2 + \tau^{-2 r}\wh\rho^{-2 s}e_1 - \tau^{-2 r}\wh\rho^{-2 s}e_2,
  \end{equation}
  where
  \begin{align*}
    b_- &= \psi_-\chi^u\chi_T\chi_\Sigma\bigl(\tfrac{1}{\sqrt{24\delta}}(\sfH_{\phi^u}\phi^s-\tfrac{1}{\sqrt 2}C_\phi^{-1})-r\tau(\tau^{-2}\sfH_{\phi^u}\tau)-(s+\half)(\wh\rho^{-1}\sfH_{\phi^u}\wh\rho)\bigr)^{1/2}, \\
    b_\delta &= \chi^u\chi_T\chi_\Sigma\wt b_\delta\sqrt{C_\phi(\sfH_{\phi^u}\phi^s)}, \\
    e_1 &= \psi^2(\chi^u)^2\bigl(\tau^2\chi_T\chi_T'\chi_\Sigma^2(\tau^{-2}\sfH_{\phi^u}\tau) + \chi_T^2\chi_\Sigma\chi_\Sigma'(\sfH_{\phi^u}\sfp)\bigr), \\
    e_2 &= (\chi^u)^2\chi_T^2\chi_\Sigma^2\Bigl((\sfH_{\phi^u}\phi^s)\wt e + \psi_+^2\bigl(r\tau(\tau^{-2}\sfH_{\phi^u}\tau)+(s+\half)(\wh\rho^{-1}\sfH_{\phi^u}\wh\rho)\bigr)\Bigr).
  \end{align*}
  The supports of $\chi^u$, $\chi_T$, and $\chi_\Sigma$ can be taken to be independent of $\delta$. Taking $\delta>0$ small ensures that the square roots in $b_-$ and $b_\delta$ are well-defined and smooth. The key point is then that
  \[
    b_\delta\geq\tfrac{1}{4}\ \ \text{on}\ \ \cU_{\eps_0}\cap\{-3\delta\leq\phi^s\leq 3\delta\},
  \]
  while the a priori control term $e_2$ has support in $\cU_{3\eps_0}$ and satisfies
  \[
    \phi^s>3\delta\ \ \text{on}\ \ \supp e_2,\quad
    |e_2|\leq C_\phi\eps_0^{-1}+C'
  \]
  for some $\delta$-independent constant $C'$. The error term $e_1$ is the sum of a piece with support in $\tau>0$ and a piece supported away from the characteristic set of $P$; both places are away from $\bar\Gamma^u$ and hence we have a priori control there.

  Quantizing these symbols, with the corresponding operators denoted by upper case letters, evaluation of the pairing $\Im\la v',A v\ra=\la \half i[\Phi^u,A]v,v\ra$ gives
  \begin{align*}
    &\tfrac{1}{48 C_\phi\delta}\|\check A_- v\|_{\Hcu^{-1/2,0}}^2 + \|B_-v\|_{\Hcu^{s,r}}^2 + \tfrac{1}{24 C_\phi\delta}\|B_\delta v\|_{\Hcu^{s,r}}^2 \\
    &\qquad\leq |\la E_1 v,v\ra| + |\la E_2 v,v\ra| + |\la\check A_-v',\check A_-v\ra| + |\la\check A_+ v',\check A_+ v\ra| + \wt C\|v\|_{\Hcu^{s-1/2,r}}^2 \\
    &\qquad\leq |\la E_2 v,v\ra| + \|B_1 v\|_{\Hcu^{s,r}}^2 + \tfrac{1}{48 C_\phi\delta}\|\check A_-v\|^2_{\Hcu^{-1/2,0}} + 24 C_\phi\delta\|\check A_-\Phi^u v\|_{\Hcu^{1/2,0}}^2 \\
    &\qquad\qquad + \tfrac{1}{2\eps_0}\|\check A_+ v\|_{\Hcu^{-1/2,0}}^2 + \eps_0\|\check A_+\Phi^u v\|_{\Hcu^{1/2,0}}^2 + \wt C\|v\|_{\Hcu^{s-1/2,r}}^2,
  \end{align*}
  where $B_1\in\Psicu^0(M)$ (with implicit dependence on $\delta$) is elliptic on $\supp e_1$ and has wave front set disjoint from an $\eps_0$-neighborhood of $\bar\Gamma^u$; the constant $\wt C$ depends on $\delta$ and arises from the fact that~\eqref{EqECPPhiuPropComm} only captures \emph{principal} symbols. The third term on the right can be absorbed into the left hand side, and the second term on the left can be dropped. In view of the bounds on the symbols, and using G\aa{}rding's inequality in the form of Lemma~\ref{LemmaESBound}, we obtain, after overall multiplication by $24 C_\phi\delta$, the estimate (cf.\ the estimate~\eqref{EqESSkProp})
  \begin{equation}
  \label{EqECPPhiuProp}
  \begin{split}
    \|B_\delta v\|_{\Hcu^{s,r}} &\leq C\bigl(\sqrt{\delta/\eps_0}\|\wt E_\delta v\|_{\Hcu^{s,r}} + \|B_1 v\|_{\Hcu^{s,r}} \\
      &\qquad\qquad + (\delta+\sqrt{\delta\eps_0})\|G\Phi^u v\|_{\Hcu^{s+1,r}} + \wt C\|v\|_{\Hcu^{s-1/2,r}} \bigr),
  \end{split}
  \end{equation}
  where $C$ is a $\delta$-independent constant; here, $\wt E_\delta\in\Psicu^0(M)$ controls the terms involving $E_2 v$ and $\check A_+ v$, has principal symbol bounded from above by $1$, and has $\phi^s>3\delta$ on $\WFcu'(\wt E_\delta)$; and
  \begin{equation}
  \label{EqECPPhiuPropBdelta}
    B_\delta\in\Psicu^0(M),\quad
    \sigma(B_\delta)\geq 1\ \ \text{on}\ \ \cU_{\eps_0}\cap\{-3\delta\leq\phi^s\leq 3\delta\}.
  \end{equation}
  The operator $G\in\Psicu^0(M)$ is microsupported in $\cU_{3\eps_0}$. A similar estimate holds in $\phi^s<0$.
  
  Let us now appeal to the estimate~\eqref{EqECPPhiuEst} (which is of course independent of $\delta$), with the present $G$ taking the place of $B_0$ there. Its third term gives a contribution $C(\delta+\sqrt{\delta\eps_0})\|\wt G v\|_{\Hcu^{s,r}}$ in~\eqref{EqECPPhiuProp}, where we may assume that $|\sigma(\wt G)|\leq 1$ and that $C$ is independent of $\delta$. But then we can bound
  \begin{equation}
  \label{EqECPPhiuProp2}
    \|\wt G v\|_{\Hcu^{s,r}}\leq 2\|B_\delta v\|_{\Hcu^{s,r}} + 2\|E_\delta v\|_{\Hcu^{s,r}} + \wt C\|B_1 v\|_{\Hcu^{s,r}} + \wt C\|v\|_{\Hcu^{s-1/2,r}},
  \end{equation}
  where $E_\delta\in\Psicu^0(M)$ satisfies $|\sigma(E_\delta)|\leq 1$,
  \begin{alignat*}{4}
    &\sigma(E_\delta)=1&& \ \text{on}\ & \{ 3\delta\leq\phi_s\leq 4\eps_0,\ &|\phi^u|\leq\eps_0,\ &&|\sfp|\leq\eps_0 \}, \\
    &\WFcu'(E_\delta)&& \,\subset & \{ \tfrac52 \delta\leq\phi_s\leq 5\eps_0,\ &|\phi^u|\leq 2\eps_0,\ &&|\sfp|\leq 2\eps_0 \}.
  \end{alignat*}
  The third term $\|B_1 v\|$, takes care of $\wt G v$ more than $\eps_0$ away from $\bar\Gamma^u$. Choosing $\delta>0$ small so that $2 C(\delta+\sqrt{\delta\eps_0})\leq\half$, we can absorb the first term in~\eqref{EqECPPhiuProp2} into the left hand side of~\eqref{EqECPPhiuProp}. Dropping irrelevant $\delta$-dependencies and fixing $\eps_0$, we conclude that
  \begin{equation}
  \label{EqECPPhiuPropEst}
    \|B_\delta v\|_{\Hcu^{s,r}} \leq C\bigl(\delta^{1/2}\|E_\delta v\|_{\Hcu^{s,r}} + \|B_1 v\|_{\Hcu^{s+1,r}} + \|\wt G P v\|_{\Hcu^{s-m+2,r}} + \wt C\|v\|_{\Hcu^{s-1/2,r}}\bigr).
  \end{equation}
  We summarize the key features: $C,\wt C$ are constants, with $C$ (but not $\wt C$) independent of $\delta$; $B_\delta$ is as in~\eqref{EqECPPhiuPropBdelta}, while $\WFcu'(B_1)$ is disjoint from an $\eps_0$-neighborhood of $\bar\Gamma^u$; and $\WFcu'(\wt G)\subset\cU_{4\eps_0}$. (We give ourselves some extra room compared to the sketch of the argument in~\S\ref{SsEC}.)

  \textit{\underline{Step 3:} quantitative propagation for $P$.} We now use quantitative real principal type propagation for $P$ to estimate $\|E_\delta v\|_{\Hcu^{s,r}}$ by $\|B_\delta v\|_{\Hcu^{s,r}}$ in~\eqref{EqECPPhiuPropEst}. The two qualitative differences to the previous step are \begin{enumerate*} \item the repelling nature of $\sfH_p$ at $\Gamma^s$, necessitating a different time scale $\log\delta^{-1}$ for propagation from $\phi^s=\delta$ to $\phi^s=\eps_0$, and \item the presence of the subprincipal part $\sfp_1$ of $P$, leading to exponential (in the propagation time) growth of the constants in the propagation estimate. \end{enumerate*} Thus, fixing $\beta\in\R$ subject to
  \begin{equation}
  \label{EqECPPPropBeta}
    \max\Bigl\{\frac{\sup_\Gamma\sfp_1}{\nu_{\rm min}},0\Bigr\}<\beta<\half,
  \end{equation}
  we consider a commutant of the form
  \begin{equation}
  \label{EqECPPPropComm}
    a = \check a^2,\quad \check a=\tau^{-r}\wh\rho^{-s+(m-1)/2}|\phi^s|^{-\beta}\psi\bigl(\log|\phi^s/\delta|\bigr)\chi^u\bigl((\phi^u)^2\bigr)\chi_T(\tau)\chi_\Sigma(\sfp),
  \end{equation}
  where $\chi^u,\chi_T,\chi_\Sigma$ are cutoffs as above and chosen to be identically $1$ on $[-3\eps_0,3\eps_0]$, while $\psi\in\CIc\bigl((0,\log(6\eps_0\delta^{-1}))\bigr)$, which in particular cuts out the singularity of $|\phi^s|^\beta$ at $\phi^s=0$. Since the weight $|\phi^s|^{-\beta}$ will give positivity, as we explain below, we may choose $\psi$ less carefully than above (and with various sign switches, as the direction of propagation is now \emph{away} from $\Gamma^s$).\footnote{In part~\eqref{ItESSkLog} of the proof sketch in~\S\ref{SsEC}, we did not exploit the freedom of inserting a weight $|\phi^s|^{-\beta}$, and instead relied on quantitative decay of $\psi$ on $[\log\tfrac94,\log(2\eps_0\delta^{-1}]$ at a rate $\psi'\sim 1/\log\delta^{-1}$, giving in~\eqref{EqECPPPropSymb} a contribution of the same sign as $b_{\ell,+}$ but by a factor $1/\log\delta^{-1}$ smaller than the fixed size control afforded by $\beta>0$.} Namely, we arrange $0\leq\psi\leq 1$, further $\psi'\leq 1$ on $[0,\log\tfrac94]$, and $\psi'<0$ and $\psi\geq\tfrac14$ on $[\log\tfrac94,\log(5\eps_0\delta^{-1})]$. We may then further arrange
  \[
    \psi\psi' = -\wt b_\ell^2 + \wt e,
  \]
  where $\supp\wt e\subset[0,\log\tfrac94)$, $|\wt e|\leq 1$, while $\wt b_\ell\geq 0$, with $\supp\wt b_\ell\subset[\log\tfrac32,\log(6\eps_0\delta^{-1}))$, is smooth. We again write $\psi^2=\psi_-^2+\psi_+^2$, where now $\psi_+=\psi$ on $[\log\tfrac94,\log(6\eps_0\delta^{-1})]$, and $0\leq\psi_-\leq 1$ has $\supp\psi_-\subset[0,\log\tfrac94)$.
  
  Denote by $\check A$ and $A=\check A^*\check A$ quantizations of $\check a$ and $a$, respectively. Similarly to~\eqref{EqECPPhiuPair}, we then consider the commutator
  \[
    \Im\la P v,A v\ra = \Big\la\bigl(\half i[P,A]+\tfrac{P-P^*}{2 i}A\bigr)v,v\Big\ra.
  \]
  Fixing $c\in(0,\beta\nu_{\rm min}-\sup_\Gamma\sfp_1)$, we then calculate, using~\eqref{EqECPHpPhius},
  \begin{equation}
  \label{EqECPPPropSymb}
    \check a H_p\check a+\wh\rho^{-m+1}\sfp_1\check a^2 = -c\wh\rho^{-m+1}\check a^2 - (\tau^{-r}\wh\rho^{-s}b_{\ell,-})^2 - (\tau^{-r}\wh\rho^{-s}b_{\ell,+})^2 + e_1 + e_2 + h p,
  \end{equation}
  where now
  \begin{align*}
    b_{\ell,\pm} &= |\phi^s|^{-\beta}\psi_\pm\chi^u\chi_T\chi_\Sigma\bigl(\beta w^s-\sfp_1-c \\
      &\qquad\qquad + r\tau(\tau^{-2}\sfH_p\tau)+(s-(m-1)/2)(\wh\rho^{-1}\sfH_p\wh\rho) + w^s\wt b_\ell^2\bigr)^{1/2}, \\
    e_1 &= \tau^{-2 r}\wh\rho^{-2 s}|\phi^s|^{-2\beta}\psi^2\chi_\Sigma^2\bigl(2(\phi^u\sfH_p\phi^u)\chi^u(\chi^u)'\chi_T^2 + (\tau^{-2}\sfH_p\tau)\tau^2(\chi^u)^2\chi_T\chi_T'\bigr), \\
    e_2 &= \tau^{-2 r}\wh\rho^{-2 s}|\phi^s|^{-2\beta}w^s\wt e(\chi^u)^2\chi_T^2\chi_\Sigma^2, \\
    h &= \tau^{-2 r}\wh\rho^{-2 s+m}|\phi^s|^{-2\beta}\psi^2(\chi^u)^2\chi_T^2\bigl((-\beta\sfr^s+\psi\psi'\tfrac{\sfr^s}{\phi^s})\chi_\Sigma^2 + m(\wh\rho^{-1}\sfH_p\wh\rho)\chi_\Sigma\chi_\Sigma'\bigr).
  \end{align*}
  Here, the weight $|\phi^s|^{-\beta}=e^{-\beta\log|\phi^s|}$ in $\check a$, which is exponentially decreasing along the $\sfH_p$ flow near $\bar\Gamma^u$, means that we are giving up control as we get farther away from $\Gamma^s$; this counteracts the allowed growth of amplitudes caused by the subprincipal symbol $\sfp_1$ of $P$. Furthermore, $e_2$, which has a quantitative bound coming from that of $\wt e$, captures errors near $\bar\Gamma^s$, which is where we propagate from; the term $e_1$ captures terms away from $\bar\Gamma^u$, where we have a priori control. The main term in this calculation in $b_{\ell,+}\geq 0$, which gives control away from $\Gamma^s$. Concretely, we have uniform lower and upper bounds
  \begin{equation}
  \label{EqECPPPropSymbBounds}
  \begin{aligned}
    &b_{\ell,+}\geq c'&\ \ \text{on}\ \ &\cU_{3\eps_0}\cap \{|\phi^s|\in[\tfrac94\delta,5\eps_0]\}, \\
    &|e_2|\leq C'\tau^{-2 r}\wh\rho^{-2 s}\delta^{-2\beta}&\ \ \text{on}\ \ &\supp e_2\subset\cU_{4\eps_0}\cap \{ |\phi^s|\in[\delta,\tfrac94\delta) \}
  \end{aligned}
  \end{equation}
  for some $\delta$-independent constants $c',C'>0$.
  
  We now proceed as usual by quantizing the relation~\eqref{EqECPPPropSymb}, obtaining
  \begin{align*}
    &c\|\check A v\|_{\Hcu^{(m-1)/2,0}}^2 + \|B_{\ell,+}v\|_{\Hcu^{s,r}}^2 + \|B_{\ell,-}v\|_{\Hcu^{s,r}}^2 \\
    &\qquad\leq |\la E_1 v,v\ra| + |\la E_2 v,v\ra| + |\la P v,H v\ra| + |\la\check A P v,\check A v\ra| + \wt C\|v\|_{\Hcu^{s-1/2,r}}^2 \\
    &\qquad\leq |\la E_1 v,v\ra| + |\la E_2 v,v\ra| + \|\wt G'P v\|_{\Hcu^{s-m+1,r}}^2 +\wt C\|\wt G'v\|_{\Hcu^{s-1,r}}^2 \\
    &\qquad\qquad\qquad + c\|\check A v\|_{\Hcu^{(m-1)/2,0}}^2 + \tfrac{1}{2 c}\|\check A P v\|_{\Hcu^{-(m-1)/2,0}}^2 + \wt C\|v\|_{\Hcu^{s-1/2,r}}^2,
  \end{align*}
  where $\wt G'\in\Psicu^0(M)$ is elliptic near $\supp a$ and has wave front set contained in a small neighborhood thereof, and $\wt C$ is a $\delta$-dependent constant. We can absorb the fifth term on the right into the left hand side, drop the third term on the left, and combine the $\wt G'v$ term with the last, error, term. We estimate $|\la E_2 v,v\ra|$ by means of G\aa{}rding's inequality and~\eqref{EqECPPPropSymbBounds}. Using the operators in~\eqref{EqECPPhiuPropEst}, this gives the estimate
  \begin{equation}
  \label{EqECPPPropEst}
    \|E_\delta v\|_{\Hcu^{s,r}} \leq C\bigl(\delta^{-\beta}\|B_\delta v\|_{\Hcu^{s,r}} + \|\wt B_1 v\|_{\Hcu^{s,r}} + \wt C\|\wt G'P v\|_{\Hcu^{s-m+1,r}} + \wt C\|v\|_{\Hcu^{s-1/2,r}}\bigr)
  \end{equation}
  for some $\delta$-independent constant $C$, while $\wt C$ may depend on $\delta$; here, $\wt B_1$ is a quantization of a symbol $\wt b_1$ which is elliptic in the complement of an $\half\eps_0$-neighborhood of $\bar\Gamma^u$ within $\cU_{6\eps_0}$, and with $\WFcu'(\wt B_1)$ disjoint from an $\tfrac14\eps_0$-neighborhood of $\bar\Gamma^u$; this takes care of those parts of $|\la E_1 v,v\ra|$ which lie outside the elliptic set of $B_\delta$ and are thus away from $\bar\Gamma^u$.

  \textit{\underline{Step 4:} combining the estimates.} Plugging the estimate~\eqref{EqECPPPropEst} into~\eqref{EqECPPhiuPropEst}, we can absorb the resulting term $\delta^{1/2-\beta}\|B_\delta u\|$ (with $\beta<\half$ by~\eqref{EqECPPPropBeta}) into the left hand side of~\eqref{EqECPPhiuPropEst} upon fixing $\delta>0$ sufficiently small, giving the estimate
  \begin{equation}
  \label{EqECPTogether}
    \|B_\delta v\|_{\Hcu^{s,r}} \leq \wt C\bigl(\|\wt B_1 v\|_{\Hcu^{s+1,r}} + \|\wt G'P v\|_{\Hcu^{s-m+2,r}} + \|v\|_{\Hcu^{s-1/2,r}}\bigr),
  \end{equation}
  which controls $v$ in a $\delta$-neighborhood of $\Gamma$. This finishes the proof of the estimate~\eqref{EqECP} for $N=s-\half$; the case of general $N$ follows as in~\S\ref{SsEC} by applying~\eqref{EqECP} to a microlocalized version of $v$ and iterating the estimate for the resulting microlocalized error term.

  \textit{\underline{Step 5:} regularization.} The pairings and integrations by parts in the above argument are not justified if we only assume the terms on the right in~\eqref{EqECP} to be finite. This is easily remedied by a standard regularization argument: in the first step, we replace the commutant $\check a$ in~\eqref{EqECPPhiuComm} by $\check a_\eta=\check a \varphi_\eta$, where $\varphi_\eta(\wh\rho):=(1+\eta\wh\rho^{-1})^{-1}$, $\eta>0$. In the commutator calculation~\eqref{EqECPPhiuCommSymb}, the $H_p$-derivative falling on $\varphi_\eta$ can be absorbed by the main term $b_0$ in view of
  \[
    \varphi_\eta\sfH_p\varphi_\eta = \tfrac{\eta}{\eta+\wh\rho}\varphi_\eta^2 (\wh\rho^{-1}\sfH_p\wh\rho), \quad \bigl|\tfrac{\eta}{\eta+\wh\rho}\bigr|\leq 1,
  \]
  and the vanishing of $\wh\rho^{-1}\sfH_p\wh\rho$ at $\tau=0$; that is, this term can be made arbitrarily small by localizing close enough to $\tau=0$. For $\eta>0$, one can then quantize $a_\eta=\check a_\eta^2$, which for $\eta>0$ is a symbol of $2$ orders less, and proceed as written (starting with $s=-N+\half$); this gives an estimate of the form~\eqref{EqECPPhiuEst0} with $B_0$ etc.\ replaced by operators $B_{0,\eta}$ (quantizing $\varphi_\eta\wh\rho^{s+1}b_0$). Taking $\eta\to 0$ and using a weak-*-compactness argument (see \cite{VasyMinicourse} or \cite[Proof of Proposition~7 and \S9]{MelroseEuclideanSpectralTheory}) then implies that $B_0\Phi^u v\in\Hcu^{s+1,r}$ and the estimate~\eqref{EqECPPhiuEst0} holds. One argues similarly for steps 2--4, where one combines the regularized estimates from steps 2 and 3 to obtain a regularized version of~\eqref{EqECPTogether} at which point one takes the regularization parameter to $0$ to conclude. One obtains~\eqref{EqECP} for general $N$ by the usual inductive argument.

  \textit{\underline{Step 6:} relaxing the regularity requirements.} Under the weaker assumption~\eqref{ItECPMfdsInt2}, the proof goes through with only minor changes. Indeed, the positivity of the Poisson brackets in the above positive (or negative) commutator arguments near $\tau=0$ is preserved upon adding $\cO(\tau^\alpha)$ errors, provided one localizes in a sufficiently small neighborhood of $\tau=0$. Quantizations of unweighted zeroth order symbols now lie in $\Psicu^0+\tau^\alpha\Psi_\infty^0$, with the orders of both summands shifted by the same amounts for weighted symbols of general order. If one defines the elliptic set of an operator $A=A_0+\wt A$, $A_0\in\Psicu^0$, $\wt A\in\tau^\alpha\Psi_\infty^0$ at $\tau=0$ to be equal to $\Ellcu(A_0)$, microlocal elliptic regularity holds by the usual proof (inverting the principal symbol $\sigma(A)=\sigma(A_0)+\sigma(\wt A)$), hence the positivity of symbols gives microlocal control of $v$ just as in the smooth setting.

  \textit{\underline{Step 7:} proof of the final statement (backward propagation).} The estimate~\eqref{EqECP} for $P^*$ in place of $P$, and with $B_1$ microlocalized away from $\Gamma^s$, is proved by a completely analogous argument. The key differences are:
  \begin{enumerate*} 
  \item the roles of $\phi^u$ and $\phi^s$ (and correspondingly $w^u$ and $w^s$) are switched;
  \item passing to $P^*$ switches the sign of the imaginary part of its subprincipal symbol, but we must now propagate \emph{backwards} along the $\sfH_p$ flow, as we assume a priori control on $\bar\Gamma^u\setminus\Gamma$ and want to propagate this into $\Gamma$---hence the threshold condition for step 3 above is still given by~\eqref{EqECPSubpr};
  \item the ${-}\sfH_p$-derivative of the time cutoff $\chi_T(\tau)$ now has a \emph{favorable} sign, as we are now propagating estimates in the direction of increasing $\tau$---that is, we do not need to place a priori assumptions on $\supp\chi_T'$.\qedhere
  \end{enumerate*}
\end{proof}

\section{Applications to wave equations on asymptotically Kerr spacetimes}
\label{SK}

For black hole masses $\bhm>0$ and subextremal angular momenta $\bha\in\R$, $|\bha|<\bhm$, we consider the Kerr metric in Boyer--Lindquist coordinates,
\begin{equation}
\label{EqKMetric}
\begin{gathered}
  g_{\bhm,\bha} = \frac{\Delta}{\rho^2}\bigl(d t-\bha\sin^2\theta\,d\varphi\bigr)^2 - \rho^2\Bigl(\frac{d r^2}{\Delta}+d\theta^2\Bigr) - \frac{\sin^2\theta}{\rho^2}\bigl(\bha\,d t-(r^2+\bha^2)d\varphi\bigr)^2, \\
  \rho^2=r^2+\bha^2\cos^2\theta,\qquad
  \Delta=r^2-2\bhm r+\bha^2,
\end{gathered}
\end{equation}
as a stationary Lorentzian metric of signature $(+,-,-,-)$ on
\[
  M^\circ := \R_t \times (r_+,\infty)_r \times \Sph^2,
\]
using spherical coordinates on $\Sph^2$, and where $r_+:=\bhm+\sqrt{\bhm^2-\bha^2}$ is the radius of the event horizon. Denote by $G_{\bhm,\bha}=g_{\bhm,\bha}^{-1}$ the dual metric function. Partially compactifying $M^\circ$ as
\begin{equation}
\label{EqKCpt}
  M := \Bigl(M^\circ \sqcup \bigl([0,\infty)_\tau\times(r_+,\infty)\times\Sph^2\bigr) \Bigr) / \sim,
  \quad
  (t,r,\omega) \sim (\tau=t^{-1},r,\omega)\ \text{for}\ t>0,
\end{equation}
the discussion in~\S\ref{SssECC} implies that $g_{\bhm,\bha}\in\CI(M;S^2\,\Tcu^*M)$ is a smooth cusp metric. Denote the characteristic set by
\[
  \Sigma := G_{\bhm,\bha}^{-1}(0)\subset\Tcu^*M\setminus o;
\]
it has two components, $\Sigma=\Sigma^+\sqcup\Sigma^-$, where $\Sigma^\pm=\{\varpi\in\Sigma\colon\pm G_{\bhm,\bha}(\varpi,d t)>0\}$. Identifying $\Sigma$ with its closure inside $\ol{\Tcu^*}M\setminus o$, let
\begin{equation}
\label{EqKCharInfty}
  \cX_\vee := \Sigma\cap \Tcu^*_{\pa M}M, \quad
  \cX := \pa\cX_\vee = \Sigma\cap\Scu^*_{\pa M}M.
\end{equation}
denote the characteristic set at future infinity (`$\vee$' indicating its conic version, rather than its boundary at fiber infinity $\cX$).

The Hamilton flow $H_{G_{\bhm,\bha}}$ restricted to $\Sigma$ has a trapped set in $r>r_+$ at which the flow is $r$-normally hyperbolic for every $r$ (the latter $r$ being an integer, not the radius function!). This was first proved by Wunsch and Zworski~\cite{WunschZworskiNormHypResolvent} for $|\bha|\ll\bhm$, and extended to the full subextremal range by Dyatlov~\cite{DyatlovWaveAsymptotics}, see also \cite{DyatlovZworskiTrapping}. (The Kerr--de~Sitter case was discussed by Vasy \cite{VasyMicroKerrdS}.) This is usually phrased as a condition on the $t$-independent spacetime trapped set (though all geodesics in the spacetime trapped set escape to $t=\infty$, hence trapping really only occurs there). Concretely, define
\begin{equation}
\label{EqKGammaInfty}
  \Gamma_\vee := \bigl\{ \varpi \in \cX_\vee \colon H_{G_{\bhm,\bha}}r=(H_{G_{\bhm,\bha}})^2 r=0\ \text{at}\ \varpi \bigr\},\quad
  \Gamma := \pa\Gamma_\vee \subset \cX.
\end{equation}
As shown in the references, $\Gamma_{(\vee)}$ is a smooth codimension $2$ submanifold of $\cX_{(\vee)}$, and $H_{G_{\bhm,\bha}}$ is tangent to $\Gamma_\vee$. Write
\[
  \Gamma^\pm:=\Gamma\cap\Sigma^\pm.
\]

We also recall that the fiber-linear function $\sigma=\pa_t(\cdot)\in\CI(\Tcu^*M)$ (which was already defined in~\eqref{EqECCTimeMomentum}) is non-zero on $\Gamma_\vee$; indeed, $\pm\sigma>0$ on $\Gamma_\vee^\pm=\Gamma_\vee\cap\Sigma^\pm$. Thus, $|\sigma|^{-1}$ is a smooth function, homogeneous of degree $-1$, near $\Gamma_\vee$, and
\begin{equation}
\label{EqKHamFbdf}
  H_{G_{\bhm,\bha}}\wh\rho=0,\quad \wh\rho:=|\sigma|^{-1},
\end{equation}
since the Kerr metric is stationary. Furthermore, $H_{G_{\bhm,\bha}}t\neq 0$ on $\Gamma_\vee$, hence the rescaled vector field
\[
  \sfH := \tfrac{1}{H_{G_{\bhm,\bha}}t}H_{G_{\bhm,\bha}},\quad \sfH t=1,
\]
induces a vector field tangent to $\Gamma$; and $\sfH$, or rather its projection to a cusp vector field on $M$ over $\pa M$, is \emph{future} null. The following is a reformulation of the results of~\cite[\S3.2]{DyatlovWaveAsymptotics}:
\begin{prop}
\label{PropKNhyp}
  The $\sfH$-flow within $\cX$ is eventually absolutely $r$-normally hyperbolic at $\Gamma$ for every $r$. The stable and unstable manifolds $\bar\Gamma^{u/s}$ at $\Gamma$ are smooth and orientable, of codimension $1$ within $\cX$, intersect transversally at $\Gamma$, and admit defining functions $\bar\phi^{u/s}\in\CI(\cX)$ which satisfy the non-degeneracy conditions~\eqref{EqECPDefBdyHyp}--\eqref{EqECPDefBdyNumin}.
\end{prop}

Let $\Xi\in\CI(M)$, $\Xi>0$, denote a smooth conformal factor. We denote by
\[
  \nu_{\rm min} > 0
\]
the quantity defined in~\eqref{EqECPDefBdyNumin} with $\sfH_{p_0}:=\pm\Xi\wh\rho H_{G_{\bhm,\bha}}$ near $\Gamma^\pm$;\footnote{In the Schwarzschild case $\bha=0$ and for $\Xi=\rho^2$, we have $\nu_{\rm min}=6\sqrt{3}\bhm$; this can be read off from \cite[Proposition~3.8]{DyatlovWaveAsymptotics} and $\Xi=9\bhm^2$, $\sigma^{-1}H_{G_{\bhm,0}}t=6$ at $\Gamma$.} $\sfH_{p_0}=\Xi\sigma^{-1}(H_{G_{\bhm,\bha}}t)\cdot\sfH$ is a smooth positive multiple of $\sfH$.

This verifies assumptions~\eqref{ItECPNondeg}--\eqref{ItECPDefBdy} of Theorem~\ref{ThmECP}, and, by stationarity, also assumption~\eqref{ItECPMfdsInt} for the Kerr spacetime $(M,g_{\bhm,\bha})$, thereby providing us with microlocal estimates at the trapped set for bundle-valued wave equations on exact Kerr. The main advance in the present paper is that we can consider significant perturbations of the Kerr spacetime:

\begin{definition}
\label{DefKPert}
  Let $\alpha>0$. A metric $g\in\CI(M^\circ;S^2 T^*M^\circ)$ is an \emph{asymptotically (subextremal) Kerr metric} if there exist parameters $\bhm$ and $\bha\in(-\bhm,\bhm)$ of a subextremal Kerr black hole and a symmetric 2-tensor $\wt g$ such that
  \[
    g = g_{\bhm,\bha} + \wt g,\qquad
    \wt g\in \cA_\cuop^\alpha(M;S^2\,\Tcu^*M).
  \]
\end{definition}

In particular, smooth metrics $g\in\CI(M;S^2\,\Tcu^*M)$ with $g-g_{\bhm,\bha}\in\tau\CI$ are examples of asymptotically Kerr metrics, though the smoothness requirement is unnatural from the point of view of applications.

Combining the results of \S\ref{SsMF} (describing the null-geodesic flow near the trapped set of the asymptotically Kerr spacetime) with those of~\S\ref{SsEC} (microlocal estimates), we obtain:
\begin{thm}
\label{ThmK}
  Let $g=g_{\bhm,\bha}+\wt g$ denote an asymptotically subextremal Kerr metric, and let $G=g^{-1}$. Let $\Sigma:=G^{-1}(0)\subset\Tcu^*M\setminus o$ denote the characteristic set, and define $\cX$ as in~\eqref{EqKCharInfty}. Define the trapped set $\Gamma$ as in~\eqref{EqKCharInfty} and \eqref{EqKGammaInfty}; let further $\bar\Gamma^{u/s}\subset\cX$ denote the unstable/stable manifold of $\Gamma$ inside of $\cX$, and denote by $\Gamma^s_0=[0,\infty)_\tau\times\bar\Gamma^s\subset\Scu^*M$ the stationary extension of $\bar\Gamma^s$. Denote the speed $1$ rescaling of the Hamilton vector field by
  \[
    \sfH := \tfrac{1}{H_G t}H_G \in (\CI+\cA_\cuop^\alpha)\Vcu(\ol{\Tcu^*}M),
  \]
  defined in a neighborhood of $\Gamma$. Then:
  \begin{enumerate}
  \item\label{ItKMfds} (Existence and smoothness of the stable manifold.) There exists a subset $\Gamma^s\subset\Sigma$, with $\Gamma^s\cap\cX=\bar\Gamma^s$ and with $\Gamma^s\cap S^*M^\circ$ a smooth manifold, such that $\Gamma^s$ approaches $\Gamma^s_0$ in an $\cA_\cuop^\alpha$ sense (see the statement of part~\eqref{ItMMapU} of Theorem~\usref{ThmMMap}), and such that $\sfH$ is tangent to $\Gamma^s$.
  \item\label{ItKEst} (Microlocal estimates.) Let $E\to M$ be a vector bundle, let $P_0\in\Psicu^2(M;E)$ denote a classical cusp ps.d.o.\ with $\sigma^2_\cuop(P_0)=\Xi G_{\bhm,\bha}$, and let $\wt P\in\tau^\alpha\Psi_\infty^2(M;E)$ denote an operator with $\sigma^2(\wt P)=\Xi(G-G_{\bhm,\bha})$; let
  \[
    P = P_0 + \wt P.
  \]
  Suppose that
  \begin{equation}
  \label{EqKSubpr}
    \sup_\Gamma \sfp_1 < \half\nu_{\rm min},\quad
    \sfp_1 := \pm\wh\rho\sigma_\cuop^1\Bigl(\frac{1}{2 i}(P_0-P_0^*)\Bigr)\ \text{on}\ \Gamma^\pm.
  \end{equation}
  Then the conclusions of Theorem~\usref{ThmECP} hold for $P$.
  \end{enumerate}
  If $\wt g\in\cA^\alpha_\bop(M;S^2\,\Tcu^*M)$, then part~\eqref{ItKMfds} holds with $\cA_\bop$ in place of $\cA_\cuop$.
\end{thm}

\begin{rmk}
\label{RmkKPertModel}
  Due to the structural stability of $r$-normally hyperbolic trapping \cite[Theorem~(4.1)]{HirschPughShubInvariantManifolds}, and due to the fact that the proof of Theorem~\ref{ThmECP} for any fixed level of Sobolev regularity only requires a finite degree of smoothness of the coefficients of $P$, one can consider a much more general situation: suppose $B$ is an open subset of a Banach space of parameters (which in Theorem~\ref{ThmK} was taken to be $B=\{(\bhm,\bha)\colon |\bha|<\bhm\}$) smoothly parameterizing $\CI$ stationary metrics by assigning $B\in b\mapsto g_b$, and let $b_0\in B$ be such that $g_{b_0}$ is a subextremal Kerr metric. Then, for regularity and weights confined to compact subsets of $\R$, the estimates of Theorem~\ref{ThmK} hold at the trapped set of $g_b$ when $b$ is close to $b_0$.
\end{rmk}

\begin{rmk}
\label{RmkKSubpr}
  The main calculation required to apply the microlocal estimates is the verification of the subprincipal symbol condition~\eqref{EqKSubpr}. This was verified for the linearization of the gauge-fixed Einstein equation in~\cite[\S9.1]{HintzVasyKdSStability} at the Schwarzschild--de~Sitter metric, but the calculations there work directly for Schwarzschild metrics as well. By continuity, the condition~\eqref{EqKSubpr} is verified also for slowly rotating Kerr black holes. More generally, it holds for wave equations on tensors on slowly rotating Kerr, as follows in the same manner from~\cite{HintzPsdoInner}, and can in fact be explicitly verified in the \emph{full} subextremal range by using the relationship, explained in~\cite{HintzPsdoInner}, between~\eqref{EqKSubpr} and parallel transport along trapped null-geodesics, with the latter being described by Marck~\cite{MarckParallelNull}; the details will be presented elsewhere.
\end{rmk}

The results of~\S\ref{SsMF} take place on the stationary manifold
\begin{equation}
\label{EqKcM}
  \cM := \R_t\times\cX.
\end{equation}
One would like to take $\cM$ to be the characteristic set $\Sigma\cap S^*M^\circ$; however, since $g$ is not stationary when $\wt g\neq 0$, we need to relate $\Sigma\cap M^\circ$ and~\eqref{EqKcM}. To this effect, we have the following general result:

\begin{lemma}
\label{LemmaKChar}
  With $M=[0,\infty)_\tau\times X$, $X$ compact, let $E\to M$ denote a vector bundle with zero section $o$, and let $S E\to M$ be the fiber bundle with fibers $S E_p:=(E_p\setminus o_p)/\R_+$. Let $p_0\in\CI(S E)$, and suppose that $d p_0\neq 0$ on $\fX:=p_0^{-1}(0)\cap S E|_{\pa M}$, which we assume to be non-empty. Put $\fM_0:=[0,\infty)_\tau\times\fX\subset S E$, which is thus smooth and of codimension $1$. Let moreover $\alpha>0$ and $\wt p\in\cA_\cuop^\alpha(S E)$, and let
  \[
    p = p_0 + \wt p,\qquad \fM:=p^{-1}(0)\subset S E.
  \]
  Let $\varpi\in\fX$, and let $\cU\times(-1,1)$ be a tubular neighborhood of an open neighborhood $\cU\subset\fX$ of $\varpi$ within $S E|_{\pa M}$; extend this to a tubular neighborhood $([0,\infty)\times\cU)\times(-1,1)$ of $[0,\infty)_\tau\times\cU\subset\fM_0$ inside $S E$.\footnote{If the normal bundle of $\fX$ is orientable, one can take $\cU=\fX$ simply.} Then there exists $\tau_0>0$ such that in this neighborhood, $\fM\cap {\{\tau<\tau_0\}}$ is the graph over $\fM_0\cap {\{\tau<\tau_0\}}$ of a function $f\in\cA_\cuop^\alpha(\fM_0)$. More generally, if $k\in\N$ is such that $\Diffcu^k(S E)\wt p\subset\tau^\alpha L^\infty$, then $\Diffcu^k(\fM_0)f\subset\tau^\alpha L^\infty$ (with continuous dependence).
\end{lemma}

This also holds true if one replaces $\cA_\cuop$ by $\cA_\bop$ throughout the statement of the lemma.

\begin{proof}[Proof of Lemma~\usref{LemmaKChar}]
  We may assume that in local coordinates $(x^0,x')\in\R\times\R^{n-2+\rank E}$ on $S E|_{\pa M}$, we have $p_0=x^0$, so $\fX=\{x^0=0\}$ and $\fM_0=\{x^0=0\}$ in the product coordinate system; the defining equation for $\fM$ reads $p(\tau,x^0,x')=x^0+\wt p(\tau,x^0,x')=0$. Write $x^0=\tau^\alpha y$ and $\wt q=\tau^{-\alpha}\wt p$, which satisfies $\Diffcu^k\cdot\wt q\subset L^\infty$. We then need to solve $y+\wt q(\tau,\tau^\alpha y,x')=0$, which can be solved using the contraction mapping principle on the $y$-ball of radius $\|\wt q\|_{L^\infty}+1$. Higher regularity follows by differentiating this equation along $\tau^2\pa_\tau$ and $\pa_{x'}$.
\end{proof}

\begin{proof}[Proof of Theorem~\usref{ThmK}]
  By the previous lemma, $\Sigma\cap S^*M^\circ$ is an $\cA_\cuop^\alpha$ graph over $\cM\subset S^*M^\circ$ near $\tau=0$. Let $\Phi\colon\cM\to\Sigma\cap S^*M^\circ$ be the diffeomorphism $\cM\ni\varpi\mapsto(\varpi,f(\varpi))$ (using the notation and the collar neighborhood of Lemma~\ref{LemmaKChar}). Then
  \[
    \Phi^*\sfH =: V = V_0 + \wt V,
  \]
  where $V_0=\tfrac{1}{H_{G_{\bhm,\bha}}t}H_{G_{\bhm,\bha}}$ is the stationary model, and $\wt V\in\rho\CI_b(\cM;T\cX)$, $\rho(t)=t^{-\alpha}$, is the perturbation. The vector field $-V$ thus satisfies the assumptions of Theorem~\ref{ThmMFlow}, proving part~\eqref{ItKMfds}.
  
  Theorem~\ref{ThmMFlow} also produces a (non-unique) unstable manifold $\Gamma^u\subset\Sigma$, with $\Gamma^u\cap\cX=\bar\Gamma^u$, which is an $\cA_\cuop^\alpha$ graph over $\Gamma^u_0=[0,\infty)_\tau\times\bar\Gamma^u$ near $\tau=0$, and to which $\sfH$ is tangent; thus, we get defining functions $\phi^{u/s}\in(\CI+\cA_\cuop^\alpha)(\Scu^*M)$ of $\Gamma^{u/s}$ within $\Sigma$. Thus, assumption~\eqref{ItECPMfdsInt2} of Theorem~\ref{ThmECP} is verified as well. In $\Sigma^+$, we can thus apply this theorem to the operator $P$, while in $\Sigma^-$, we apply it to $-P$ since the Hamilton vector field of $-P$ is future null in $\Sigma^-$. This proves part~\eqref{ItKEst}.
\end{proof}

\begin{rmk}
\label{RmkKdS}
  For those parameters of Kerr--de~Sitter black holes for which normal hyperbolicity was verified in~\cite{DyatlovWaveAsymptotics,VasyMicroKerrdS}, we can similarly analyze asymptotically Kerr--de~Sitter metrics $g$: in this case, natural metric perturbations $\wt g$ of such a Kerr--de~Sitter metric lie in the class $e^{-\alpha t}\cA_\cuop$ on the compactification~\eqref{EqKCpt} (or equivalently $\tau_e^\alpha\cA_\bop$ on the exponential compactification~\eqref{EqECPNhyptrExp}). Theorem~\ref{ThmK} then provides a description of the spacetime stable/unstable manifolds and gives microlocal trapping estimates even on function spaces with (sufficiently mild) exponential decay in $t$; see Remark~\ref{RmkECPNhyptr}.
\end{rmk}

\bibliographystyle{alpha}

\begin{thebibliography}{MSBV14}

\bibitem[AB15a]{AnderssonBlueHiddenKerr}
Lars Andersson and Pieter Blue.
\newblock {H}idden symmetries and decay for the wave equation on the {K}err
  spacetime.
\newblock {\em Annals of Mathematics}, 182:787--853, 2015.

\bibitem[AB15b]{AnderssonBlueMaxwellKerr}
Lars Andersson and Pieter Blue.
\newblock {U}niform energy bound and asymptotics for the {M}axwell field on a
  slowly rotating {K}err black hole exterior.
\newblock {\em Journal of Hyperbolic Differential Equations}, 12(04):689--743,
  2015.

\bibitem[Are12]{AretakisExtremalKerr}
Stefanos Aretakis.
\newblock {D}ecay of axisymmetric solutions of the wave equation on extreme
  {K}err backgrounds.
\newblock {\em Journal of Functional Analysis}, 263(9):2770--2831, 2012.

\bibitem[BBR10]{BonyBurqRamondTrapping}
Jean-Fran\c{c}ois Bony, Nicolas Burq, and Thierry Ramond.
\newblock Minoration de la r\'esolvante dans le cas captif.
\newblock {\em Comptes Rendus Mathematique}, 348(23--24):1279--1282, 2010.

\bibitem[BH08]{BonyHaefnerDecay}
Jean-Fran{\c{c}}ois Bony and Dietrich H{\"a}fner.
\newblock Decay and non-decay of the local energy for the wave equation on the
  de {S}itter--{S}chwarzschild metric.
\newblock {\em Communications in Mathematical Physics}, 282(3):697--719, 2008.

\bibitem[BS06]{BlueSterbenzSemilinear}
Pieter Blue and Jacob Sterbenz.
\newblock {U}niform decay of local energy and the semi-linear wave equation on
  {S}chwarzschild space.
\newblock {\em Communications in mathematical physics}, 268(2):481--504, 2006.

\bibitem[Chr07]{ChristiansonNonconc}
Hans Christianson.
\newblock Semiclassical non-concentration near hyperbolic orbits.
\newblock {\em Journal of Functional Analysis}, 246(2):145--195, 2007.

\bibitem[DHR16]{DafermosHolzegelRodnianskiSchwarzschildStability}
Mihalis Dafermos, Gustav Holzegel, and Igor Rodnianski.
\newblock The linear stability of the {S}chwarzschild solution to gravitational
  perturbations.
\newblock {\em Preprint, arXiv:1601.06467}, 2016.

\bibitem[DHR17]{DafermosHolzegelRodnianskiTeukolsky}
Mihalis Dafermos, Gustav Holzegel, and Igor Rodnianski.
\newblock {B}oundedness and decay for the {T}eukolsky equation on {K}err
  spacetimes {I}: the case $|a|\ll{M}$.
\newblock {\em Preprint, arXiv:1711.07944}, 2017.

\bibitem[DRSR16]{DafermosRodnianskiShlapentokhRothmanDecay}
Mihalis Dafermos, Igor Rodnianski, and Yakov Shlapentokh-Rothman.
\newblock Decay for solutions of the wave equation on {K}err exterior
  spacetimes {III}: {T}he full subextremal case {$\vert a\vert <M$}.
\newblock {\em Ann. of Math. (2)}, 183(3):787--913, 2016.

\bibitem[DSS11]{DonningerSchlagSofferPrice}
Roland Donninger, Wilhelm Schlag, and Avy Soffer.
\newblock A proof of {P}rice's law on {S}chwarzschild black hole manifolds for
  all angular momenta.
\newblock {\em Advances in Mathematics}, 226(1):484--540, 2011.

\bibitem[DSS12]{DonningerSchlagSofferSchwarzschild}
Roland Donninger, Wilhelm Schlag, and Avy Soffer.
\newblock On pointwise decay of linear waves on a {S}chwarzschild black hole
  background.
\newblock {\em Communications in Mathematical Physics}, 309(1):51--86, 2012.

\bibitem[Dya11a]{DyatlovQNMExtended}
Semyon Dyatlov.
\newblock Exponential energy decay for {K}err--de {S}itter black holes beyond
  event horizons.
\newblock {\em Mathematical Research Letters}, 18(5):1023--1035, 2011.

\bibitem[Dya11b]{DyatlovQNM}
Semyon Dyatlov.
\newblock Quasi-normal modes and exponential energy decay for the {K}err--de
  {S}itter black hole.
\newblock {\em Comm. Math. Phys.}, 306(1):119--163, 2011.

\bibitem[Dya12]{DyatlovAsymptoticDistribution}
Semyon Dyatlov.
\newblock Asymptotic distribution of quasi-normal modes for {K}err--de {S}itter
  black holes.
\newblock In {\em Annales Henri Poincar{\'e}}, volume~13, pages 1101--1166.
  Springer, 2012.

\bibitem[Dya15a]{DyatlovWaveAsymptotics}
Semyon Dyatlov.
\newblock Asymptotics of {L}inear {W}aves and {R}esonances with {A}pplications
  to {B}lack {H}oles.
\newblock {\em Comm. Math. Phys.}, 335(3):1445--1485, 2015.

\bibitem[Dya15b]{DyatlovResonanceProjectors}
Semyon Dyatlov.
\newblock Resonance projectors and asymptotics for {$r$}-normally hyperbolic
  trapped sets.
\newblock {\em J. Amer. Math. Soc.}, 28(2):311--381, 2015.

\bibitem[Dya16]{DyatlovSpectralGaps}
Semyon Dyatlov.
\newblock Spectral gaps for normally hyperbolic trapping.
\newblock {\em Ann. Inst. Fourier (Grenoble)}, 66(1):55--82, 2016.

\bibitem[DZ13]{DyatlovZworskiTrapping}
Semyon Dyatlov and Maciej Zworski.
\newblock Trapping of waves and null geodesics for rotating black holes.
\newblock {\em Physical Review D}, 88(8):084037, 2013.

\bibitem[DZ18]{DyatlovZworskiBook}
Semyon Dyatlov and Maciej Zworski.
\newblock Mathematical theory of scattering resonances.
\newblock {\em Book in progress, http://math.mit.edu/dyatlov/res/}, 2018.

\bibitem[Fen71]{FenichelInvariant}
Neil Fenichel.
\newblock {P}ersistence and {S}moothness of {I}nvariant {M}anifolds for
  {F}lows.
\newblock {\em Indiana University Mathematics Journal}, 21(3):193--226, 1971.

\bibitem[FKSY06]{FinsterKamranSmollerYauKerr}
Felix Finster, Niky Kamran, Joel Smoller, and Shing-Tung Yau.
\newblock {D}ecay of {S}olutions of the {W}ave {E}quation in the {K}err
  {G}eometry.
\newblock {\em Communications in Mathematical Physics}, 264(2):465--503, 2006.

\bibitem[G{\'e}r91]{GerardDefect}
Patrick G{\'e}rard.
\newblock Microlocal defect measures.
\newblock {\em Communications in Partial Differential Equations},
  16(11):1761--1794, 1991.

\bibitem[GS87]{GerardSjostrandHyperbolic}
Christian G{\'e}rard and Johannes Sj{\"o}strand.
\newblock Semiclassical resonances generated by a closed trajectory of
  hyperbolic type.
\newblock {\em Communications in Mathematical Physics}, 108(3):391--421, 1987.

\bibitem[Hin16]{HintzQuasilinearDS}
Peter Hintz.
\newblock Global analysis of quasilinear wave equations on asymptotically de
  {S}itter spaces.
\newblock {\em Annales de l'Institut Fourier}, 66(4):1285--1408, 2016.

\bibitem[Hin17]{HintzPsdoInner}
Peter Hintz.
\newblock Resonance expansions for tensor-valued waves on asymptotically
  {K}err--de {S}itter spaces.
\newblock {\em J. Spectr. Theory}, 7:519--557, 2017.

\bibitem[Hin18]{HintzKNdSStability}
Peter Hintz.
\newblock {N}on-linear {S}tability of the {K}err--{N}ewman--de {S}itter
  {F}amily of {C}harged {B}lack {H}oles.
\newblock {\em Annals of PDE}, 4(1):11, Apr 2018.

\bibitem[HPS77]{HirschPughShubInvariantManifolds}
Morris~W. Hirsch, Charles~C. Pugh, and Michael Shub.
\newblock {\em Invariant manifolds}, volume 583.
\newblock Springer-Verlag, Berlin-New York, 1977.

\bibitem[HV14]{HintzVasyNormHyp}
Peter Hintz and Andr{\'a}s Vasy.
\newblock Non-trapping estimates near normally hyperbolic trapping.
\newblock {\em Math. Res. Lett.}, 21(6):1277--1304, 2014.

\bibitem[HV15]{HintzVasySemilinear}
Peter Hintz and Andr{\'a}s Vasy.
\newblock Semilinear wave equations on asymptotically de {S}itter, {K}err--de
  {S}itter and {M}inkowski spacetimes.
\newblock {\em Anal. PDE}, 8(8):1807--1890, 2015.

\bibitem[HV16]{HintzVasyQuasilinearKdS}
Peter Hintz and Andr{\'a}s Vasy.
\newblock {G}lobal {A}nalysis of {Q}uasilinear {W}ave {E}quations on
  {A}symptotically {K}err--de {S}itter {S}paces.
\newblock {\em International Mathematics Research Notices},
  2016(17):5355--5426, 2016.

\bibitem[HV18a]{HintzVasyKdsFormResonances}
Peter Hintz and Andr{\'a}s Vasy.
\newblock {A}symptotics for the wave equation on differential forms on
  {K}err--de {S}itter space.
\newblock {\em Journal of Differential Geometry}, 110(2):221--279, 2018.

\bibitem[HV18b]{HintzVasyKdSStability}
Peter Hintz and Andr{\'a}s Vasy.
\newblock {T}he global non-linear stability of the {K}err--de {S}itter family
  of black holes.
\newblock {\em Acta mathematica}, 220:1--206, 2018.

\bibitem[KW87]{KayWaldSchwarzschild}
Bernard~S. Kay and Robert~M. Wald.
\newblock Linear stability of {S}chwarzschild under perturbations which are
  non-vanishing on the bifurcation 2-sphere.
\newblock {\em Classical and Quantum Gravity}, 4(4):893, 1987.

\bibitem[LP93]{LionsPaulDefect}
Pierre-Louis Lions and Thierry Paul.
\newblock {S}ur les mesures de {W}igner.
\newblock {\em Revista Matem{\'a}tica Iberoamericana}, 9(3):553--618, 1993.

\bibitem[Luk13]{LukKerrNonlinear}
Jonathan Luk.
\newblock The null condition and global existence for nonlinear wave equations
  on slowly rotating {K}err spacetimes.
\newblock {\em Journal of the European Mathematical Society}, 15(5):1629--1700,
  2013.

\bibitem[Mar83]{MarckParallelNull}
Jean-Alain Marck.
\newblock {P}arallel-tetrad on null geodesics in {K}err--{N}ewman space-time.
\newblock {\em Physics Letters A}, 97(4):140 -- 142, 1983.

\bibitem[Mel94]{MelroseEuclideanSpectralTheory}
Richard~B. Melrose.
\newblock Spectral and scattering theory for the {L}aplacian on asymptotically
  {E}uclidian spaces.
\newblock In {\em Spectral and scattering theory ({S}anda, 1992)}, volume 161
  of {\em Lecture Notes in Pure and Appl. Math.}, pages 85--130. Dekker, New
  York, 1994.

\bibitem[MM99]{MazzeoMelroseFibred}
Rafe~R. Mazzeo and Richard~B. Melrose.
\newblock Pseudodifferential operators on manifolds with fibred boundaries.
\newblock {\em Asian J. Math.}, 2(4):833--866, 1999.

\bibitem[MSBV14]{MelroseSaBarretoVasySdS}
Richard~B. Melrose, Ant{\^o}nio S{\'a}~Barreto, and Andr{\'a}s Vasy.
\newblock Asymptotics of solutions of the wave equation on de
  {S}itter--{S}chwarzschild space.
\newblock {\em Communications in Partial Differential Equations},
  39(3):512--529, 2014.

\bibitem[NZ13]{NonnenmacherZworskiDecay}
St{\'e}phane Nonnenmacher and Maciej Zworski.
\newblock Decay of correlations for normally hyperbolic trapping.
\newblock {\em Inventiones mathematicae}, 200(2):345--438, 2013.

\bibitem[Ral69]{RalstonLocalized}
James~V. Ralston.
\newblock Solutions of the wave equation with localized energy.
\newblock {\em Communications on Pure and Applied Mathematics}, 22(6):807--823,
  1969.

\bibitem[SBZ97]{SaBarretoZworskiResonances}
Ant{\^o}nio S{\'a}~Barreto and Maciej Zworski.
\newblock Distribution of resonances for spherical black holes.
\newblock {\em Mathematical Research Letters}, 4:103--122, 1997.

\bibitem[Tar90]{TartarHMeasures}
Luc Tartar.
\newblock {$H$}-measures, a new approach for studying homogenisation,
  oscillations and concentration effects in partial differential equations.
\newblock {\em Proc. Roy. Soc. Edinburgh Sect. A}, 115(3-4):193--230, 1990.

\bibitem[Tat13]{TataruDecayAsympFlat}
Daniel Tataru.
\newblock Local decay of waves on asymptotically flat stationary space-times.
\newblock {\em American Journal of Mathematics}, 135(2):361--401, 2013.

\bibitem[Toh12]{TohaneanuKerrStrichartz}
Mihai Tohaneanu.
\newblock Strichartz estimates on {K}err black hole backgrounds.
\newblock {\em Trans. Amer. Math. Soc.}, 364(2):689--702, 2012.

\bibitem[TT11]{TataruTohaneanuKerrLocalEnergy}
Daniel Tataru and Mihai Tohaneanu.
\newblock A local energy estimate on {K}err black hole backgrounds.
\newblock {\em Int. Math. Res. Not.}, (2):248--292, 2011.

\bibitem[Vas]{VasyMinicourse}
Andr{\'a}s Vasy.
\newblock A minicourse on microlocal analysis for wave propagation.
\newblock In {\em Asymptotic Analysis in General Relativity}, London
  Mathematical Society Lecture Note Series. Cambridge University Press, to
  appear.

\bibitem[Vas13]{VasyMicroKerrdS}
Andr{\'a}s Vasy.
\newblock Microlocal analysis of asymptotically hyperbolic and {K}err--de
  {S}itter spaces (with an appendix by {S}emyon {D}yatlov).
\newblock {\em Invent. Math.}, 194(2):381--513, 2013.

\bibitem[Wal79]{WaldSchwarzschild}
Robert~M. Wald.
\newblock Note on the stability of the {S}chwarzschild metric.
\newblock {\em Journal of Mathematical Physics}, 20(6):1056--1058, 1979.

\bibitem[Wun12]{WunschMild}
Jared Wunsch.
\newblock Resolvent estimates with mild trapping.
\newblock {\em Journ{\'e}es {\'e}quations aux d{\'e}riv{\'e}es partielles},
  2012:1--15, 2012.

\bibitem[WZ11]{WunschZworskiNormHypResolvent}
Jared Wunsch and Maciej Zworski.
\newblock Resolvent estimates for normally hyperbolic trapped sets.
\newblock In {\em Annales Henri Poincar{\'e}}, volume~12, pages 1349--1385.
  Springer, 2011.

\bibitem[Zwo17]{ZworskiResonanceReview}
Maciej Zworski.
\newblock Mathematical study of scattering resonances.
\newblock {\em Bulletin of Mathematical Sciences}, 7(1):1--85, 2017.

\end{thebibliography}

\end{document}